\documentclass[10pt]{article}

\usepackage{amsthm,amsfonts,amsmath,amscd,amssymb,times,epsfig,verbatim}
\usepackage[all]{xy}
\usepackage{enumerate,array}

\newtheorem{thm}{Theorem}[section]
\newtheorem{prop}[thm]{Proposition}
\newtheorem{lem}[thm]{Lemma}
\newtheorem{cor}[thm]{Corollary}

\theoremstyle{remark}
\newtheorem{rem}[thm]{Remark}

\theoremstyle{definition}
\newtheorem{defn}[thm]{Definition}
\newtheorem{ex}[thm]{Example}

\newcommand{\C}{\mathbb{ C}}

\newcommand{\lra}{\longrightarrow}

\newcommand{\LLL}{\mathcal{ L}}
\newcommand{\MMM}{\mathcal{ M}}

\newcommand{\Q}{\mathbb{ Q}}
\newcommand{\R}{\mathbb{ R}}
\newcommand{\HH}{\mathbb{ H}}
\newcommand{\Ker}{\operatorname{Ker}}
\newcommand{\rk}{\operatorname{rk}}
\newcommand{\ind}{\operatorname{ind}}

\newcommand{\T}{\operatorname{T}}

\newcommand{\Z}{\mathbb{ Z}}

\newcommand{\sg}{\mathrm{sign}}
\newcommand{\rw}{\mathrm{w}}
\newcommand{\ru}{\mathrm{u}}
\newcommand{\CC}{\mathcal{ C}}
\newcommand{\Ha}{\mathcal{ H}}

\DeclareMathOperator{\dummygg}{\mathfrak{g}}
\renewcommand{\gg}{\dummygg}

\DeclareMathOperator{\hh}{\mathfrak{h}}
\DeclareMathOperator{\TT}{\mathfrak{t}}

\newcommand{\dr}[3]{\ensuremath{#1\stackrel{#2}
{\longrightarrow}#3}}

\title{Toric genera of homogeneous spaces and their fibrations}
\author{Victor M.~Buchstaber and Svjetlana Terzi\'c}

\begin{document}
\maketitle

\begin{abstract}
The  aim of this paper is to study further the universal toric genus 
of compact homogeneous spaces and their
homogeneous fibrations. We consider the homogeneous spaces with positive Euler characteristic. It is well known that such spaces carry many stable complex structures equivariant under the canonical action of the maximal torus $T^k$. As the torus action in this case has only isolated fixed points it is possible to effectively apply localization formula for the universal toric genus.  Using this we prove that the famous topological results  related to rigidity and multiplicativity of a Hirzebruch genus can be obtained on homogeneous spaces just using representation theory. In that context for homogeneous $SU$-spaces we prove the well known result about rigidity of the Krichever genus. We also prove that for a large class of stable complex homogeneous spaces any $T^k$-equivariant Hirzebruch genus given by an odd power series vanishes. Related to the problem of multiplicativity we provide construction of stable complex $T^k$-fibrations for which the universal toric genus is twistedly multiplicative. We prove that it is always twistedly multiplicative for almost complex homogeneous fibrations and describe those fibrations for which it is multiplicative. As a consequence for such fibrations  the strong relations between rigidity and multiplicativity for an equivariant  Hirzebruch genus is established.  The universal toric genus of the fibrations for which the base does not admit any stable complex structure is also considered. The main examples here for which we compute the universal toric genus are the homogeneous fibrations  over  quaternionic projective spaces.  
\end{abstract}

\tableofcontents

\section{Introduction}
The toric genus is an invariant of a stable complex manifold with an action of the torus $T^k$, $k\geq 1$. It plays an important role in studying  such manifolds, especially in those cases when the torus action has isolated fixed points~\cite{BPR}.  The wide class of these manifolds has appeared recently in toric geometry and toric topology. The other class of well known classical  manifolds with such action are compact homogeneous spaces $G/H$ of positive Euler characteristic with the action of the maximal torus for $G$. In this case  due to representation and Lie group theories it is possible  to effectively  describe invariant almost complex structures as well as the corresponding weights at the fixed points for the canonical action of a maximal torus~\cite{Buch_Terz}. It allows us to, using toric genera,  solve an outstanding problem of description of complex cobordism classes of homogeneous spaces~\cite{Buch_Terz}. In particular, we  effectively  computed complex cobordism  classes of  homogeneous spaces such as flag and Grassmann manifolds related to the standard invariant complex structure. 

In this paper  we further describe complex cobordism classes of  flag manifolds and generalized Grassmann manifolds related to an arbitrary invariant almost complex structure. As a consequence we obtain the results on top Chern numbers $s_{n}$ for these spaces.  The problem of the numbers $s_{n}$ is well known as the problem of multiplicative generators in the complex cobordism ring ~\cite{NM}.

Within the questions of  toric genus of the equivariant stable complex fibration that we consider the two cases can be recognized.  
The first one is when the total space,  base and  the fiber are equivariant stable complex manifolds whose  structures are compatible.  In the theory of Hirzebruch genera  the problem of multiplicativity of genus for fibrations is well known. We consider this problem for toric genera  and prove the twisted product formula in this case.

The second case is related to the notion, introduced in~\cite{BPR}, of toric genera for the family of the stable complex manifolds with torus action. This family consists of the bundles whose total space admits equivariant stable complex structure which also induces such a structure on the fiber and projection to the base is equivariant map. The analogous localization theorems are obtained for this generalization of toric genera.  In our paper we study generalization of toric genera in the case of corresponding bundles of homogeneous spaces. We pay  special attention to the bundles for which the base does not admit any stable complex structure such as  the bundles over quaternionic projective spaces.

The Hirzebruch genera appeared in algebraic geometry and became widely known due to the famous Atiyah-Singer theorem on index of differential operators on manifolds. They play a fundamental role in the theory of complex cobordisms and its applications. In all these areas in the focus of interest are classical genera such as the Todd genus, the signature and an arithmetic genus~\cite{HBJ} as well as modern genera such as an elliptic genus~\cite{OCH}, the Krichever genus~\cite{KR} and the general Krichever genus~\cite{B-10},~\cite{B-Bun-10}. The toric genus of a manifold has its values in the complex cobordism ring of the classifying space of the torus, while the toric genus for the families has its values in the complex cobordism ring of  the Borel construction for the torus action on the base. Taking the composition of the toric genus and mappings from the complex cobordism ring given by the Hirzebruch genera we obtain specialized toric genera. This opens an approach to the problems which are analogous to the well known problems such as rigidity~\cite{HBJ}. One of the most important result in this direction is the famous Krichever rigidity theorem for the Krichever genus on $SU$-manifolds~\cite{KR}. In our paper we obtain the results based on the deep connection between the rigidity problem of specialized genera and the theory of functional equations.

The context of the paper is as follows.

In Section 2 we recall the notion of a tangentially stable complex $T^k$-manifold and for the case when the action
has isolated fixed points  the formula from~\cite{BPR} which computes toric genera in terms of weights and signs of fixed points is given. We also recall~\cite{Krichever-74} the notion of $G$-genus for any smooth $G$-manifold admitting a $G$-invariant stable complex structure, where $G$ is  a compact connected Lie group. This genus takes the values in the complex cobordism ring of the classifying space $BG$.  Using the  embedding of the torus $T^k$ in the group $G$ we describe the relation among the corresponding genera. 

Section 3 is devoted to the universal toric genus on compact homogeneous spaces of positive Euler characteristic with an invariant almost complex structure and the canonical action of a maximal torus. We generalize our result from~\cite{Buch_Terz} to the flag manifolds endowed with an arbitrary invariant almost complex structure as well as to the generalized Grassmann manifolds. We compute  the top Chern numbers $s
_n$ for these manifolds and note that most of them vanish. The class of homogeneous spaces admitting an invariant $SU$-structure is also considered and   such known examples among symmetric and $3$-symmetric spaces are recalled. We  describe the flag manifolds and some $k$-symmetric spaces admitting a $SU$-structure.

Section 4 deals with the Krichever genus and  the Hirzebruch genus given by an odd power series. The main tool we use is the representation theory of Lie groups and the localization  criterion for $T^k$-rigidity of Hirzebruch genus~\cite{BPR} based on the localization formula for the universal toric genus. In this way we prove the result of~\cite{KR} that the Krichever genus is $T^k$-rigid on homogeneous $SU$-spaces. We also prove that any equivariant Hirzebruch genus defined by an odd series vanishes on flag manifolds and some $k$-symmetric spaces. It, in particular, implies that an elliptic genus as well as an arithmetic genus vanish on these manifolds.  Moreover, we prove that in the case of consideration for an arbitrary $T^n$-stable complex manifold its equivariant genus is, up to sign, independent of the given equivariant stable complex structure. 

In Section 5 we deduce the explicit formula for the $\chi _{y}$ - Hirzebruch genus in terms of weights and signs at a fixed point for  homogeneous spaces of positive Euler characteristic endowed with the action of a maximal torus and an arbitrary equivariant stable complex structure. As the consequence, we deduce the known formulas~\cite{BHII},~\cite{HS} for the signature as well as for the Todd genus of homogeneous spaces related to an invariant almost complex structure. 

In Section 6 we consider homogeneous fibrations endowed with compatible invariant almost complex structures and prove that for the universal toric genus of such fibrations the twisted product formula holds.  We further describe those fibrations for which the universal toric genus is multiplicative.  The twisted product formula enables us  to establish in this case  the strong relations between the notion of $T$-rigidity and muliplicativity for an arbitrary Hirzebruch genus. This is another way to obtain a large class of homogeneous spaces having a trivial signature, an arithmetic genus and an elliptic genus.

In Section 7 we generalize the notion of compatible almost complex structures on homogeneous fibrations to  arbitrary fibrations. Moreover supposing the base and the fiber are $T^k$-equivariant stable complex manifolds, we provide  a construction of the corresponding fibration for whose universal toric genus the twisted product formula holds.

Section 8 deals with the universal toric genus of the bundles whose total space and  base are endowed with the smooth actions of the torus $T^k$, related to which the projection is an equivariant map. It is also assumed that the total space admits a $T^k$-equivariant stable complex structure which induces a stable complex structure on the tangent bundle along the fibers. We explicitly compute the universal toric genus for some examples of such bundles over quaternionic projective spaces, for which it is  known that they do not admit any stable complex structure.

\section{Toric genera}
We refer to~\cite{BPR} and~\cite{Buch_Terz} for detailed account regarding the notions of universal toric genus.  

\subsection{Tangentially stable complex $T^k$-manifolds.} 
\numberwithin{thm}{subsection}
We  consider a tangentially stable complex $T^k$-manifold $(M^{2n}, c_{\tau},\theta )$. It means that on a compact  manifold $M^{2n}$ are given  the stable complex structure $c_{\tau}$ and the action $\theta$ of the torus $T^k$ which are compatible meaning that
\[
c_{\tau} : \tau (M^{2n}) \oplus \R ^{2l} \to \xi
\]
is a real isomorphism for some trivial bundle $\R ^{2l}$  and some complex vector bundle $\xi$ on $M^{2n}$ and the transformation given by
\begin{equation}\label{scd}
r(t)\colon\xi\stackrel{c_\tau^{-1}}{\longrightarrow}\tau(M^{2n})\oplus
\R ^{2l}\stackrel{d\theta(t)\oplus I}{\longrightarrow}
\tau(M^{2n})\oplus \R ^{2l}\stackrel{c_\tau}{\longrightarrow}\xi 
\end{equation}
is a complex transformation for any $t\in T^k$.
 
For the triple $(M^{2n}, c_{\tau}, \theta)$  it is defined the universal toric genus as follows. The Borel construction provides  fibration
\[
M^{2n}\to ET^{k}\times _{T^{k}} \dr{M^{2n}}{p} BT^k \ , 
\] 
whose tangent bundle along the fibers is endowed with the stable complex structure.
Consider the corresponding Gysin homomorphism  
\[
p_{!} : U^{*}(ET^{k}\times _{T^k}M^{2n})\to U^{*}(BT^{k}), 
\]
where $U^{*}(X)$ denotes the theory of unitary cobordisms of the space $X$. {\it The universal toric genus} of the triple $(M^{2n}, c_{\tau}, \theta)$ is defined by
\[
\Phi (M^{2n}, c_{\tau}, \theta)=p_{!}(1) \ .
\]
Recall that $U^{*}(BT^{k}) = \Omega _{U}^{*}[[u_1,\ldots ,u_k]]$ is the algebra of formal power series over $\Omega_{U}^{*}$, where by  Milnor-Novikov theorem~\cite{NM} we have that $U^{*}(pt)=\Omega _{U}^{*}=\Z [a_1,a_2,\ldots ]$ and $\deg a_i=-2i$. It is proved in~\cite{BR} that the universal toric genus  can be expressed in the following way
\begin{equation}\label{tg_geo}
\Phi (M^{2n},c_{\tau},\theta) = [M^{2n}]+\sum\limits_{|\omega |>0}[G_{\omega}(M^{2n})]{\bf u}^{\omega} \ ,
\end{equation}
where $\omega = (i_1,\ldots ,i_k)$ and ${\bf u}^{\omega}=u_1^{i_1}\cdots u_{k}^{i_k}$. Here $[M^{2n}]$ denotes the complex cobordism class of the manifold $M^{2n}$ with the stable complex structure $c_{\tau}$  and $G_{\omega}(M^{2n})$ is the stable complex manifold which is obtained as the total space of the fibration over the base $B_{\omega}$ with the fiber $M^{2n}$.  The base $B_{\omega}$ is obtained using Bott tower and it is cobordant to zero for any $\omega$ such that $|\omega |>0$. 

Under some additional assumptions the universal toric genus can be expressed in terms of the  formal group in complex cobordisms. {\it The formal group in complex cobordisms} is studied in detail in~\cite{Novikov-67} and it is defined by
\begin{equation}\label{formalgroup}
F(u,v)=c_{1}^{U}(\eta_{1}\otimes \eta_{2})\in U^{2}(BT^{2}) \ ,
\end{equation}
where by $\eta_1$ and $\eta _{2}$ we denote  the universal complex line bundle over $\C P^{\infty}$. After putting  $c_{1}^{U}(\eta _{1})=u, c_{1}^{U}(\eta _2)=v \in U^{2}(\C P^{\infty})$  we have that
\[
F(u,v)=u+v+\sum\alpha _{ij}u^iv^j 
\]
for some $\alpha _{ij} \in \Omega _{U}^{-2(i+j-1)},\; i\geqslant 1,\; j\geqslant 1$. This group is, by the result of~\cite{Q} isomorphic to Lazard's universal formal group.

For the formal group $F(u,v)$ it can be uniquely defined the corresponding power system $\{ [n](u)\in \Omega_{U}^{*}[[u]],n\in \Z \}$ by the following rule:
\[
[0](u)=0,\; [n](u)=F(u,[n-1](u))\;\mbox{for}\; n\in \Z \ .
\]
For ${\bf n} =(n_1,\ldots ,n_k)\in \Z ^{k}$ and ${\bf u}=(u_1,\ldots, u_k)$ one defines $[{\bf n}]({\bf u})$ inductively with $[{\bf n}]({\bf u})=[n](u)$ for $k=1$ and
\[
[{\bf n}]({\bf u})=F_{q=1}^{k}[n_{q}](u_{q})=F(F_{q=1}^{k-1}[n_{q}](u_{q}), [n_{k}](u_k)) \ 
\] 
for $k\geq 2$.
Using~\eqref{formalgroup} this also can be presented in terms of the first Chern class for the universal complex line bundle $\eta _{i}$ over $\C P^{\infty}$ as  
\[
[{\bf n}]({\bf u}) = c_{1}^{U}(\eta _{1}^{n_1}\otimes \cdots \otimes \eta _{k}^{n_k})\;\;  \text{for}\;\;  {\bf n}=(n_1,\ldots ,n_k)\in \Z _{\geq 0}^{k} \ .
\]
If all fixed points for the action $\theta$ on $M^{2n}$ are isolated, it is proved in~\cite{BPR},~\cite{BR} that,  appealing to the formal group, the  universal toric genus for $(M^{2n}, c_{\tau}, \theta )$  can be localized meaning that it can be expressed in terms of the local data at the fixed points  for the action $\theta$ related to $c_{\tau}$.  

Let $p$ be an isolated fixed point. It is defined 
the  $\sg (p)$ to be equal $+1$ if the map
\[
\tau_p(M^{2n})\stackrel{I\oplus\hspace{.1ex}0}{\longrightarrow}
\tau_p(M^{2n})\oplus \R^{2l} \stackrel{c_{\tau,p}}
{\longrightarrow}\xi_p \cong\C ^n \oplus\C ^{l}
\stackrel{\pi}{\longrightarrow} \C ^n \; 
\]
preserves the orientation, otherwise it is equal $-1$. Recall that here the orientation in $\tau _{p}(M^{2n})$ comes from the stable complex structure $c_{\tau}$ and it is not related to the action of the torus $T^k$, while the orientation of $\C ^{n}$ on the right hand side coincides with the orientation in $\tau _{p}(M^{2n})$ which comes from the torus action.
The action $\theta$ at a fixed point $p$ gives rise to the representation $r_{p}$ of $T^k$ in $GL(n, \C)$ for which is known from representation theory to be determined by its weights vector $\{ \Lambda _{1}(p),\ldots ,\Lambda _{n}(p)\}$. Namely, the representation $r_{p}$ decomposes into the sum of $n$ non-trivial one-dimensional representations  $r_{p,j}$. Each of $r_{p,j}$ can be written as $r_{p,j}(e^{2\pi ix_1},\ldots ,e^{2\pi ix_k})v = e^{2\pi i\left\langle \Lambda _{j}(p),{\bf x}\right\rangle}v$, for some $\Lambda _{j}(p) = (\Lambda_{j}^{1}(p),\ldots ,\Lambda_{j}^{k}(p))\in \Z ^{k}$ where ${\bf x}=(x_1,\ldots ,x_k)\in \R ^k$ and $\left\langle \Lambda _{j}(p), {\bf x}\right\rangle = \sum_{l=1}^{k}\Lambda _{j}^{l}(p)$. Then the localization formula for the universal toric genus of  $(M^{2n}, c_{\tau}, \theta)$ is given by the following Theorem~\cite{BPR}.
\begin{thm}\label{main}
If all fixed points for the action $\theta$ are isolated  then
\begin{equation}\label{TG-weights}
\Phi (M^{2n},c_{\tau},\theta)=\sum\limits_{p\in P}\sg (p)\prod\limits_{j=1}^{n}\frac{1}{[\Lambda _{j}(p)]({\bf u})} \ ,
\end{equation}
where $P$ denotes the set of fixed points and $\{ \Lambda _{j}(p)$, $1\leq j\leq n\}$ are  the corresponding weight vectors at $p\in P$.    
\end{thm}    
\subsection{The notion of $G$-genus.}
Assume that $(M^{2n}, c_{\tau})$ is a stable complex manifold and it is given on $M^{2n}$ the  smooth action $\Gamma$ of a compact connected  Lie group $G$  such that $(M^{2n}, c_{\tau}, \Gamma)$ is a tangentially stable complex $G$-manifold. It means that the composition
\begin{equation}
\xi\stackrel{c_{\tau }^{-1}}{\longrightarrow}\tau(M^{2n})\oplus
\R ^{2l}\stackrel{dg\oplus I}{\longrightarrow}
\tau(M^{2n})\oplus \R ^{2l}\stackrel{c_{\tau }}{\longrightarrow}\xi 
\end{equation}
is a complex transformation for any $g\in G$. Using  Borel construction one obtains fibration
\[  
M^{2n}\to EG\times _{G} \dr{M^{2n}}{p} BG \ . 
\]
The corresponding Gysin homomorphism
 \[
p_{!} : U^{*}(EG\times _{G}M^{2n})\to U^{*}(BG) 
\] 
defines the $G$-genus for $M^{2n}$ to be an element in $U^{*}(BG)$ given with
\[
\Phi _{G}(M^{2n}, c_{\tau}, \Gamma)=p_{!}(1) \ .
\]
\begin{ex}
Let  $M^{2n}=G/H$, where $G$ is a compact connected Lie group, $H$ is its closed connected subgroup and $J$ is an invariant almost complex structure on $G/H$, see Subsection~\ref{iacs}. We can take $\Gamma$ to be given by the canonical action of $G$ on $G/H$. The action $\Gamma$ commutes with the structure $J$, so
it is defined $\Phi _{G}(G/H, J)$.
\end{ex}
Let $T^k$ be the maximal torus in $G$. The action $\Gamma$ restricted to $T^k$ defines the smooth action $\theta$ of $T^k$ on $M^{2n}$ such that the triple $(M^{2n}, c_{\tau}, \theta)$ is a tangentially stable complex $T^k$-manifold. Since the universal toric genus $\Phi (M^{2n}, c_{\tau}, \theta)$ is an element in $U^{*}(BT^k)$ it naturally arises the question how the genera $\Phi _{G}$ and $\Phi$ are related. We already considered the similar question in~\cite{Buch_Terz} for homogeneous spaces $G/H$  of positive Euler characteristic endowed with the  canonical action of the maximal torus and a $G$-equivariant stable complex structure.
\begin{lem}\label{G-genus}
Let $Bj^{*} : U^{*}(BG)\to U^{*}(BT^k)$ be the homomorphism induced by the embedding $j: T^k\to  G$. Then 
\begin{equation}
Bj^{*}(\Phi _{G}(M^{2n}, c_{\tau}, \Gamma)) = \Phi (M^{2n}, c_{\tau}, \theta) \ .
\end{equation}
\end{lem}
\begin{proof}
The embedding $j: T^k\subset G$ produces  the commutative diagram
\[
\xymatrix{ ET^{k}\times _{T^k}M^{2n}\ar[r]\ar[d] & EG\times _{G}M^{2n}\ar[d] \\
BT^k \ar[r] & BG ,}
\]
what implies the statement since the Gysin homomorphism is functorial for bundles connected by the commutative diagram.
\end{proof}
Note that  if $H^{*}(G, \Z )$ has no torsion, since $\Omega _{U}^{*}$ has no torsion,  the Atiyah-Hirzebruch spectral sequence gives that  $U^{*}(BG)$ is a free $\Omega_{U}^{*}$-module and the natural map $U^{*}(BG)\otimes _{\Omega_{U}^{*}} \Z \to H^{*}(BG, \Z )$ is an isomorphism. It implies that in this case $U^{*}(BG)$ is the algebra of formal power series over the variables that correspond to the integer generators of the Weyl invariant polynomial algebra $\R [x_1,\ldots ,x_k]^{W_{G}}$. Namely, recall a classical result~\cite{B} that $H^{*}(BG,\R )\cong \R [x_1,\ldots x_k]^{W_{G}}$ and it has $k$  generators $P_1,\ldots ,P_k$ whose degrees are given by the twice of the  exponents of the group $G$. Then  $Bj^{*}U^{*}(BG)=\Omega _{U}^{*}[[P_1,\ldots ,P_k]]$. Together with previous Lemma it implies:
\begin{lem}\label{G-invariance}
If the action $\theta$ of the torus $T^k$ on $M^{2n}$ can be extended to the stable complex action of a compact connected Lie group $G$ having $T^k$ as a maximal torus, such that $H^{*}(G,\Z )$ has no torsion then the universal toric genus $\Phi (M^{2n}, c_{\tau}, \theta)\in \Omega _{U}^{*}[[u_1,\ldots ,u_k]]$ is invariant under the action of the Weyl group $W_{G}$ meaning that  $\Phi (M^{2n}, c_{\tau}, \theta)\in \Omega _{U}^{*}[[P_1,\ldots ,P_k]]$ where $P_1,\ldots P_k$ are the integer generators of the Weyl invariant polynomial algebra $\R [u_1,\ldots ,u_k]^{W_{G}}$.
\end{lem}

\section{The universal toric genus  of homogeneous spaces}\label{HF}
We consider compact homogeneous spaces $G/H$ meaning that $G$ is a compact connected Lie group and $H$ is its closed connected subgroup.
\subsection{On invariant almost complex and complex structures on homogeneous spaces.}\label{iacs}
In~\cite{Buch_Terz} we studied cobordism classes of compact homogeneous spaces $G/H$ of positive Euler characteristic related to the stable complex structures which are equivariant under the natural action of the maximal torus $T$ for $H$ and $G$. We devoted special attention to the almost complex structures which are furthermore invariant under the natural action of the group $G$. Recall that an almost complex structure $J$ on a connected manifold $M^{2n}$ is given by the smooth field of endomorphisms $J_{x}$, $x\in M^{2n}$ of the tangent bundle $\tau (M^{2n})$ such that for any $x\in M^{2n}$ the identity $J_{x}^2=-I_{d}$ is satisfied. It follows from~\eqref{scd} that $J$ is a stable complex structure as well where we take $\xi = \tau (M^{2n})$ and $c_{\tau}=J$.  An almost complex structure $J$ on a homogeneous space $G/H$   is said to be invariant under the canonical action of the group $G$ on $G/H$ if $J_{T_{g}(G/H)}=g^{*}J_{T_{e}(G/H)}$  for any $g\in G$, where $T_{g}(G/H)$ denotes the tangent space for $G/H$ at the point $g\cdot H$.

We refer to~\cite{BH} as the  seminal work where the problems of the  existence and the description of invariant almost complex structures on homogeneous spaces are studied. It is proved that:
\begin{itemize}
\item any such space for which subgroup $H$ is the centralizer of an element of odd order of the group $G$ admits an invariant almost complex structure; example of such space is sphere $S^6$.
\item any such space $G/H$ for which $H$ is the centralizer of a toral subgroup in $G$ admits an invariant complex structure;
\item if $G/H$ admits an invariant almost complex structure then it admits exactly $2^s$ invariant almost complex structures, where $s$ is the number of irreducible summands for the isotropy representation of $H$ at $T_{e}(G/H)$.
\end{itemize}
The quaternionic projective spaces $\HH P^{n}$, $n\geq 1$ provide examples of compact homogeneous spaces of positive Euler characteristic which admit no almost complex structure~\cite{H},~\cite{M}.

Regarding the question of the existence of invariant  complex structures on arbitrary compact homogeneous spaces we recall the results from~\cite{Wang}. A homogeneous space $G/H$ of a compact Lie group $G$ is said to be a $M$-space if $H$ is the semisimple part of the centralizer of a toral subgroup for $G$.  
Among  examples of $M$-manifolds are the compact simple Lie groups, the complex Stiefel manifolds, the quaternionic Stiefel manifolds and the ordinary Stiefel manifolds of the form $V_{n,2k}$. A homogeneous space $G/H$ is said to be a $C$-space if the semisimple part of $H$ coincides with the semisimple part of the centralizer of a toral subgroup of $G$. Then in~\cite{Wang} it is proved:
\begin{itemize}
\item Any even dimensional $M$-space admits infinitely many non-equivalent invariant complex structures and so does the product of two odd-dimensional ones. They have vanishing the second Betti number and hence they are not K\"ahler.
\item Any even dimensional $C$-space admits an invariant complex structure.
\item Any homogeneous space which admits an invariant complex structure is homeomorphic to some $C$-space. 
\end{itemize}   
We want to point that homogeneous spaces $G/H$ for which $\rk H <\rk G$ or what is equivalent with saying that Euler characteristic for $G/H$ is zero, are not interesting from the point of view of equivariant cobordism theory.
Namely, let $G/H$ be a $C$-space such that $\rk H< \rk G$ and let $L$ be the centralizer of a toral subgroup in $G$ whose semisimple  part coincides with the semisimple part of  $H$. Since $\rk L=\rk G$, we have that the dimension of the center of $H$ is strictly less than the dimension of the center of $L$, what implies that there exists a toral subgroup $T^l$, $l\geq 1$ in $G$ such that the canonical action of $T^l$ on $G/H$ is free. Therefore  we have the existence of the circle $S^1\subseteq T^l$ acting freely on $G/H$.  Consider the associated disc bundle $W=G/H\times _{S^{1}} D^{2}$. The boundary of $W$ is diffeomorphic to $G/H$ by $gH\to [(gH, 1)]$ what means that $G/H$ is equivariantly cobordant to zero.

\subsection{Generalities.}\label{Generalities}
Following~\cite{BH} we shortly remind on  the description of  invariant almost complex structures on compact homogeneous spaces of positive Euler characteristic. Any such structure  $J$ can be identified with a complex structure on $T_{e}(G/H)$ which commutes with the isotropy representation for $H$ at $T_{e}(G/H)$.  Denote  by $\gg$, $\hh$ and $\TT$ the Lie algebras for $G$, $H$ and $T^k$ respectively, where $k=\rk G=\rk H$ and $T^k$ is the maximal torus for $G$ and $H$. 
Let $\alpha_{1},\ldots, \alpha _{m}$ be the roots for $\gg$ related to $\TT$,
where $\dim G=2m+k$.  One can always choose these roots   such that $\alpha
_{n+1},\ldots, \alpha _{m}$ give the roots for $\hh$ related to
$\TT$, where $\dim H=2(m-n)+k$. The roots $\alpha _{1},\ldots,
\alpha  _{n}$ are called the {\it complementary} roots for $\gg$
related to $\hh$. Using root decomposition for $\gg$ and $\hh$ it
follows that $T_{e}(G/H) \cong \gg _{\alpha _{1}}^{\C}\oplus \ldots
\oplus \gg _{\alpha _{n}}^{\C}$, where by $\gg _{\alpha _{i}}$ is
denoted the root subspace defined with the root $\alpha _{i}$.

Since $J$ is invariant under the isotropy representation for $H$ and, thus for $T^k$, it induces 
the complex structure on each complementary root subspace $\gg _{\alpha
_{1}},\ldots ,\gg _{\alpha _{n}}$. Therefore, $J$ can be completely
described by the root system $\varepsilon _{1} \alpha _{1},\ldots
,\varepsilon _{n}\alpha _{n}$, where $\varepsilon _{i}= \pm
1$ depending  if $J$ and the adjoint representation $Ad_{T}$ define the
same orientation on $\gg _{\alpha _{i}}$ or not, where $1\leqslant
i\leqslant n$. The roots $\epsilon _{i}\alpha _{i}$ are called
{\it the roots of the almost complex structure} $J$.
 
We consider a homogeneous space $G/H$ with the canonical action of the maximal torus $T^k$ and an invariant almost complex structure $J$. It is proved in~\cite{Buch_Terz} that the weights for the canonical action of the maximal torus $T^k$ on $G/H$ at the fixed point $\rw \in W_{G}/W_{H}$ for this action related to the structure $J$ are given with $\rw (\epsilon _i\alpha _i)$, $1\leq i\leq n$. Consequently it is deduced in~\cite{Buch_Terz} that  the universal toric genus for $(G/H, J)$ is given by the formula
\begin{equation}\label{utg}
\Phi(G/H, J) = \sum_{\rw \in W_{G}/W_{H}}\prod_{i=1}^{n}\frac{1}{[\rw (\epsilon _{i}\alpha_{i})]({\bf u})}.
\end{equation}
For the Chern-Dold character~\cite{BCH} of the universal toric genus for $(G/H, J)$ it is proved in~\cite{Buch_Terz} to be given by the formula
\begin{equation}\label{Ch-Do}
ch_{U}\Phi(G/H, J) = \sum_{\rw \in W_{G}/W_{H}}\prod_{i=1}^{n}\frac{f(\left\langle \rw (\epsilon_{i}\alpha_{i}), {\bf x}\right\rangle)}{\left\langle \rw (\epsilon_{k}\alpha_{k}), {\bf x}\right\rangle} \ ,
\end{equation}
where $W_{G}$ and $W_{H}$ are the Weyl groups for $G$ and $H$, and $f(x)=1+a_1x+a_2x^2+\ldots +a_{n}x^{n}+\ldots$ for $a_{i}\in \Omega _{U}^{-2i}(\Z)$. Here $x$ denotes the first Chern class of the Hopf line bundle over $\C P^{\infty}$ and $x=ch_{U}g(u)$, where $g(u)$ is the logarithm of the formal group in complex cobordisms~\cite{Novikov-67}.
\begin{rem}
Recall~\cite{BCH} that $\Omega ^{*}_{U}(\Z )$ is the subring of $\Omega _{U}^{*}\otimes \Q$  generated by the elements having  integer Chern numbers. More precisely, $\Omega _{U}^{*}(\Z )= \sum\limits_{n\geq 0} \Omega _{U}^{-2n}(\Z )$ where $\Omega _{U}^{-2n}(\Z )=\{ a\in \Omega _{U}^{2n}\otimes \Q : s_{\omega}a\in \Omega _{U}^{0}=\Z\;\text{for}\; s_{\omega}\in S_{2n} \}$ and $S=\sum\limits_{n\geq 0} S_{2n}$ is the Landweber-Novikov algebra. 
It is proved in~\cite{BCH} that $\Omega _{U}^{*}(\Z ) = \Z [b_1,b_2,\ldots ]$, where $b_{n}=\frac{[\C P^n]}{n+1}$, $n\geq 1$.
Recall also~\cite{BCH} that the Chern-Dold character in complex cobordisms  splits into the composition
$
\dr{U^{*}(X)}{\hat{ch}_{U}}H^{*}(X, \Omega ^{*}_{U}(\Z ))\to H^{*}(X,\Omega _{U}^{*}\otimes \Q )
$
and $\hat{ch} _{U}(g(u))=x$, $\hat{ch}_{U}(u)=g^{-1}(x)$, where $u=c_{1}^{U}(\eta)$, $x=c_{1}^{H}(\eta )$ are the cobordism and the cohomology first Chern classes for  the Hopf line $\eta $ bundle over $\C P^{\infty}$. The series $g(u)=u+\sum\limits_{n\geq 1}\frac{[\C P^{n}]}{n+1} u^{n+1}$ is the logarithm of the formal group low in complex cobordisms. For its functional inverse series it is proved in~\cite{BCH} to be given by $g^{-1}(x)=x+\sum\limits_{n\geq 1}\frac{[M^{2n}]}{(n+1)!} x^{n+1}$ where $\frac{[M^{2n}]}{(n+1)!}\in \Omega _{U}^{-2n}(\Z )$. Since  $\hat{ch}_{U}u=\frac{x}{f(x)}=g^{-1}(x)$, it implies  that  $f(x)=1+a_1x+a_2x^2+\ldots$, where $a_i\in \Omega _{U}^{-2i}(\Z )$.
\end{rem}
It follows~\cite{Buch_Terz} that the complex cobordism class for $(G/H, J)$ can be obtained as the coefficient in $t^{n}$ in the polynomial   
\begin{equation}\label{cob_class}
\sum _{\rw \in W_{G}/W_{H}}\prod_{i=1}^{n}\frac{f(t\langle \rw (\varepsilon _{i}\alpha _{i}),\bf{x}\rangle)}{\langle \rw(\varepsilon _{i}\alpha _{i}),\bf{x}\rangle} \ .
\end{equation}

It is also proved in~\cite{Buch_Terz} that the tangential  characteristic numbers $s_{\omega}$  in cobordisms  for $(\tau (G/H),J)$, where $\omega=(i_1,\ldots i_n)$ and $\|\omega\|=\sum\limits_{j=1}^{n}ji_{j}=n$,  can be computed by the following formula
\begin{equation}\label{charact_numbers}
s_{\omega}(\tau (G/H),J)=\sum _{\rw \in W_{G}/W_{H}}\rw \Big(
\frac{f_\omega(t_1, \ldots, t_n)}{t_1 \cdots t_n}\Big),
\end{equation}
where $t_j=\langle  \varepsilon _{j}\alpha_j,{\bf x} \rangle$ and $f_{\omega}(t_1,\ldots t_n)$ is given by the expression 
\[
\prod_{i=1}^{n}f(t_i)=1+\sum f_\omega(t_1, \ldots, t_n)a^\omega.
\]

In~\cite{Buch_Terz} are obtained  the formulae for the cobordism classes and computed some characteristic numbers $s_{\omega}$ of the flag manifolds and the Grassmann manifolds related to the standard invariant complex structure.  We generalize further  these  formulas to  an arbitrary invariant almost complex structure on flag manifolds and generalized Grassmann manifolds. 

\subsection{Flag manifolds $U(n)/T^n$.}
Recall that as $U(n)/T^n$ admits $2^m$ invariant almost complex structures where $m=\frac{n(n-1)}{2}$, as it is totally reducible. Let $J$ be an arbitrary invariant almost complex structure on $U(n)/T^n$ and let $\alpha _{ij}=x_i-x_j$, $1\leq i<j\leq n$ be the roots for $U(n)$. The structure $J$ is described by its roots $\epsilon _{ij}\alpha _{ij}$, where $\epsilon _{ij}=\pm 1$. The Weyl group for $U(n)$ is the symmetric group $S_n$. Using~\eqref{Ch-Do}, in analougos way to that in~\cite{Buch_Terz}, we straightforwardly deduce the following statement.
\begin{thm}\label{chern-flag}
The Chern-Dold character of the toric genus for the flag manifold $(U(n)/T^n, J)$ is given by the formula
\begin{equation}\label{ChDflag}
ch_{U}\Phi (U(n)/T^n, J) = \frac{\lambda}{\Delta _{n}}\sum\limits_{\sigma \in S_{n}}\sg (\sigma)\sigma \Big( \prod\limits_{1\leq i<j\leq n}f(\epsilon _{ij}(x_i-x_j)\Big ) \ ,
\end{equation}
where $\lambda = \prod\limits_{1\leq i<j\leq n}\epsilon _{ij}$, then $\Delta _{n}=\prod\limits_{1\leq i<j\leq n}(x_i-x_j)$ and  $f(t)=1+\sum\limits_{i\geq 1}a_it^i$, while $\sg (\sigma)$ is the sign of the permutation $\sigma$.
\end{thm}
Following~\cite{Buch_Terz}, we simplify this formula using divided difference operator $L$. Recall~\cite{MC} that $L$ is defined by
\[
L{\bf x}^{\xi} = \frac{1}{\Delta _{n}}\sum\limits_{\sigma \in S_n}\sg (\sigma)\sigma({\bf x}^{\xi}) \ ,
\]
where $\xi = (j_1,\ldots ,j_n)$ and ${\bf x}^{\xi}=x_1^{j_1}\cdots x_n^{j_n}$. We make a use of the following properties of the operator $L$:
\begin{itemize}
\item $L{\bf x}^{\delta}=1$ for $\delta = (n-1,n-2,\ldots,1,0)$;
\item $L{\bf x}^{\xi}=\sg (\sigma)L\sigma({\bf x}^{\xi})$ for $\sigma \in S_n$;
\item $L{\bf x}^{\xi}=0$ for $\xi=(j_1\geq\cdots j_n\geq 0)$ and $\xi\neq \delta+\mu$ for some $\mu=(\mu_1\geq\cdots\mu_{n}\geq 0)$.
\end{itemize} 
Let
\begin{equation}
\prod\limits_{1\leq i<j\leq n}f(\epsilon_{ij}t(x_i-x_j)) = 1+\sum _{|\xi |>0}P_{\xi}(a_1,\ldots ,a_n,\ldots)t^{|\xi |}{\bf x}^{\xi} \ ,
\end{equation}
where $\xi =(j_1,\ldots ,j_{n})$ and  $|\xi |=\sum\limits_{q=1}^{n}j_{q}$. It follows from~\eqref{ChDflag} that
\[
ch_{U}\Phi(U(n)/T^n, J) = \lambda \cdot \sum\limits_{|\xi |\geq m}P_{\xi}L{\bf x}^{\xi}. 
\]  
In particular, we obtain the formula for the corresponding cobordism class:
\begin{cor}
\[
[U(n)/T^n, J]=\lambda \cdot \sum\limits_{\sigma\in S_{n}}\sg (\sigma)P_{\sigma (\delta )}(a_1,\ldots ,a_n,\ldots ) \ ,
\]
where $\delta =(n-1,n-2,\ldots ,1,0)$.
\end{cor}
Using~\eqref{charact_numbers} we derive that the characteristic number $s_{m}(U(n)/T^n, J)$, where $m=\frac{n(n-1)}{2}$ can be expressed in terms of the operator $L$ as:
\begin{equation}
s_{m}(U(n)/T^n, J)= \lambda \cdot \sum\limits_{1\leq i<j\leq n}L(\epsilon_{ij}(x_i-x_j))^{m}\ .
\end{equation}
For $n>3$ this will be polynomial in more then three variables, whose each monomial contains at most two variables.  Therefore, taking into account the given  properties of $L$ we obtain :
\begin{cor}\label{Fl-top-Chern}
\begin{equation}
s_{m}(U(n)/T^n, J)=0 \; \text{for} \; n > 3, \; \text{where}\;  m=\frac{n(n-1)}{2} \ .
\end{equation}
\end{cor}
In general, to describe an arbitrary $s_{\omega}$, $\omega =(i_1,\ldots ,i_m)$, $\sum _{l=1}^{m}li_{l}=m$ we can proceed as in~\cite{Buch_Terz}. For $(u_1,\ldots ,u_m) = (\epsilon _{ij}(x_i-x_j), i<j)$ we consider the orbit $O_{\omega}$  of the monomial $(u_1\cdots u_{i_1})(u_{i_1+1}^2\cdots u_{i_1+i_2}^2)\cdots (u_{i_1+\cdots +i_{m-1}+1}^{m}\cdots u_{i_1+\cdots +i_{m}}^{m})$ under the action of the symmetric group $S_m$. Recall that the orbit of a monomial ${\bf u}^{\eta}$ is defined by $O({\bf u}^{\eta})=\sum {\bf u}^{\eta ^{'}}$ where the sum goes over the orbit $\{ \eta ^{'} = \sigma \eta, \sigma \in S_m \}$ and ${\bf u}^{\eta}=u_1^{\eta _1}\cdots u_{m}^{\eta _{m}}$ for $\eta = (\eta _1,\ldots ,\eta _{m})$. The orbit $O_{\omega}$ can be written as  
\[
O_{\omega}=\sum\limits_{|\xi |=m}\alpha_{\omega ,\xi}{\bf x}^{\xi} \ ,
\]
for $\alpha _{\omega, \xi }\in \Z$. Therefore we obtain  that
\[
s_{\omega}(U(n)/T^n) =\lambda \cdot \sum\limits _{|\xi |=m}\alpha _{\omega ,\xi}L{\bf x}^{\xi} = \lambda\cdot \sum\limits_{\sigma \in S_n}\sg (\sigma )\alpha _{\omega ,\sigma (\delta )} \ .
\]    
It can be also proved in the same way as it was done in~\cite{Buch_Terz} for the standard invariant complex structure, that for $n\geq 4$ the cobordism class of the flag manifold $U(n)/T^n$ related to an arbitrary invariant almost complex structure $J$ is given as the coefficient at $t^{\frac{n(n-1)}{2}}$ in the series 
\[
 L(\tilde{f}(t\epsilon _{12}(x_1-x_2))\tilde{f}(t\epsilon _{n-1n}(x_{n-1}-x_{n}))\prod\limits_{\underset{(i,j)\neq (1,2),(n-1,n)}{1\leq i<j\leq n}}f(t\epsilon _{ij}(x_i-x_j)) \ ,
\]
where $\tilde{f}(t)=\sum\limits_{l\geq 1}a_{2l-1}t^{2l-1}$. It is the same  as the coefficient at $t^{\frac{n(n-1)}{2}}$
in the series
\begin{equation}\label{fl_numb}
\epsilon _{12}\epsilon _{n-1n}\Big ( L(\tilde{f}(t(x_1-x_2))\tilde{f}(t(x_{n-1}-x_{n}))\prod\limits_{\underset{(i,j)\neq (1,2),(n-1,n)}{1\leq i<j\leq n}}f(t\epsilon _{ij}(x_i-x_j))\Big ) \ .
\end{equation}
 
\begin{ex}\label{nondecomposable}
Analogously  to~\cite{Buch_Terz}, using~\eqref{fl_numb} we compute $s_{(1,0,0,0,1,0)}(U(4)/T^4, J)=\epsilon _{12}\epsilon _{34}80$ for an arbitrary invariant almost complex structure $J$.  It is also easy to compute 
\begin{align*}
s_{(0,0,2,0,0,0)}(U(4)/T^4, J)& = \epsilon _{12}\epsilon _{34}L((x_1-x_2)^3(x_3-x_4)^3)\\ &=3\epsilon _{12}\epsilon _{34}L((x_1^3x_3x_4^2-x_1^3x_3^2x_4)+(x_1x_2^2x_3^3-x_1^2x_2x_3^3)+\\ &+(x_1^2x_2x_4^3-x_1x_2^2x_4^3)+(x_2^3x_3^2x_4-x_2^3x_3x_4^2))\\& =6\epsilon _{12}\epsilon _{34}L(x_1^3x_3x_4^2+x_1x_2^2x_3^3+x_1^2x_2x_4^3+x_2^3x_3x_4)\\& =-24\epsilon _{12}\epsilon _{34}.
\end{align*}
The fact that the characteristic numbers $s_{(1,0,0,0,1,0)}$ and $s_{(0,0,2,0,0,0)}$ for $(U(4)/T^4, J)$ are nontrivial implies that the cobordism class for $(U(4)/T^4, J)$ is multiplicatively indecomposable in $\Omega _{U}^{*}$.   Otherwise, it can be represented as the product of two factors one of which should be multiple of $a_1$ and the other should be from $\Z [a_1,\ldots ,a_5]$. This further would imply that  $s_{(0,0,2,0,0,0)}(U(4)/T^4, J)=0$ what is not the case.
\end{ex}
\subsection{Generalized Grassmann manifolds} They are defined with $G_{q_1+\ldots +q_k,q_1,\ldots ,q_{k-1}}=U(q_1+\ldots+q_k)/U(q_1)\times \ldots \times U(q_k)$, where $q_1,\ldots ,q_k\geq 2$ and $k\geq 2$.
In~\cite{Buch_Terz} we computed the  cobordism classes of the  classical Grassmann manifolds $G_{q_1+q_2,q_1}$, $q_1,q_2\geq 2$, related to the unique, up to conjugation,  invariant complex structure. We push up further these  results calculating  the top Chern numbers $s_{q_1q_2}$ for $G_{q_1+q_2, q_1}$.
\begin{prop}\label{Ch_Gr}
\begin{equation}
s_{q_1q_2}(G_{q_1+q_2,q_1}) = 0\; \text{for}\; q_1,q_2\geq 3 \ .
\end{equation}
\end{prop}
\begin{proof}
It is proved in~\cite{Buch_Terz} that the cobordism class for $G_{q_1+q_2,q_1}$ is given as the coefficient in $t^{q_1q_2}$ in the polynomial 
\[
\frac{1}{q_{1}!q_{2}!}L\Big( \Delta _{q_1}\Delta _{q_{1}+1, q_1+q_2}\prod\limits_{\underset{q_1+1\leq j\leq q_1+q_2}{1\leq i\leq q_1}} f(t(x_i-x_j))\Big ),
\]
where $\Delta _{p,q}=\prod\limits_{1\leq i<j\leq q}(x_i-x_j)$. It implies that the top Chern number $s_{q_1q_2}$ is given by
\[
s_{q_1q_2}(G_{q_1+q_2,q_1})=\frac{1}{q_{1}!q_{2}!}L\Big (\Delta _{q_1}\Delta _{q_{1}+1,q_1+q_2}\sum\limits_{\underset{ q_1+1\leq j\leq q_1+q_2}{1\leq i\leq q_1}} (x_i-x_j)^{q_1q_2}\Big ) \ .
\]
In order to see when this number is non-trivial, because of the properties of the operator $L$,  we need to consider on the right hand side only  monomials having $q_1+q_2-1$ variables with the degrees $q_1+q_2-1,q_1+q_2-2,\ldots ,1$ in some order. Therefore we must have a variable in such monomial which comes from $\Delta _{q_1}$ or $\Delta _{q_{1}+1, q_1+q_2}$ and whose degree is $\geq q_1+q_2-3$. The degrees of variables in the monomials $\Delta _{q_1}$ or $\Delta _{q_{1}+1, q_1+q_2}$ are $\leq q_1-1, q_2-1$ respectively. It follows that $q_1-1\geq q_1+q_2-3$ or $q_2-1\geq q_1+q_2-3$, what implies that one of $q_1$ or $q_2$ is equal to $2$.
\end{proof}

\begin{rem}\label{Rem-Gr-Ch}
Because of  Proposition~\ref{Ch_Gr}  classical Grassmann manifolds that might have non-trivial top Chern numbers are $G_{q+2,2}=U(q+2)/U(2)\times U(q)$. In~\cite{Buch_Terz} we computed $s_{4}(G_{4,2})=-20$. For $q=3$ we obtain
\[
s_{6}(G_{5,2})=\frac{1}{12}L\Big( (x_1-x_2)(x_3-x_4)(x_3-x_5)(x_4-x_5)(\sum_{i=3}^{5}(x_1-x_i)^6+\sum_{i=3}^{5}(x_2-x_i)^6\Big ) = 70 \ .
\]
\end{rem}
Let us consider now an arbitrary generalized Grassmann manifold $G_{q_1+\ldots +q_k,q_1,\ldots q_{k-1}}$, where $q_1,\ldots ,q_k\geq 2$ and $k\geq 2$. 
Recall that any invariant almost complex structure on a homogeneous space $G/H$ is determined by its roots which can be divided into the groups that correspond to the irreducible summands in $T_{e}(G/H)$ related to the isotropy representation for $H$.
\begin{lem}
The generalized Grassmann manifold $G_{q_1+\ldots +q_k,q_1,\ldots q_{k-1}}$ has $2^{\frac{k(k-1)}{2}}$ invariant almost complex structures.
\end{lem}
\begin{proof}
The complementary roots for $U(q_1+\ldots +q_k)$ related to $U(q_1)\times \ldots \times U(q_k)$ are, up to sign, $x_i-x_j$, where $q_{1}+\ldots +q_{l-1}+1\leq i\leq q_{1}+\ldots +q_{l}$ and $q_{1}+\ldots +q_{l}+1\leq j\leq q_{1}+\ldots +q_{k}$ for $1\leq l\leq k-1$. Therefore they can be divided into the groups $R_{12},\ldots ,R_{1k},\ldots ,R_{23},\ldots ,R_{2k},\ldots ,R_{k-1k}$ where $R_{ij}=\{ x_{l_{i}}-x_{l_{j}}, q_1+\ldots q_{i-1}+1 \leq l_{i}\leq q_1+\ldots + q_{i}, q_1+\ldots +q_{j-1}+1\leq l_{j}\leq q_1+\ldots + q_{j}\}$. The sum of the root subspaces corresponding to each $R_{ij}$ will be invariant under the isotropy representation for $H$ what implies that on $G_{q_1+\ldots +q_k,q_1,\ldots ,q_{k-1}}$ we have $2^{\frac{k(k-1)}{2}}$ invariant almost complex structures.   
\end{proof}

Let now $J$ be an invariant almost complex structure on $G_{q_1+\ldots +q_k,q_1,\ldots ,q_{k-1}}$. It is defined by the roots $\epsilon_{ij}(x_i-x_j)$, where $q_{1}+\ldots +q_{l-1}+1\leq i\leq q_{1}+\ldots +q_{l}$ and $q_{1}+\ldots +q_{l}+1\leq j\leq q_{1}+\ldots +q_{k}$ for $1\leq l\leq k-1$,  and $\epsilon_{i}\pm 1$. Note that for a fixed $l$ the corresponding $\epsilon _{ij}$ are equal for $q_{1}+\ldots +q_{s}+1\leq j\leq q_{1}+\ldots +q_{s+1}$, where $l\leq s\leq k-1$. 
\begin{thm}
The cobordims class of the generalized Grassmann manifold $G_{q_1+\ldots +q_k,q_1,\ldots ,q_{k-1}}$ related to an invariant almost complex structure $J$ is given as the coefficient in $t^{m}$, $m=\sum\limits_{1\leq i<j\leq k}q_iq_j$ in the series in $t$
\begin{equation}\label{Co-Gr}
\frac{\lambda}{q_1!\cdots q_{k}!}L\Big( \Delta_{q_1}\times \Delta_{q_1+1,q_1+q_2}\times \ldots \times \Delta_{q_1+\ldots q_{k-1}+1,q_{1}+\ldots +q_{k}}\prod f(t\epsilon_{ij}(x_i-x_j))\Big ) \ ,
\end{equation}
where $\lambda = \prod \epsilon_{ij}$ and 
$\Delta_{p,q}=\prod\limits_{p\leq i<j\leq q}(x_i-x_j)$.
\end{thm}
\begin{cor}\label{GR-top-Chern}
\begin{equation}
s_{m}(G_{q_1+\ldots +q_k,q_1,\ldots ,q_{k-1}}, J)=0\;\; \text{for}\; k\geq 3,\;\; \text{where}\;\; m=\sum\limits_{1\leq i<j\leq k}q_iq_j.
\end{equation}
\end{cor}
\begin{proof}
We obtain from the formula~\eqref{Co-Gr} that the  number $s_{m}$, $m=\sum\limits_{1\leq i<j\leq k}q_iq_j$  is  given by
\begin{equation}\label{charact-Gr}
s_{m}(G_{q_1+\ldots +q_k,q_1,\ldots ,q_{k-1}}, J)=
\end{equation}
\[
\frac{\lambda}{q_1!\cdots q_k!}L\Big( \Delta_{q_1}\Delta_{q_1+1,q_1+q_2}\cdots \Delta_{q_1+\ldots +q_{k-1}+1,q_1+\ldots +q_k}(\epsilon_{ij}(x_i-x_j))^{m}\Big ) \ ,
\]
where $q_{1}+\ldots +q_{l}+1\leq i\leq q_{1}+\ldots +q_{l+1}$ and $q_{1}+\ldots +q_{l+1}+1\leq j\leq q_{1}+\ldots + q_{k}$ for $0\leq l\leq k-2$. If assume this number to be non-trivial, the properties of the operator $L$ as in the case of the proof of Proposition~\ref{Ch_Gr}, would imply that we must have a monomial in the right hand side of~\eqref{charact-Gr} in which the variable coming from  $\Delta_{q_1}\Delta_{q_1+1,q_1+q_2}\cdots \Delta_{q_1+\ldots +q_{k-1}+1,q_1+\ldots +q_{k}}$ has degree $\geq q_1+\ldots +q_{k}-3$. Since any such variable has degree $\leq \mbox{max} \{q_1-1,\ldots ,q_{k}-1\}$ it would imply that the sum of some $k-1$ numbers among $q_1,\ldots ,q_{k}$ is $\leq 2$ what is impossible for $k\geq 3$.
\end{proof}
\begin{rem}
The top Chern numbers $s_n$ have a special role in the complex cobordism theory as they determine the multiplicative generators in the complex cobordism ring~\cite{NM}. Therefore  when considering homogeneous spaces  equipped with an invariant almost complex structure, the following questions arise to be important to consider: compute their top Chern numbers, determine those homogeneous spaces having non-trivial top Chern numbers and find homogeneous space of a given dimension having the minimal non-trivial top Chern number.  Related to these questions, the Grassmann and the generalized  Grassman manifolds reduce to the case $G_{q+2,2}$, as   Proposition~\ref{Ch_Gr}, Remark~\ref{Rem-Gr-Ch} and Corollary~\ref{GR-top-Chern} show. We want to point that the question of stable complex structures on the Grassmann and the generalized Grassmann manifolds with non-trivial top Chern numbers we consider to be open.  Note that although most of the Grassmann manifolds have trivial top Chern number, they do not fiber in the class of homogeneous spaces. On the other hand flag manifolds fiber, but as Example~\ref{nondecomposable} show they still do not have to be decomposable in the complex cobordism ring. Recall also~\cite{Buch_Terz} that the top Chern number for $\C P^n$ related to the standard complex structure is $n+1$ and it is minimal for $n=p-1$, where $p$ is a prime number. For example, it is minimal for $n=1$ and  $n=2$, but it is not minimal for $n=3$. On the other hand, as it is showed in~\cite{Buch_Terz}, there exists on $\C P^3$ the equivariant stable complex structure whose top Chern number is $2$ and thus minimal. Note that in the same way one can show  that  $\C P^{2n+1}$ for any $n\geq 0$  admits equivariant stable complex structure whose top Chern number is equal to $2$, but for $n\geq 2$ it is not minimal. We also want to recall that among $6$-dimensional stable complex $SU$-manifolds the minimal value for the top Chern number is $6$ and it is realized by $S^6$ endowed with the invariant almost complex structure.    
\end{rem}  
\subsection{Homogeneous $SU$-manifolds}
We say that an even-dimensional stable complex manifold $(M^{2n}, J)$ is a $SU$-manifold if it's first Chern class vanishes. In this paper we consider homogeneous $SU$-manifolds $(G/H, J)$ of positive Euler characteristic, meaning that $J$ is an invariant almost complex $SU$-structure. 
It follows from~\cite{BH} that the total Chern class for $(G/H, J)$ can be obtained using root description  for $J$. If $\varepsilon _{1} \alpha _{1},\ldots ,\varepsilon _{n}\alpha _{n}$ are the roots for $J$ then  the total Chern class for $(G/H, J)$ is given by
\[
c(G/H,J)=\prod_{i=1}^{n}(1+\varepsilon _{i}\alpha _{i}).
\]
It implies that the first Chern class is 
\[
c_{1}(G/H,J)=\sum _{i=1}^{n}\varepsilon _{i}\alpha _{i}.
\]
Therefore for $(G/H,J)$ being a $SU$-manifold the roots for $J$ satisfy the condition
\[
\sum_{i=1}^{n}\varepsilon _{i}\alpha _{i}=0.
\]
\begin{rem}\label{int-not-SU}
If we assume $J$ to be integrable, it follows from~\cite{BH} that it can be
chosen an ordering on the canonical coordinates for $\TT$ such that
the roots  $\varepsilon _{1} \alpha _{1},\ldots ,\varepsilon _{n}
\alpha _{n}$ which define $J$ form the closed system of positive
roots. Note that this implies that for an integrable $J$  a homogeneous complex manifold 
$(G/H,J)$ can not be a $SU$-manifold. 
\end{rem}

\subsubsection{Examples of symmetric and 3-symmetric  $SU$-manifolds.} 
\numberwithin{thm}{subsubsection}If the second Betti number for a manifold $M^{2n}$ is zero, then $(M^{2n}, J)$ is a $SU$-manifold for any almost complex structure $J$ on $M^{2n}$. 
The list of all  symmetric spaces which admit an invariant almost complex structure, but which do not admit any invariant complex structure since  their second Betti number is zero is given in~\cite{BH}: $S^6$, $F_4/A_2\times A_2$, $E_6/A_2\times A_2\times A_2$, $E_7/A_2\times A_5$, $E/_8/A_8$ and 
$E_8/A_2\times A_6$. The existence of an invariant almost complex structure on these spaces, as it is shown in~\cite{BH}, follows from the fact that for any of them the stationary subgroup is the centralizer of an element of order $3$ or $5$ of the group. These spaces  provide  examples of $SU$-manifolds.  
 
If $G/H$ is a $3$-symmetric space then $H$ is the fixed point subgroup of some automorphism $\theta$ of the group $G$ of order $3$.  It is well known that any such space admits the canonical almost complex structure $J$ defined by
\[
\Theta=\frac{1}{2}Id+\frac{\sqrt{3}}{2}J,
\]
where $Id$ is the identity transformation of $T(G/H)$ and $\Theta = d\theta$.

It is proved in~\cite{K} that the first Chern class that corresponds to the canonical almost complex structure vanishes for exactly the following compact irreducible $3$-symmetric spaces of the classical Lie groups:
\begin{align*}
&SU(3m)/S(U(m)\times U(m)\times U(m)),\;\; m\geqslant 1,\\
&SO(3m-1)/U(m)\times SO(m-1),\;\; m\geqslant 2,\\
&Sp(3m-1)/U(2m-1)\times Sp(m),\;\; m\geqslant 1.
\end{align*}
Therefore, these spaces provide examples of homogeneous $SU$-manifolds.
    
\begin{rem}
The question of  vanishing of the first Chern class  for the canonical almost complex structure on $3$-symmetric spaces  is considered in~\cite{K} and in a connection with the existence of Einstein metric.   It is known~\cite{G} that the canonical almost complex structure $J$ on a $3$-symmetric space is nearly K\"ahler. From the
classification  of $3$-symmetric spaces done in~\cite{WG}  it follows that the canonical almost complex structure is not K\"ahler only for the following these spaces: $SU(n+1)/S(U(k)\times U(m-k)\times U(n-m+1))$, $1\leqslant k < m\leqslant n$, $SO(2n+1)/U(m)\times SO(2n-2m+1)$, $2\leqslant m\leqslant n$, $Sp(n)/U(n)\times Sp(n-m)$, $1\leqslant m\leqslant n-1$, $SO(2n)/U(m)\times SO(2n-2m)$, $2\leqslant m\leqslant n-1$, $n\geqslant 4$.  As it is proved in~\cite{WT} the vanishing of the first Chern class for a nearly K\"ahler structure   that is not a K\"ahler gives the  sufficient condition for the existence of a compatible with that structure Einstein metric. 
Within the nearly K\"ahler and not K\"ahler $3$-symmetric spaces $G/H$ the listed ones have vanishing the first Chern class and it gives the sufficient condition for the normal homogeneous metrics on $G/H$ to be Einstein. This is proved in~\cite{K} using classification of irreducible compact simply connected $3$-symmetric spaces  from~\cite{WG} and the results of~\cite{BH}.   
\end{rem}

We generalize the first series of the above $3$-symmetric spaces to  analogous $k$-symmetric spaces~\cite{TF}.  
\begin{prop}\label{k-SU}
The $k$-symmetric space $M=U(km)/(\underbrace{U(m)\times \cdots \times U(m))}_{k}$ admits an invariant almost complex $SU$-structure for odd $k$. 
\end{prop}
\begin{proof}
The complementary roots for $U(km)$ related to $(U(m))^{k}$ can be, up to sign, divided into the groups $R_{12},\ldots ,R_{1k},R_{23},\ldots ,R_{2k},\ldots ,R_{k-1k}$, where $R_{ij}=\{ x_{(i-1)m+s}-x_{(j-1)m+l}, 1\leq l,s\leq m\}$.  The roots of an arbitrary invariant almost complex structure $J$ are given with $\epsilon _{ij}R_{ij}$ where $\epsilon _{ij}=\pm 1$ can be chosen arbitrarily  and $1\leq i<j\leq k$. It implies that 
\[
c_{1}(M, J) = m\sum\limits_{i=1}^{km}(-\sum\limits_{j=1}^{i-1}\epsilon _{ji}+\sum\limits_{j=i+1}^{k}\epsilon _{ij})x_i.
\]
Let $J$ be defined by putting $\epsilon _{ij}=(-1)^{i+j+1}$. Then
\[
c_{1}(M, J) = m\sum\limits_{i=1}^{km}(-1)^{i+1}(-\sum_{j=1}^{i-1}(-1)^j+\sum_{j=i+1}^{k}(-1)^j)x_i = 0,
\]
since by assumption $k$ is odd.  
\end{proof} 
For $k=3$ proceeding as in the proof of Proposition~\ref{k-SU} we obtain:
\begin{cor}\label{canonical-roots}
The canonical almost complex structure  is unique, up to conjugation, invariant almost complex $SU$-structure on $U(3m)/(U(m)\times U(m)\times U(m))$ and its roots are given with 
\begin{align*}
&x_i-x_j,\;\; 1\leq i\leq m,\; m+1\leq j\leq 2m,\\
&x_j-x_i,\;\; 1\leq i\leq m,\; 2m+1\leq j\leq 3m,\\
&x_i-x_j,\;\; m+1\leq i\leq 2m,\; 2m+1\leq j\leq 3m. 
\end{align*}
\end{cor}
\subsubsection{Examples of $SU$-flag manifolds.}
Let us consider the flag manifolds $U(n)/T^n$, $n\geqslant 1$. The roots of any almost complex structure are given with $\varepsilon _{ij}\alpha _{ij}$, where $\alpha _{ij}=(x_i-x_j)$, $1\leqslant i<j\leqslant n$,  are the roots for $U(n)$.   
\begin{lem}\label{flag-even}
For $n$ even the first Chern class of any invariant almost complex structure on the flag manifold $U(n)/T^n$ is not trivial.
\end{lem}
\begin{proof}
 Note that $\alpha _{ij}=\alpha _{1j}-\alpha _{1i}$ for $2\leqslant i<j\leqslant n$, where $\alpha _{12},\ldots ,\alpha _{1n}$ are linearly independent. Let $J$ be an invariant almost complex structure on $U(n)/T^n$ with roots  $\varepsilon _{ij}\alpha _{ij}$, $1\leqslant i <j\leqslant n$.  If $c_{1}(U(n)/T^n, J)=0$ we would have that
\[
\sum_{1\leqslant i<j\leqslant n}\varepsilon _{ij}\alpha _{ij}=0,
\]
what can be written as 
\[
\sum _{j=2}^{n}(\sum _{i=1}^{j-1}\varepsilon_{ij} -\sum_{i=j+1}^{n}\varepsilon _{ji})\alpha_{1j}=0.
\]
This would imply that
\begin{equation}\label{cond_vanish}
\sum_{i=1}^{j-1}\varepsilon _{ij}-\sum_{i=j+1}^{n}\varepsilon _{ji}=0\;\; \text{for}\;\; 2\leqslant j\leqslant n.
\end{equation}
For $n$ even there are no choices for  $\varepsilon_{ij}=\pm 1$ which satisfy these relations as each of these sums has $n-1$ summands which is an odd number.
\end{proof}

\begin{lem}\label{U_SU}
If $n$ is odd then the flag manifold $U(n)/T^n$ admits an invariant almost complex structure whose first Chern class vanishes.
\end{lem}
\begin{proof}
Let us consider an invariant almost complex structure $J$ defined with the roots $\varepsilon _{ij}\alpha _{ij}$, $1\leqslant i<j\leqslant n$, where $\epsilon _{ij} = (-1)^{i+j+1}$.
Taking into account that $n$ is an odd number, it checks directly that $\varepsilon_{ij}$ defined in this way satisfy
equation~\eqref{cond_vanish}, and therefore the first Chern class for $(U(n)/T^n, J)$ is trivial. 
\end{proof}  

\begin{defn}
Flag manifolds $U(2n+1)/T^{2n+1}$ with an invariant almost complex $SU$-structure will be called $SU$-flag manifolds.
\end{defn}

The similar statement is not true any more for generalized  flag manifolds. 

\begin{prop}
No generalized flag manifold $U(n)/(T^k\times U(n-k)$), $1\leq k\leq n-2$  admits an invariant almost complex  $SU$-structure.
\end{prop}
\begin{proof}
The complementary roots for $U(n)$ related to $T^{k}\times U(n-k)$  are given, up to sign, with $x_i-x_j$, where $1\leq i\leq k$, $k+1\leq j\leq n-k$. If $J$ is an invariant almost complex structure then its roots must be invariant under the isotropy representation of $U(n-k)$ and thus under the action of the Weyl group $W_{U(n-k)}$. Therefore for a fixed $i$, $1\leq i\leq k$ all the roots for $J$ of the form $\pm (x_i-x_j)$ have to be of the same sign. It follows that $x_i$ may not vanish in the sum of the roots for $J$. 
\end{proof}

\begin{rem}  The result which in terms of the  Koszul form gives the necessary and the sufficient condition for a homogeneous complex space $G/H$ to have vanishing the first Chern class is obtained recently in~\cite{GG}. Being complex a such space according to Remark~\ref{int-not-SU} is of zero Euler characteristic meaning that $\rk H <\rk G$ what further implies that it is cobordant to zero.  Appealing to this result,  several series of homogeneous spaces whose first Chern class vanishes are provided in~\cite{GG}. For example among these are the spaces $M=SU(n)/(SU(n_1)\times \cdots \times SU(n_k))$, where $n_1+\ldots +n_k=n$ and  $k$ is odd. They admit a free action of the torus $T^{k-1}$ implying that they are null cobordant.   
\end{rem}
\subsubsection{Cobordism classes of homogeneous $SU$-manifolds.}
Using~\eqref{charact_numbers} we straightforwardly obtain:
\begin{prop}\label{SU-numb}
If $(G/H,J)$ is a homogeneous $SU$-manifold of  positive Euler characteristic then
\[
s_{(1,0,\ldots,0,1,0)}(G/H,J)=-s_{(0,\ldots ,0,1)}(G/H,J)= - \sum _{\rw \in
W_{G}/W_{H}} \rw \Big( \frac{\sum\limits_{j=1}^n t_j^n}{t_1 \cdots
t_n}\Big).
\]
\end{prop} 

\begin{ex}\label{S6}
It is shown in~\cite{Buch_Terz} that the roots for the unique invariant almost complex structure $J$ on $S^{6}=G_2/SU(3)$  are $\alpha _{1}=x_1$, $\alpha _{2}=x_2$ and $\alpha _{3}=x_3=-(x_1+x_2)$. Note that $(S^6, J)$ is a $SU$-manifold, its cobordism class is  $[S^6,J]=2(a_1^3-3a_1a_2+3a_3)$ and it is computed in~\cite{Buch_Terz}.
\end{ex}
\begin{lem}
The cobordism class of a  $6$-dimensional homogeneous $SU$-manifold $(G/H, J)$ is given by 
\[
[G/H, J] = \frac{\chi (G/H)}{2}\cdot [S^6, J].
\]
\end{lem}
\begin{proof}
Let $(G/H,J)$ be a $6$-dimensional homogeneous $SU$-manifold and let  $\alpha _1,\alpha _2,\alpha_3$ be the roots for $J$. The cobordism class $[G/H, J]$ is by~\eqref{cob_class} given as the coefficient in the polynomial
\begin{equation}\label{cobord_3}
\sum_{\rw \in W_{G}/W_{H}}\rw \Big( \frac{\prod_{i=1}^{3}(1+a_1\alpha _{i}t+a_2\alpha_{i}^2t^2+a_3\alpha_{i}^3t^3)}{\alpha_1\alpha_2\alpha_3}\Big).
\end{equation}
Since $\alpha _1+\alpha _2+\alpha_3=0$ it follows that the coefficient in $t^3$  in the polynomial $\prod\limits_{i=1}^{3}(1+a_1\alpha_{i}t+a_2\alpha_{i}^2t^2+a_3\alpha_{i}^3t^3)$ is 
\[
(a_1^3-3a_1a_2+3a_3)\alpha_1\alpha_2\alpha_3.
\]
It follows then from~\eqref{cobord_3} that 
\[
[G/H, J] = \chi (G/H)(a_1^3-3a_1a_2+3a_3).
\]

As $\chi (G/H)$ is an even number together with Example~\ref{S6} we obtain that the cobordism class for any  $6$-dimensional homogeneous $SU$-manifold $(G/H,J)$ is given with $[G/H, J]= c\cdot [S^6, J]$, where $c=\frac{\chi (G/H)}{2}$.
\end{proof}
\begin{ex}
The flag manifold $U(3)/T^3$  with the invariant almost complex structure $\hat{J}$ defined as in the proof of Lemma~\ref{U_SU} is a homogeneous $SU$-space. The roots for $\hat{J}$ are: $x_1-x_2, -x_1+x_3, x_2-x_3$. Being a $6$-dimensional manifold its cobordism class is 
\[
[U(3)/T^3, \hat{J}]=3[S^6, J].
\]
\end{ex}

\begin{ex}
Let us consider the flag manifold $U(n)/T^n$ where $n$ is an odd number and assume that it is  endowed with the invariant almost complex structure $J$ defined in the proof of Lemma~\ref{U_SU}. The roots for $J$ are
\begin{equation}\label{SU-roots}
(-1)^{i+j+1}(x_i-x_j),\;\;\text{for}\;\; 1\leqslant i< j\leqslant n.
\end{equation} 
Using Theorem~\ref{chern-flag} we deduce that the cobordism class for $(U(n)/T^n,J)$ is given as the coefficient at $t^{\frac{n(n-1)}{2}}$ in the polynomial
\[
\sum _{\sigma \in S_{n}}\sigma \Big(\prod_{1\leqslant i<j\leqslant n}\frac{f((-1)^{i+j-1}t(x_i-x_j))}{(-1)^{i+j-1}(x_i-x_j)}\Big),
\]
where $f(x)=1+a_1x+\ldots+a_{m}t^m$ for $m=\frac{n(n-1)}{2}$. Corollary~\ref{Fl-top-Chern} and Proposition~\ref{SU-numb} imply that
\[
s_{(1,0,\ldots,1,0)}(U(n)/T^n,J)=s_{(0,\ldots,0,1)}(U(n)/T^n,J)=0\;\; \text{for}\;\; n>3.
\]
\end{ex} 

\section{Rigidity of a Hirzebruch genus on homogeneous spaces} 
\numberwithin{thm}{section}
We refer to~\cite{HTM} and~\cite{HBJ} as the comprehensive background for the topics related to a Hirzebruch genus.

Let $A$ be a torsion-free ring. Recall the notion of a Hirzebruch genus. Assume we are given a power series 
\[
f(u)=u+\sum\limits_{k\geq 1}f_ku^{k+1},
\]
where $f_k\in A\otimes \Q$. The formal series 
\[
\prod\limits_{i=1}^n\frac{u_i}{f(u_i)}
\] 
is a symmetric function in variables $u_1,\ldots ,u_n$ what implies that it can be represented in the form $\LLL_{f}(\sigma _1,\ldots ,\sigma _n)$, where $\sigma _k$ is the $k$-th elementary symmetric function in variables $u_1,\ldots ,u_n$. The Hirzebruch genus $\LLL_{f}(M^{2n})$ of a stable complex manifold $M^{2n}$ is defined to be the value of $\LLL_{f}(c_1,\ldots ,c_n)$ on the fundamental class $[M^{2n}]$, where $c_i$ are the Chern classes of the tangent bundle for $M^{2n}$.  
Any Hirzebruch genus defines the ring homomorphism $\LLL_{f} : \Omega _{U}^{*}\to A\otimes \Q$. The vice versa is also true: for any ring homomorphism $\phi : \Omega _{U}^{*}\to A\otimes \Q$ there exists  series $f(u)\in A\otimes \Q [[u]]$, $f(0)=0$ and $f^{'}(0)=1$ such that $\phi =\LLL_{f}$. 

A Hirzebruch genus is said to be $A$-integer if $\LLL _{f}(M)\in A$. Among the examples of $A$-integer genera are the signature and the Todd genus. Without loss of generality we assume further that $A$ is a $\Q$-algebra.

For any Hirzerbruch genus
$\LLL _{f}: \Omega _{U}^{*}\to A$ there is  $T^k$-equivariant extension $\LLL _{f}^{T^k} : \Omega _{U:T^k}^{*}\to A[[u_1,\ldots ,u_k]]$ defined by the composition $\LLL _{f}\circ \Phi$, where $\Omega _{U:T^k}^{*}$ denotes the cobordism classes of $T^k$-equivariant stable complex manifolds.  It follows from~\eqref{tg_geo} that 
\[
\LLL _{f} ^{T^k}(M, c_{\tau}) = \LLL _{f}(M) + \sum\limits_{|\omega |>0}\LLL _{f} (G_{\omega}(M)){\bf u}^{\omega}.  
\]
The genus $\LLL _{f}$ is said to be $T^k$-rigid on $M$ if $\LLL _{f}^{T^k}(M, c_{\tau})=\LLL _{f}(M)$. If the action of the torus $T^k$ on $M$ has isolated  fixed points the formula~\eqref{TG-weights} implies that the conditions for the Hirzebruch genus $L_{f}$ to be $T^k$-rigid can be described in terms of functional equations  in signs and weights at fixed points  of the given action~\cite{BPR}.
\begin{prop}\label{prop-rigidity}
For any series $f$ over a $\Q$-algebra $A$, the genus $\LLL _{f}$  
is $T^k$-rigid on $M^{2n}$ if and only if the functional equation
\begin{equation}\label{rigidity}
\sum\limits_{x\in Fix(M)}\sg (x)\prod _{j=1}^{n}\frac{1}{f(\omega _{j}(x)\cdot u)}=c
\end{equation}
is satisfied in $A[[u_1,\ldots,u_k]]$, for the constant  $c = \LLL _{f}(M^{2n})$.
\end{prop}
The genus $\LLL _{f}$ is said to be $T^k$-rigid on a  given class of $T^k$-manifolds $\MMM$ if it is $T^k$-rigid for any manifold from the class $\MMM$. 
\subsection{Rigidity of the Krichever genus}
\numberwithin{thm}{subsection}
Recall that the Krichever genus~\cite{KR}, also known as the generalized elliptic genus, is a Hirzebruch genus defined by the power series 
\begin{equation}\label{KR-defn}
f(u)=\frac{exp(\mu u)}{B (u,v)}.
\end{equation}
Here $B(u,v)$ is the Baker-Akhiezer function defined by
\[
B(u,v)= B(u,v;\omega _1,\omega _2)=\frac{\sigma (v-u)}{\sigma (u)\sigma (v)}e^{\zeta (v)u} ,
\]
where $\sigma (u)$ and $\zeta (u)$ are Weierstrass sigma and zeta functions and $\omega _1$, $\omega _2$ are half-periods for an elliptic curve $\Gamma$  such that $Im\frac{\omega _{2}}{\omega _{1}}>0$.  Related to the periods $\omega _1$ and $\omega _{2}$ the function $\sigma (u)$ behaves as follows:
\begin{equation}\label{sigma-period}
\sigma (u+2\omega _k ) = -e^{2\eta _{k}(u+\omega _{k})}\sigma (u),\;\; \text{where}\; \eta _{k}=\zeta (\omega _k),\; k=1,2.
\end{equation}

In the case when  $\MMM$ is the class of $SU$-manifolds with an equivariant circle action, it is proved in~\cite{KR} that the Krichever genus is $S^1$-rigid on this class. Note that it implies  that the Krichever genus is $G$-rigid on $\MMM$  for any compact connected Lie group $G$. One can see it by passing to a generic circle subgroup in the maximal torus for $G$, see~\cite{KR}. 

For the class of homogeneous $SU$-manifolds of positive Euler characteristic endowed with the canonical $T^k$-action we can prove that the Krichever genus is $T^k$-rigid  just using Proposition~\ref{prop-rigidity} and the representation theory of Lie groups.  It will  imply that the Krichever genus  is also $S^1$-rigid on these spaces related to the canonical $S^1$-action.   

\begin{rem}\label{pointing}
For the sake of further clearness we want to point that all  weights $\rw (\alpha)$ at an arbitrary fixed point $\rw \in W_{G}/W_{H}$ for the canonical action of the maximal torus $T^k$ on $G/H$ related to an invariant almost complex structure,  are of multiplicity $1$. It follows from the well known fact   that $k\alpha$ may not be the root for $G$   for $k\neq \pm 1$, where  $\alpha$ is an arbitrary root for $G$.  
\end{rem}
 
\begin{thm}\label{Krichever-rigidity}
The Krichever genus is $T^k$-rigid on homogeneous $SU$-manifolds  of positive Euler characteristic endowed with the canonical action of the maximal torus.
\end{thm}
\begin{proof}
In the case of a homogeneous space with an invariant almost complex structure the  left hand side of the expression~\eqref{rigidity} for the Krichever genus obtains the form
\begin{equation}\label{Kr-homogeneous}
\sum_{\rw \in W_{G}/W_{H}}e^{(\zeta (v)-\mu)\sum\limits_{j=1}^{n}\rw (\alpha_{j}\cdot u)}\prod_{j=1}^{n}\frac{\sigma (v-\rw (\alpha_{j}\cdot u))}{\sigma (v)\sigma(\rw (\alpha_{j}\cdot u))} \ ,
\end{equation}
where $\alpha _{j}$, $1\leq j\leq n$  are the roots for  $J$. For $J$ being a $SU$-structure we have  $\sum_{j=1}^{n}\alpha _j=0$ what implies that $\sum_{j=1}^{n}\rw (\alpha _{j}\cdot u)=0$ for any $\rw \in W_{g}/W_{H}$  and the expression~\eqref{Kr-homogeneous} becomes
\[
\sum_{\rw \in W_{G}/W_{H}}\prod_{j=1}^{n}\frac{\sigma (v-\rw (\alpha_{j}\cdot u)}{\sigma (\rw (\alpha_{j}\cdot u))\sigma (v)}.
\]
Choose a maximal system  $\alpha _1,\ldots ,\alpha _k$ of linearly independent roots for $J$. Then for any $k+1\leq j\leq n$ we have $\alpha _j=\sum_{i=1}^{k}c_{i}^{j}\alpha _{i}$, where $c_{i}^{j}\in \Z$. If denote $z_{i}=\alpha _{i}\cdot u$, $1\leq i\leq k$, the above expression becomes
\begin{equation}\label{kr-rig}
\sum_{\rw \in W_{G}/W_{H}}\prod_{j=1}^{k}\frac{\sigma (v-\rw (z_j))}{\sigma (\rw (z_j))\sigma (v)}\prod_{j=k+1}^{n}
\frac{\sigma (v-\sum _{i=1}^{k}c_{i}^{j}\rw (z_i))}{\sigma (\sum_{i=1}^{k}c_{i}^{j}\rw (z_i))\sigma (v)}.
\end{equation}
Set $\rw (z_i)=y_i$, where $1\leq i\leq k$, and consider the function
\begin{equation}\label{p-function}
p(y_1,\ldots, y_k)=\prod_{j=1}^{k}\frac{\sigma (v-y_j)}{\sigma (y_j))\sigma (v)}\prod_{j=k+1}^{n}
\frac{\sigma (v-\sum _{i=1}^{k}c_{i}^{j}y_i)}{\sigma (\sum_{i=1}^{k}c_{i}^{j}y_i)\sigma (v)},
\end{equation}
where $\sum _{i=1}^{k}y_i + \sum _{j=k+1}^{n}\sum _{i=1}^{k}c_{i}^{j}y_i = 0.$ This implies that $1+\sum_{j=k+1}^{n}c_{i}^j=0$ for $1\leq i\leq k$.
This function is  two-periodic in each of variables $y_1,\ldots ,y_k$ with periods equal to $2\omega _1$ and $2\omega _2$. We prove it here for $y_1$, using~\eqref{sigma-period}. We first note the following  
\[
\sigma (c_{1}^{j}(y_1+2\omega _{l})+\sum _{i=2}^{k}c_{i}^{j}y_i) = -e^{2\eta_{l}(c_{1}^{j}(  \sum_{i=1}^{k}c_{i}^{j}y_i+c_{1}^{j}\omega _{l}))}\sigma (\sum _{i=1}^{k}c_{i}^{j}y_i),
\]
\[
\sigma (v-c_{1}^{j}(y_1+2\omega _{l})-\sum _{i=2}^{k}c_{i}^{j}y_i) = -e^{2\eta_{l}(c_{1}^{j}(  \sum_{i=1}^{k}c_{i}^{j}y_i-v+c_{1}^{j}\omega _{l}))}\sigma (v-\sum _{i=1}^{k}c_{i}^{j}y_i),
\]
what implies 
\[
\frac{\sigma (v-c_{1}^{j}(y_1+2\omega _{l})-\sum _{i=2}^{k}c_{i}^{j}y_i)}{\sigma (c_{1}^{j}(y_1+2\omega _{l})+\sum _{i=2}^{k}c_{i}^{j}y_i)} = e^{-2c_{1}^{j}\eta_{l}v}\frac{\sigma (v-\sum _{i=1}^{k}c_{i}^{j}y_i)}{\sigma (\sum _{i=1}^{k}c_{i}^{j}y_i)}.
\]
In this way we  obtain
\begin{align*}
p(y_1+2\omega _l,\ldots ,y_k) &= e^{-2\eta _{l}v}\frac{\sigma (v-y_1)}{\sigma (y_1)\sigma (v)}\prod_{j=2}^{k}\frac{\sigma (v-y_j)}{\sigma (y_j)\sigma (v)}\prod _{j=k+1}^{n}e^{-2c_{1}^{j}\eta_{l}v}\frac{\sigma (\sum _{i=1}^{k}c_{i}^{j}y_i-v)}{\sigma (\sum_{i=1}^{k}c_{i}^{j}y_i)\sigma (v)}\\&
=e^{-2\eta_{l}(1+\sum_{j=k+1}^{n}c_{1}^{j})v}p(y_1,\ldots,y_k)=p(y_1,\ldots ,y_k).
\end{align*}
It implies that the function given by~\eqref{kr-rig} is two-periodic in each variable.

On the other hand formula~\eqref{tg_geo} coming from topology proves that any equivariant Hirzebruch genus $\LLL _{f}^{T^k}$ has no pole at $u=0$ what implies that the function given with~\eqref{kr-rig} has no zero poles, i.~e.~ at $\rw(z_j)=0$, $1\leq j\leq n$.  By Remark~\ref{pointing} the poles for this function  are all of multiplicity $1$ and, therefore, they are all at the lattice $2s\omega _1+2m\omega _2$, i.~e.~ at the  points $\rw (z_j)= 2s\omega _1+2m\omega _2$, $1\leq j\leq n$, $m,s\in \Z$. Being periodic in each variable with the periods $2\omega _1$ and $2\omega _2$ we conclude that the function~\eqref{kr-rig} has  no poles.    

By  the Liouville theorem~\cite{Rudin} we obtain that the function~\eqref{kr-rig} has to be a constant, what together with Proposition~\ref{prop-rigidity} proves that the  Krichever genus is $T^k$-rigid on $G/H$.
\end{proof}

\subsection{Proof based on the representation theory}
The fact that the function~\eqref{Kr-homogeneous} has no poles for a $SU$-structure can be proved directly using representation theory. For that purpose we first  prove the results which ensure that the zero poles of the function~\eqref{Kr-homogeneous}  cancel in pairs for an arbitrary invariant almost complex structure. 
More precisely, for any fixed point and for any weight at that fixed point we prove the existence of an another fixed point containing this weight with an opposite sign such that the poles of the function~\eqref{p-function} at the given weight  at these two fixed points cancel. For the background on the root theory of Lie algebras  we refer to~\cite{VO}.

\begin{prop}\label{weights-pairs}
If there exist at least two canonical coordinates of the group $G$ which are not canonical coordinates of the semisimple part of the subgroup $H$ then there exists complementary root $\alpha$ for $G$ related to $H$ such that for an arbitrary invariant almost complex structure $J$ on $G/H$ the fixed points of the canonical action of the maximal torus can be divided into the pairs $(\rw ,\tilde{\rw })$ such that $-\rw (\alpha)$ is the weight at $\tilde{\rw }$ and the weights at  $\rw$ and $\tilde {\rw}$ differ only by signs and the number of opposite signs is odd. Moreover apart from $\rw (\alpha)$  all the other weights at $\rw$ that have the opposite  signs at $\tilde{\rw }$ can be divided into the pairs whose sum is multiple of $\rw (\alpha)$.  
\end{prop}
\begin{proof}
Without loss of generality we may assume that $G$ is a simple compact Lie group, as any compact connected Lie group can be decomposed into a locally direct product of connected simple normal subgroups~\cite{On}. Let $x_1$ and $x_2$ be the canonical coordinates for $G$ which are not the coordinates for the  semisimple part of $H$. For an arbitrary fixed point $\rw$ we consider the fixed point $\tilde{\rw}$ which we obtain by transposition of $x_1$ and $x_2$. Note that such transposition  belongs to $W_{G}/W_{H}$ for any $G$ and $H$.  On the other hand $\alpha = x_1-x_2$ will always be complementary root for $G$ related to $H$ and thus, $\pm (\rw (x_1)-\rw (x_2))$ will be the weight for $J$ at $\rw$, what implies that $\pm (\rw (x_2)-\rw (x_1))$ will be the weight for $J$ at $\tilde{\rw }$. The other weights for $J$ at $\rw$ either do not contain $\rw (x_1)$, $\rw (x_2)$ and on them this permutation  does not reflect, either are of the form $\pm\rw (x_l)\pm \rw (x_k)$, $\pm (2)\rw (x_l)$, $\rw (x_l)+\rw (x_k)+\rw (x_i)\pm \rw (x)$,  $\rw (x_l)+\rw (x_k)+\rw (x_i)+\rw (x_j)$, $\pm(\rw (x_i)+\rw (x_j)+\rw (x_l))$, $\frac{1}{2}(\pm x_{1}\pm x_2\pm x_3\pm x_4)$ depending of the type of the group $G$, where $l=1,2$. We see then that under the action of  the transposition of $\rw (x_1)$ and $\rw (x_2)$ some of the remaining  weights do not change while the other ones can be divided into the pairs (of the same type, each of them containing $\rw (x_1)$ either $\rw (x_2)$) such that at each pair we have that each weight maps to the other or each weight maps to the opposite of the other. Therefore the weights at $\rw$ and $\tilde {\rw}$ differ by odd number of signs. It will also imply that the sum of two weights that change the signs and belong to the same pair will be multiply of the weight $\pm (\rw (x_2)-\rw (x_1))$.
\end{proof}       

Under the assumption of Proposition~\ref{weights-pairs} we deduce further the following.

\begin{cor}
If $J$ is an invariant almost complex $SU$-structure then the number of opposite signs at each pair of fixed points is 
at least three.
\end{cor}
\begin{proof}
Since the sum of weights at each fixed point is zero, we may not have two fixed points at which the weights differ only in one sign.
\end{proof}    

\begin{rem}
By examining  the Dynkin diagrams of the simple compact Lie groups we see that if the group $H$ is the centralizer of the torus $T^k$ such that  $k\geq 3$, or $H=T^2\times H^{'}$, the assumption of Proposition~\ref{weights-pairs} will always be satisfied. 
\end{rem} 
We are able to prove the general statement:
\begin{prop}\label{weights-arbitrary}
If $\alpha$ is a complementary root for $G$ related to $H$  then for an arbitrary invariant almost complex structure $J$ on $G/H$ the fixed points of the canonical action
of the maximal torus can be divided into the pairs $(\rw ,\tilde{\rw })$ such that $-\rw (\alpha)$ is the weight at $\tilde{\rw}$ and each weight at $\tilde{\rw}$ is either,  up to sign, the weight at $\rw$ either it is the sum of some weights at $\rw$ and the multiple of the  weight $\rw (\alpha)$. Moreover apart from $\rw (\alpha)$ all  other weights at $\rw$ that have the opposite  signs at $\tilde{\rw }$ can be divided into the pairs whose sum is multiple of $\alpha$.
\end{prop}
\begin{proof}
Again we assume that $G$ is a simple Lie group. We have that $\pm \alpha$ is the root for $J$. Consider the element $t$ of the Weyl group $W_{G}$ given by the  reflection related to the root $\alpha$.  As $\alpha$ is a complementary root this reflection belongs to $W_{G}/W_{H}$. For the simplicity take $\rw =e$ and let $\tilde {e}$ be the fixed point which corresponds to the reflection $t$. The reflection $t$ maps the roots for $J$ into the weights at $\tilde{e}$. We provide here the proof for $\alpha = x_1-x_2$, where $x_1$ and $x_2$ are canonical coordinates on $G$ since it is the common root for all simple Lie groups. For the other possibilities for $\alpha$ the proof goes analogously. Here the reflection related to $x_1-x_2$ is transposition $t$ that interchanges $x_1$ and $x_2$.  We differentiate the following cases. If the root $\alpha _j$ does not contain none of $x_1$ and $x_2$, or it contains both of them with the same sign then $t(\alpha _j)=\alpha _{j}$. It for some $\alpha _j$ we have that $t(\alpha _j)=\pm \alpha _i$, where $i\neq j$,  then we  also have $t(\alpha _i) = \pm \alpha _j$. If $\alpha _j$ does not belong to the previous cases it implies that $\alpha _{j}$ contains $x_1$ or $x_2$ or both of them with different signs and  $t(\alpha _j)$ is the root for $H$. Therefore we have for $\alpha _{j}$ the following possibilities: 
$\pm x_l \pm x_i$,  $\pm (2)x_l$, $x_l+x_i+x_j\pm x$, $x_l+x_i+x_j+x_k$, $\pm (x_l+x_j+x_k)$, $\frac{1}{2}(\pm (x_1-x_2)\pm x_3\pm x_4)$ depending of the type of the group $G$, where $l=1,2$. We see that in all these cases $\alpha _{j}-t(\alpha _j)$ is the multiple of the root $\pm (x_1-x_2)$.  Arguing as in the proof of Proposition~\ref{weights-pairs} we prove the second statement.
\end{proof}      

\begin{thm}\label{Kr-no-zero-poles}
If $f$ is the power series which defines the Krichever genus, then for any homogeneous space $G/H$  the function 
\begin{equation}\label{no-eq}
\sum_{\rw \in W_{G}/W_{H}}\prod_{j=1}^{n}\frac{1}{f(\rw (\alpha_{j}\cdot u))} \ ,
\end{equation}
has no zero poles, where $\alpha _j$ are the roots of an arbitrary invariant almost complex structure on $G/H$.
\end{thm}
\begin{proof}
The function~\eqref{no-eq} may have zero poles at the points $\rw (z_i) =0$,
where $z_i$ are the roots for the fixed invariant almost complex structure $J$ and $\rw \in W_{G}/W_{H}$.  Let us fix some root $z_1$  for $J$. We first consider the case when $z_1$ comes from Proposition~\ref{weights-pairs}.  Consider the weight $\rw (z_1)$ at the fixed point $\rw$. 
According to Proposition~\ref{weights-pairs} there exists fixed point $\tilde{\rw}$ such that $-\rw (z_1)$ is the weight at $\tilde{\rw}$,  the weights  at $\rw$ and $\tilde{\rw }$  may differ only in signs and the number of opposite signs is  odd.  Choose an arbitrary linear basis $\rw (z_1),\ldots ,\rw (z_k)$ for the weights at $\rw$ which contains $\rw (z_1)$. Put $y_i=\rw (z_i)$, $1\leq i\leq k$, and consider the function
\begin{equation}\label{kr-poles}
q(y_1,\ldots, y_k) = \hat{p}(y_1,\ldots, y_k) + \hat{p}(\epsilon _1 y_1,\epsilon _2 y_2,....,\epsilon _k y_k)=
\end{equation}
\[ 
p(y_1,\ldots ,y_k)e^{(\zeta (v)-\mu)\sum\limits_{i=1}^{k}(1+\sum\limits_{j=k+1}^{n}c_{i}^jy_i)}+
p(\epsilon _{1}y_1,\ldots ,\epsilon _{k}y_{k})e^{(\zeta (v)-\mu)\sum\limits_{i=1}^{k}(\epsilon _{i}+\sum\limits_{j=k+1}^{n}\epsilon_{j}c_{i}^jy_i)}.
\]    
Here $\epsilon _i = \pm 1$ denote the signs of $y_i$ at the fixed point $\tilde{\rw }$, then $\epsilon _1 = -1$  and   $\prod _{i=2 }^{n}\epsilon _{i}=1$, and the function  $p(y_1,\ldots ,y_k)$ is given with~\eqref{p-function}, where $y_{j}=\sum\limits_{i=1}^{k}c_{i}^jy_k$, $k+1\leq j\leq n$.  

According to Proposition~\ref{weights-pairs} we can divide the indexes $l$, $2\leq l\leq n$,  with respect to the  weights at $\rw $ into three groups as follows: 
\[
I=\{l\; |\;  ({\tilde{\rw}}\circ \rw ^{-1}) (y_l) = y_l \},\;\; II=\{ \{{l,s\}}\; |\; ({\tilde{\rw}}\circ \rw ^{-1}) (y_l) = y_s,\; ({\tilde{\rw}}\circ \rw ^{-1}) (y_s) =y_l \},
\]
\begin{equation}\label{division}
III = \{ \{ l,s\}\; |\; ({\tilde{\rw}}\circ \rw ^{-1}) (y_l)=-y_s,\; ({\tilde{\rw}}\circ \rw ^{-1}) (y_s) = -y_l \}.
\end{equation}

It also follows from Proposition~\ref{weights-pairs} that if the pair of weights change the sign or in the other words if $\{ l, s\}\in III$ then $y_l+y_s=k_{ls}y_1$, where $k_{ls}\in \Z$. It implies that the sum of weights at the fixed point $\rw$ and the sum of weights at the fixed point $\tilde{\rw}$ are given with 
\[
\rho (y_2,\ldots y_k)+y_1+\sum\limits_{\{l,s\}\in III}k_{ls}y_1\; \text{and}\; \rho (y_2,\ldots y_n)-y_1-\sum\limits_{\{l,s\}\in III}k_{ls}y_1, 
\]  
where $\rho (y_2,\ldots y_n)=\sum_{l\in I, II}y_l$. It further gives  that the function~\eqref{kr-poles}
can be written as
\[
q(y_1,\ldots, y_k) = p(y_1,\ldots ,y_k)e^{(\zeta (v)-\mu)(\rho (y_2,\ldots ,y_k)+y_1+\sum\limits_{\{l,s\}\in III}k_{ls}y_1)}
+\]
\[p(\epsilon _{1}y_1,\ldots ,\epsilon _{k}y_{k})e^{(\zeta (v)-\mu)(\rho (y_2,\ldots ,y_k)-y_1-\sum\limits_{\{l,s\}\in III}k_{ls}y_1)}.
\]

We want to prove that $q(y_1,\ldots ,y_k)$  has no pole $y_1=0$. Consider the function
\[
y_1q(y_1,\ldots ,y_k)|_{y_1=0} = \frac{y_1}{\sigma (y_1)}|_{y_1=0}e^{(\zeta (v)-\mu)\rho (y_2,\ldots ,y_k)}(p(y_2,\ldots ,y_{k})-p(\epsilon _{2}y_2\ldots ,\epsilon _{k}y_k)).
\]
In order to prove that $y_1=0$ is not a pole we will prove that the function  $p(y_2,\ldots ,y_{k})- p(\epsilon _{2}y_2,\ldots ,\epsilon _{k}y_k)$ is the zero function.
Using~\eqref{division}, we can write further  this function as
\begin{equation}\label{zero}
\prod\limits_{l\in I,II}\frac{\sigma (v-y_l)}{\sigma (v)\sigma (l)}\prod\limits_{\{ l,s \}\in III}(\frac{\sigma (v-y_l)\sigma (v-y_s)}{(\sigma (v))^2\sigma (y_l)\sigma (y_s)} - \frac{\sigma (v+y_l)\sigma (v+y_s)}{(\sigma (v))^2\sigma (y_l)\sigma (y_s)}).
\end{equation} 
Since for $\{ l,s\}\in III$ we have that $y_l+y_s$ is multiple of $y_1$, it follows that $y_s=-y_l$ when putting $y_1=0$.
Therefore we conclude that
\[
\frac{\sigma (v-y_l)\sigma (v-y_s)}{(\sigma (v))^2\sigma (y_l)\sigma (y_s)} - \frac{\sigma (v+y_l)\sigma (v+y_s)}{(\sigma (v))^2\sigma (y_l)\sigma (y_s)}  = 0,
\]
what implies that~\eqref{zero} is the zero function.

If $z_1$ does not satisfy Proposition~\ref{weights-pairs} then for the fixed point $\rw $ and the weight $y_1=\rw (z_1)$ we can choose the fixed point $\tilde{\rw}$ as in Proposition~\ref{weights-arbitrary}. Then, excluding $y_1$,  we can again  divide the weights at $\rw $ into four groups 
\[
\tilde{I}=\{l\; |\;  ({\tilde{\rw}}\circ \rw ^{-1}) (y_l) = y_l \},\;\; \tilde{II}=\{ \{{l,s\}}\; |\; ({\tilde{\rw}}\circ \rw ^{-1})(y_l) = y_s,\; ({\tilde{\rw}}\circ \rw ^{-1})(y_s) =y_l \},
\] 
\[ \tilde{III} = \{ \{ l,s\}\; |\; ({\tilde{\rw}}\circ \rw ^{-1})(y_l)=-y_s,\; ({\tilde{\rw}}\circ \rw ^{-1})(y_s) = -y_l \},
\]
\[
\tilde{IV} = \{ l \; | ({\tilde{\rw}}\circ \rw ^{-1})(y_l)=y_l+d_ly_1,\; d_l\in \Z \}. 
\]
Now the sum of weights at the fixed point $\rw$ and the sum of weights at the fixed point $\tilde{\rw}$ are given by 
\[
\rho (y_2,\ldots y_k)+y_1+\sum\limits_{\{l,s\}\in \tilde{III}}k_{ls}y_1\; \text{and}\; \rho (y_2,\ldots y_n)+\sum\limits_{l\in \tilde{IV}}d_{l}y_1-y_1-\sum\limits_{\{l,s\}\in \tilde{III}}k_{ls}y_1, 
\]  
where $\rho (y_2,\ldots y_n)=\sum_{l\in \tilde{I}, \tilde{II}, \tilde{IV}}y_l$.
Repeating the procedure from the previous case we end up with the analogous of function~\eqref{zero} given by
\[
\prod\limits_{l\in \tilde{I},\tilde{II}}\frac{\sigma (v-y_l)}{\sigma (v)\sigma (l)}\prod\limits_{\{ l,s \}\in \tilde{III}}\frac{\sigma (v-y_l)\sigma (v-y_s)}{(\sigma (v))^2\sigma (y_l)\sigma (y_s)}\prod\limits_{l\in \tilde{IV}}\frac{\sigma (v-y_l)}{\sigma (v)\sigma (y_l)} - \]
\[
\prod\limits_{l\in \tilde{I},\tilde{II}}\frac{\sigma (v-y_l)}{\sigma (v)\sigma (l)}\prod\limits_{\{ l,s \}\in \tilde{III}}\frac{\sigma (v+y_l)\sigma (v+y_s)}{(\sigma (v))^2\sigma (y_l)\sigma (y_s)})\prod\limits_{l\in \tilde{IV}}\frac{\sigma (v-y_l-d_ly_1)}{\sigma (v)\sigma (y_l+d_ly_1)}.
\]
After putting $y_1=0$ we deduce as above that this function has to be the zero function.
\end{proof}
\begin{thm}\label{SUKr-no-arbitrary-poles}
If $f$ is the power series which defines the Krichever genus, then for any homogeneous space $G/H$  the function 
\begin{equation}\label{no-eq2}
\sum_{\rw \in W_{G}/W_{H}}\prod_{j=1}^{n}\frac{1}{f(\rw (\alpha_{j}\cdot u))} \ ,
\end{equation}
has no poles, where $\alpha _j$ are the roots of an arbitrary invariant almost complex $SU$-structure on $G/H$.
\end{thm}
\begin{proof}
We can proceed as in the proof of Theorem~\ref{Krichever-rigidity} and prove that the function~\eqref{no-eq2} is periodic for a $SU$-structure what together with Theorem~\ref{Kr-no-zero-poles} gives that it has no poles. We provide here the direct proof not appealing to the periodicity but following the proof of  Theorem~\ref{Kr-no-zero-poles}. The function~\eqref{no-eq2} may have poles at the points $\rw (z_i) =2s\omega _1+2m\omega _2$. 
where $z_i$ are the roots for the fixed invariant almost complex $SU$-structure $J$ and $\rw \in W_{G}/W_{H}$. Recall that by Remark~\ref{pointing}, multiplicity of each pole is $1$. We prove here that the function~\eqref{no-eq2} does  not have a pole at $\rw (z_i) =2\omega _1$. The general case goes analogously.  Fix the weight $\rw (z_1)$ and consider first consider the case when it comes from Proposition~\ref{weights-pairs}. Since $J$ is a $SU$-structure it follows that $\sum_{i=1}^{n} \rw (z_i)=0$ for any $\rw \in W_{G}/W_{H}$ what implies  that the function~\eqref{kr-poles} is  given by
\[
q(y_1,\ldots ,y_k) = p(y_1,\ldots ,y_k)+p(\epsilon _1y_1,\ldots ,\epsilon _ky_k).
\]
We want to prove that this function has no pole at $y_1=2\omega _1$ and therefore we consider the function
\[
\frac{1-\frac{y_1}{2\omega _1}}{\sigma (y_1)}\sigma (y_1)(p(y_1,\ldots ,y_k)-p(y_1,\epsilon _2y_2,\ldots \epsilon _ky_k))|_{y_1=2\omega _1}
\]
and prove that its non-fractal part equals  zero. Now we have that
\begin{equation}\label{step}
\sigma (2\omega _1)p(2\omega _1,y_2,\ldots ,y_k) = \frac{\sigma (v-2\omega _1)}{\sigma (v)}\prod\limits_{l\in I,II}\frac{\sigma (v-y_l)}{\sigma (v)\sigma (l)}\prod\limits_{\{ l,s \}\in III}\frac{\sigma (v-y_l)\sigma (v-y_s)}{(\sigma (v))^2\sigma (y_l)\sigma (y_s)}.
\end{equation}
Since for $\{l ,s\}\in III$ we have that $y_l+y_s=2k_{ls}\omega _1$ when $y_1=2\omega _1$, it further implies that~\eqref{step} is equal to 
\[ 
e^{-2\eta _1(1 +\sum\limits_{\{l,s\}\in III}k_{ls})v}\prod\limits_{l\in I,II}\frac{\sigma (v-y_l)}{\sigma (v)\sigma (l)}\prod\limits_{\{ l,s \}\in III}\frac{\sigma (v-y_l)\sigma (v+y_l)}{(\sigma (v))^2(\sigma (y_l))^{2}}.
\]
In the same way we obtain that $\sigma (2\omega _1)p(-2\omega _1,\epsilon _2 y_2,\ldots ,\epsilon _k y_k)$ is equal to
\[
e^{2\eta _1(1 +\sum\limits_{\{l,s\}\in III}k_{ls})v}\prod\limits_{l\in I,II}\frac{\sigma (v-y_l)}{\sigma (v)\sigma (l)}\prod\limits_{\{ l,s \}\in III}\frac{\sigma (v+y_l)\sigma (v-y_l)}{(\sigma (v))^2(\sigma (y_l))^{2}}.
\]
The assumption on the structure $J$ to be a $SU$-structure implies that  $1+\sum _{\{l,s\}\in III}k_{ls}=0$, what proves that $y_1=2\omega _1$ is not the pole.  
In the case when $\rw (z_1)$ does not come from Proposition~\ref{weights-pairs} we  apply Proposition~\ref{weights-arbitrary} and follow the same pattern. 
\end{proof}

\begin{ex}
Theorem~\ref{Krichever-rigidity} is not true without  assumption on invariant almost complex structure to be a $SU$-structure. In order to see that, let us consider $U(3)/T^3$ with the canonical complex structure and the canonical action of the torus $T^2$. This action has six fixed points with the weights $\rw (\alpha _1),\rw (\alpha _2), \rw (\alpha _3)$, where $\alpha _1 = (1,-1,0), \alpha _{2}=(1,0,-1), \alpha _{3}=(0,1,-1)$ and $\rw \in S_3$. Denote $\rw(\alpha _i \cdot u)= \rw ({\bar u_i})$, for $1\leq i\leq 3$.  Then the expression~\eqref{rigidity} becomes
\begin{equation}\label{Krichever-CP}
\sum_{\rw \in S_3}\frac{B(\rw ({\bar u_1}),v)B(\rw ({\bar u_2}),v)B(\rw ({\bar u_3}),v)}{e^{\mu ( \rw ({\bar u_1}) +\rw ({\bar u_2})+\rw ({\bar u_3}))}},
\end{equation}
being for ${\bar u_1}=  {\bar u_3} =v$ equal to
\[
\psi (v)=-\frac{\sigma (2v)\sigma (3v) e^{(\zeta (v)-\mu )4v}}{(\sigma (v))^5},
\]
while for ${\bar u_1}= {\bar u_2}= v+2\omega _k$ it is equal to
\[
\psi (v)\cdot e^{(2\eta _{k}(-6v+3\omega _{k})-6\omega _{k}(\mu -\zeta (v)))}.
\]
Therefore, it follows that~\eqref{Krichever-CP} can not be a constant what implies that the Krichever genus is not $T^2$-rigid for the canonical complex structure on $U(3)/T^3$. Note that at the same time $U(3)/T^3$ admits an invariant almost complex $SU$-structure for which  the Krichever genus will be $T^2$-rigid.
\end{ex}
\begin{ex}
We demonstrate in the case of  $\C P^2$ how the zero poles in the function from Theorem~\ref{Kr-no-zero-poles} cancel.
We take the fixed point $123$. The zero pole at $123$ given by the weight $x_1-x_2$  is canceled with the zero pole given by the weight $x_2-x_1$ at $213$. The zero pole at $123$ given by $x_1-x_3$ may not be canceled using the point $213$ as it is not the weight at this point. It is  canceled with the zero pole given by the weight $x_3-x_1$ at $321$. Now the zero  pole at $213$ given by $x_2-x_3$ is canceled with the zero  pole at $321$ given by $x_3-x_1$. 
\end{ex}  
\begin{ex}
Theorem~\ref{SUKr-no-arbitrary-poles} is not true without assumption on the structure $J$ to be a $SU$-structure.  We show  it by proving that the pole for $\C P^{2}$ in the function~\eqref{no-eq2} at the point $x_1-x_2=2\omega _1$ does not cancel. It  is the pole at the fixed points $123$ and $213$. Therefore we put $z_1=x_1-x_2$, $z_2=x_1-x_3$ and consider the function in $z_2$ given with
\[e^{(\zeta (v)-\mu)(2\omega_1+z_2)}\frac{\sigma (v-2\omega_1)\sigma (v-z_2)}{\sigma (z_2)(\sigma (v))^2}-e^{(\zeta (v)-\mu)(-4\omega _1+z_2)}\frac{\sigma (v+2\omega _1)\sigma (v-z_2+2\omega_1)}{(\sigma (v))^2\sigma (z_2-2\omega _1)},
\]
which transforms to the function
\[
\frac{\sigma (v-z_2)}{\sigma (v)\sigma (z_2)}(-e^{((\zeta (v)-\mu)(2\omega _1+z_2)+2\eta_1(-v+\omega_1))} + e^{((\zeta (v)-\mu)(-4\omega _1+z_2)+2\eta_1(2v+\omega _1))})
\]
not always  being identically equal to zero.
\end{ex}
\begin{rem}
Let  $S^1\subseteq T^k$ be a regular subgroup of the maximal torus $T^k$ of a homogeneous space $G/H$ equipped with an invariant almost complex structure $J$. Recall~\cite{Adams} that regularity of $S^1$ means that $T^k$ is the unique maximal torus in $G$ that contains $S^1$. We can  consider  $S^1$-action  on $G/H$ induced from the canonical action of $T^k$.  The weights for this action and  each fixed point $\rw \in W_{G}/W_{H}$  are given by its rotation numbers $\rho _{i}(\rw )$, see proof of Theorem~\ref{TH} below.  If  $J$ is a $SU$-structure it will follow that $\sum \rho _{i}(\rw )=0$ for any $\rw \in W_{G}/W_{H}$. It means that any such $S^1$ action will be of the zero type as it is defined in~\cite{KR}.
\end{rem}
\begin{rem}
It is proved in~\cite{KR} that the Krichever genus vanishes on $SU$-spaces with the given $S^1$-actions whose type is not zero. We illustrate here that the assumption on the type is essential. For $f$ given by~\eqref{KR-defn},  it follows from Example~\ref{S6} that 
\[
\LLL_{f}^{T^k}(S^6, J) = \frac{\sigma (v+u_1)\sigma (v+u_2)\sigma (v-u_1-u_2)-\sigma (v-u_1)\sigma (v-u_2)\sigma (v+u_1+u_2)}{\sigma (u_1)\sigma (u_2)\sigma(u_1+u_2)}
\]
which checks directly not to be equal to zero. Therefore  Theorem~\ref{Kr-no-zero-poles} gives  that 
$\LLL _{f}(S^6, J)=\LLL _{f}^{T^2}(S^6, J)\neq 0$.  
\end{rem}   
 
\subsection{Rigidity of the elliptic genus of level $N$}
We define the elliptic genus of level $N$ following~\cite{KR}.  On an elliptic curve $\Gamma$ fix the points $v_{mn}$ of order $N$:
\[
v_{sm}=\frac{2s}{N}\omega _{1}+\frac{2m}{N}\omega _{2},\;\; s,m=0,1,\ldots ,N-1,
\]
and for $\eta _{l} = \zeta (\omega _{l})$, $l=1,2$, put 
\[
\mu _{sm} = -\frac{2s}{N}\eta _{1} - \frac{2m}{N}\eta _{2} + \zeta (v_{sm}).
\]
\begin{defn}
The elliptic genus of level $N$ is a Hirzebruch genus defined by the series
\begin{equation}\label{elliptic-level}
f_{sm}(u)= \frac{exp(\mu _{sm}u)}{B(u,v_{sm})}.
\end{equation}
\end{defn}   
\begin{rem}
As it is pointed in~\cite{KR} function $f_{sm}$ generates the elliptic genus of level $N$ defined in~\cite{HLN}.  
\end{rem}
It is proved in~\cite{HLN} that the elliptic genus of level $N$ is $S^1$-rigid on $S^1$-equivariant (stable) complex manifold  whose first Chern class is divisible by $N$. 
\begin{thm}
The elliptic genus of level $N$ is $T^k$-rigid on homogeneous spaces of positive Euler characteristic endowed with the canonical action of the maximal torus $T^k$ and with an  invariant almost complex structure whose sum of roots is divisible by $N$.
\end{thm}
\begin{proof}
Let $J$ be an invariant almost complex structure on $G/H$ with  roots $\alpha _{j}$, $1\leq j\leq n$. Then
$c_{1}(G/H,J)=\sum\limits_{j=1}^{n}\alpha _j$. We want to apply Proposition~\ref{prop-rigidity}, thus  put $z_{j}=\alpha _{j}\cdot u$ and  analogously to the proof of Theorem~\ref{Krichever-rigidity} consider the function  
\begin{equation}\label{levelN}
\sum _{\rw \in W_{G}/W_{H}}e^{(\zeta (v_{sm})-\mu _{sm})\rw (c_{1}(G/H, J))}
\prod_{j=1}^{k}\frac{\sigma (v_{sm}-\rw (z_j))}{\sigma (\rw (z_j))\sigma (v_{sm})}\prod_{j=k+1}^{n}
\frac{\sigma (v_{sm}-\sum _{i=1}^{k}c_{i}^{j}\rw (z_i))}{\sigma (\sum_{i=1}^{k}c_{i}^{j}\rw (z_i))\sigma (v)}.
\end{equation}
We put $y_i=\rw (z_i)$, $1\leq i\leq k$ and prove that the function 
\[
{\hat p}(y_1,\ldots ,y_k) = e^{(\zeta (v_{sm})-\mu _{sm})c_{1}(G/H, J)}
p(y_1,\ldots ,y_k)
\]
is two-periodic in each variable, where 
\[
p(y_1,\ldots ,y_k) =\prod_{i=1}^{k} \frac{\sigma (v_{sm}-y_i)}{\sigma (y_i)\sigma (v_{sm})}\prod_{j=k+1}^{n}
\frac{\sigma (v_{sm}-\sum _{i=1}^{k}c_{i}^{j}y_i)}{\sigma (\sum_{i=1}^{k}c_{i}^{j}y_i)\sigma (v_{sm})}.
\]  
If we fix $y_1$ we obtain as in the proof of Theorem~\ref{Krichever-rigidity} that $p(y_1+2\omega _{l},\ldots ,y_k) = e^{-2\eta_{l}(1+\sum_{j=k+1}^{n}c_{1}^{j})v_{sm}}p(y_1,\ldots,y_k)$, what gives that
\[
{\hat p}(y_1+2\omega _1,\ldots ,y_k) = e^{4c_{11}m(\omega _1\eta _2 -\eta _1\omega _2)}{\hat p}(y_1,\ldots ,y_k),
\]
where $c_{11}=\frac{1}{N}(1+\sum\limits_{j=k+1}^{n}c_1^j)$ is an integer number as the first Chern class here is divisible by $N$. 
Since $\eta _2=\frac{\omega _2}{\omega _1}\eta _1 - \frac{\pi i}{2\omega _1}$ we have that $\eta _{2}\omega _1-\eta _1 \omega _2 = -\frac{\pi i}{2}$ what implies that ${\hat p}$ is periodic with the period $2\omega _1$.  
In the same way we prove that it is periodic with the period $2\omega _2$, what implies that the function~\eqref{levelN} is two-periodic. Since  no equivariant Hirzebruch genus has zero pole,  we have that the function~\eqref{levelN} has no zero poles as well. By Remark~\ref{pointing} all the poles for function~\eqref{levelN} are of multpilicity $1$. Therefore it only  may have the other poles at the points of the lattice $2s\omega _{1}+2m\omega _{2}$ for $s,m\in \Z$, what is impossible because of  its periodicity. Thus, by the Liouville theorem,  the function~\eqref{levelN} has to be a constant.
\end{proof}
\begin{rem}
In this context one can consider the problem to describe all invariant almost complex structures on a homogeneous space $G/H$ whose first Chern class is divisible by some $N>1$.  This  can be also looked at as the problem of representation theory as any invariant almost complex structure is determined by its roots. The necessary condition  for the existence of a such structure is the existence of the $T^k$-rigid elliptic genus of level $N$ on $G/H$, what is on the other hand topological problem.
\end{rem} 
\begin{ex}
It is proved in~\cite{KT} that generalized flag manifolds $F_n=U(n+2)/(U(1)\times U(1)\times U(n))$ has two invariant almost complex structures $J_1$ and $J_2$ for which the  first Chern class is divisible by $n-1$ and $n+1$ respectively . Therefore the elliptic genus of level $n-1$ will be $T^{n+2}$-rigid on $(F_n, J_1)$ and the elliptic genus of level $n+1$ is $T^{n+2}$-rigid on $(F_n, J_2)$. 
\end{ex}

\subsection{A Hirzebruch genus of an odd series} 
We prove  here that the  Hirzebruch genus $\LLL _{f}$ defined by an odd power series $f$ is $T^k$-rigid and equal to zero on a large class of homogeneous spaces, being a stronger result than $T^k$-rigidity.   Using Proposition~\ref{prop-rigidity} we first note that genus $\LLL _{f}$  is $T^2$-rigid and  trivial  on the sphere $S^6$. 
\begin{ex}\label{rig-S6}
Consider the canonical action of the torus $T^2$ on $S^6$ endowed with $SU(3)$-invariant almost complex structure. It follows from Example~\ref{S6} that  for an odd series $f$ the left hand side of the functional equation~\eqref{rigidity}
is equal to zero, what gives that the genus $\LLL _{f}$ is $T^2$-rigid on $S^6$. 
\end{ex} 
In the case of the flag manifolds the following consequence of Proposition~\ref{weights-pairs} will be crucial.
\begin{prop}\label{pairs}
Let $J$ be an invariant almost complex structure on the flag manifold $U(n)/T^n$. The fixed points for the canonical  action of the maximal torus on $U(n)/T^n$ can be divided into the pairs such that at each pair the weights for this action related to the structure $J$ differ in odd number of  signs.
\end{prop}
\begin{proof}
For the sake of clearness we provide the proof. Let $J$ be an invariant almost complex structure on $U(n)/T^n$. Its roots are  $\epsilon _{ij}\alpha _{ij}$, where $\alpha _{ij}=x_i-x_j$, $1\leq i<j\leq n$ and $\epsilon _{ij}=\pm 1$. The set of fixed points is given by the symmetric group $S_{n}$. Therefore at the fixed point  $\rw =i_1\ldots i_n\in S_n$ the weights $\omega _{j}(\rw)$  for the action of $T^{n}$ related to   $J$ are obtained by the action of the permutation $\rw$ on the roots for $J$:
\[
\epsilon _{12}(x_{i_1}-x_{i_2}),\ldots ,\epsilon _{1n}(x_{i_1}-x_{i_n}),\epsilon _{23}(x_{i_2}-x_{i_3})\ldots ,\epsilon _{2n}(x_{i_2}-x_{i_n}),\ldots ,\epsilon _{n-1n}(x_{i_{n-1}}-x_{i_{n}}).
\]
Thus, the weights at the fixed point $\tilde{\rw}=i_2i_1i_3\ldots i_{n}$ are given with:
\[
\epsilon _{12} (x_{i_2}-x_{i_1}),\ldots,\epsilon _{1n} (x_{i_2}-x_{i_{n}}),\epsilon _{23}(x_{i_1}-x_{i_3}),\ldots \epsilon _{2n} (x_{i_1}-x_{i_{n}}),\ldots ,\epsilon _{n-1n}(x_{i_{n-1}}-x_{i_{n}}).
\]
It follows that the weights at fixed points $\rw$ and $\tilde{\rw}$ differ in the sign for $x_{i_1}-x_{i_2}$, while $\epsilon_{1k}(x_{i_1}-x_{i_k})$ and $\epsilon_{2k}(x_{i_2}-x_{i_k})$ change the sign equally. It implies  that the number of weights with the opposite  signs
is odd.
\end{proof}
\begin{thm}\label{flag-rigidity}
The genus $\LLL _{f}$ is $T^n$-rigid  and equal to zero on  flag manifolds $U(n)/T^n$ endowed with an arbitrary invariant almost complex structure for any odd series $f$.
\end{thm}
\begin{proof}
Proposition~\ref{pairs} implies that, for an odd function $f$ and for each pair $(\rw, \tilde{\rw })$ of the fixed points, we have  
\[
\prod _{j=1}^{m}\frac{1}{f(\omega _{j}(\rw )\cdot u)} + \prod _{j=1}^{m}\frac{1}{f(\omega _{j}(\tilde{\rw})\cdot u)} =0.
\]
Since $J$ is an almost complex structure, all fixed points have sign $+1$, what gives that the rigidity equation~\eqref{rigidity} is satisfied for $c=0$.
\end{proof}
Recall the following well known facts.
\begin{rem}
The classical result of Hirzebruch~\cite{HTM} states that the signature of a $2n$-dimensional manifold which is  defined to be the  signature of the intersection form of  its cohomology algebra for even $n$, while it is equal to zero for odd $n$,  coincides with the $L$-genus of a manifold. This is further defined by an odd  series $f(u)=$ tanh$(u)$.  Note that being invariant of an oriented cobordism class, the signature does not depend on the chosen stable complex structure on manifold, up to the orientation the chosen structure induces. Another important example is the $\hat{A}$-genus which is defined by an odd 
series $f(u) = 2$sinh$(\frac{u}{2})$.
\end{rem}
\begin{rem}
The elliptic genus, first appeared in~\cite{OCH}, is a Hirzebruch genus defined by the requirement that it vanishes on the manifolds of the form $\C P(\xi)$, where $\xi$ is an even-dimensional  complex vector bundle over a closed oriented base. It is completely characterized by the condition that its logarithm $g(x)$ is given by an integral $g(x)=\int_{0}^{x}\frac{dt}{\sqrt{1-2\delta t^2+\epsilon t^4}}$. For $A=\C$ and $\delta ^{2}\neq \epsilon \neq 0$ the inverse function $f(u)=g^{-1}(x)$ is an odd  elliptic function. For $\delta =\epsilon =0$ the degenerate elliptic genus gives  the signature, while for $\delta =\frac{-1}{8}$ and $\epsilon=0$ it gives $\hat{A}$-genus.
\end{rem} 
Then Example~\ref{rig-S6} and Proposition~\ref{flag-rigidity} implies:
\begin{ex}
Both thw $\hat{A}$-genus as well as the elliptic genus  are trivial on $S^6$ endowed with the $T^2$-invariant almost complex structure. 
\end{ex}

\begin{cor}
The signature, the $\hat{A}$-genus and the elliptic genus of the flag manifold $U(n)/T^n$ endowed with an arbitrary invariant almost complex structure are equal to zero.
\end{cor}
\begin{rem}
Note that in the same time  the cobordism class of the flag manifold $U(n)/T^n$ is non-trivial for an arbitrary invariant almost complex structure.
\end{rem} 

\begin{rem}
It follows from Theorem~\ref{flag-rigidity} and Proposition~\ref{flag-even} that flag manifolds $U(n)/T^n$  for even $n$, provide examples of manifolds for which any Hirzebruch genus $\LLL _{f}$ defined by an odd series is $T^n$-rigid and equal to zero related to any invariant almost complex structure. Note that none of these  structures is a $SU$-structure.
\end{rem}

\begin{ex}
The fact that $f$ is an odd power series is essential. Namely, take $f(u)=\frac{u}{1+u^3}$ and $M=U(3)/T^3$. The direct computation gives that 
\[
\sum_{\rw \in S_3}\prod_{j=1}^{3}\frac{1}{f(\omega _{j}(\rw )\cdot u)} = 2\cdot \frac{(u_1-u_2)^3-(u_1-u_3)^3+(u_2-u_3)^3}{(u_1-u_2)(u_1-u_3)(u_2-u_3)},
\]
what is not identically equal to a constant.  
\end{ex}

In the same way as for the flag manifolds we prove:
\begin{thm}\label{large}
Assume that $G/H$ satisfies the assumptions of Proposition~\ref{weights-pairs}. Then for an odd series $f$ the genus $\LLL _{f}$ is $T^k$-rigid  and equal to zero on $G/H$ endowed with an arbitrary invariant almost complex structure.
\end{thm} 
We extend this result to some $k$-symmetric spaces:
\begin{prop}\label{k-symm-rigid}
For an odd $m$ any  Hirzebruch genus $\LLL _{f}$ defined by an odd series $f$ is $T^{km}$-rigid and equal to zero  on the $k$-symmetric space $U(km)/\underbrace{U(m)\times \cdots \times U(m)}_{k}$ endowed with an arbitrary invariant almost complex structure.
\end{prop}
\begin{proof}
The roots for an invariant almost complex structure $J$ are, following the proof of Proposition~\ref{k-SU}, given by
$\epsilon _{ij}R_{ij}$, $1\leq i<j\leq k$.  The fixed point set for the canonical $T^{km}$-action is  $S_{km}/(S_{m})^k$, where $S_{n}$ denotes the symmetric group. We divide the fixed points into the pairs in the following way. To the fixed point
given by the permutation $\rw =i_1\ldots i_mi_{m+1}\ldots i_{(k-1)m}i_{(k-1)m+1}\ldots i_{km}$ we assign the fixed point given by the permutation 
${\tilde \rw} = i_{(k-1)m+1}\ldots i_{km}i_{m+1}\ldots i_{(k-1)m}i_{1}\ldots i_{m}$.  The weights at the fixed points $\rw$ and $\tilde{\rw }$ differ only by signs and those with opposite signs are $\epsilon _{1k}R_{\rw (1)\rw (k)}$  and, some of   $\epsilon _{1j}R_{\rw (1)\rw (j)}$, $2\leq j\leq k-1$ and  $\epsilon _{jk}R_{\rw (j)\rw (k)}$, $2\leq j\leq k-1$.  Note that the roots from $\epsilon _{1j}R_{\rw (1)\rw (j)}$ and $\epsilon _{jk}R_{\rw (j)\rw (k)}$ change signs equally.  It follows from Proposition~\ref{k-SU} that the weights at the fixed points $\rw$ and ${\tilde \rw}$ differ in $m^2+2lm^2$ signs for $0\leq l\leq k-2$. Since $m$ is odd it implies that
\[
\prod _{j=1}^{n}\frac{1}{f(\omega _{j}(\rw )\cdot u)} + \prod _{j=1}^{n}\frac{1}{f(\omega _{j}(\tilde{\rw})\cdot u)} =0 \ ,
\]
where $n=\frac{k(k-1)m^2}{2}$. Thus the equation~\eqref{rigidity} is satisfied.
\end{proof}
\begin{ex}
The statement on rigidity in Proposition~\ref{k-symm-rigid} is not valid for even $m$. To verify that consider the Grassmannian $G_{4,2}$. It has two invariant almost complex structures, $J$ and its conjugate. The roots for $J$ 
are $x_1-x_3,x_1-x_4,x_2-x_3,x_2-x_4$, the action of $T^4$ on $G_{4,2}$ has $6$ fixed points given by the quotient 
$S_4/S_2\times S_2$. For an odd series given with $f(u)=\frac{u}{1+u^2}$ we have that expression~\eqref{rigidity} in the case of $G_{4,2}$ is given by
\[
2\cdot (\frac{(1+(u_1-u_3)^2)(1+(u_1-u_4)^2)(1+(u_2-u_3)^2)(1+(u_2-u_4)^2)}{(u_1-u_3)(u_1-u_4)(u_2-u_3)(u_2-u_4)}+\]   
\[\frac{(1+(u_1-u_2)^2)(1+(u_1-u_4)^2)(1+(u_2-u_3)^2)(1+(u_3-u_4)^2)}{(u_1-u_2)(u_1-u_4)(u_3-u_2)(u_3-u_4)}+\] \[\frac{(1+(u_1-u_2)^2)(1+(u_1-u_3)^2)(1+(u_2-u_4)^2)(1+(u_3-u_4)^2)}{(u_1-u_2)(u_1-u_3)(u_2-u_4)(u_3-u_4)}).
\]
This can not be a constant since the direct computation shows that for $u_1=3, u_2=2, u_3=1, u_4=0$ it takes the value $80$ while for $u_1=4, u_2=2, u_3=1, u_4=0$ it takes the value $140$. 
\end{ex}
The relations between the weights and the signs for two different stable complex structures, equivariant for the same torus action, are described in~\cite{Buch_Terz}.
As a consequence we are able to establish the relation between the corresponding values of an equivariant Hirzebruch genus given by an odd power series.

\begin{thm}\label{rig-any}
Assume that manifold $M$ with the given action $\theta$ of the torus $T^k$ admits $\theta$-equivariant stable complex structures $c_{\tau}$ and  $c^{'}_{\tau}$. Then
\begin{equation}
\LLL _{f}^{T^k}(\Phi(M,\theta, c_{\tau}))=\LLL _{f}^{T^k}(\Phi (M,\theta ,c^{'}_{\tau})),
\end{equation}
for any odd power series $f$.  
\end{thm}  
\begin{proof}
Assume that the weights for the action $\theta$ related to the structure $c_{\tau}$  at the point $x\in Fix(M)$ are given by the integer vector  $\Lambda _{i}(x)$,
$1\leq i\leq n$, $\dim M=2n$. It follows from~\cite{Buch_Terz} that the weights and the signs for $\theta$ at $x\in Fix(M)$  related to $c_{\tau}^{'}$ are $\Lambda_{i}^{'}(x)=\epsilon _{i}(x)\Lambda _{i}(x)$ and $\sg (x)^{'}=\epsilon \cdot \prod\limits_{i=1}^{n}\epsilon _{i}(x)\cdot \sg (x)$ where $\epsilon, \epsilon _{i}(x)=\pm 1$ and $\sg (x)$ is the sign at $x$ related to $c_{\tau}$.  Therefore, since $f$ is an odd power series we obtain from Theorem~\ref{main} that:
\begin{align*}
&\LLL _{f}^{T^k}(\Phi (M, \theta ,c^{'}_{\tau})) = \epsilon \cdot \sum\limits_{x\in Fix (M)}\prod\limits_{i=1}^{n}\epsilon _{i}(x)\cdot \sg (x)\prod _{i=1}^{n}\frac{1}{f(\epsilon _{i}(x)\Lambda _{i}(x)\cdot u)}=\\
&\epsilon \sum\limits_{x\in Fix (M)}\prod\limits_{i=1}^{n}\epsilon _{i}(x)\cdot \sg (x)\prod _{i=1}^{n}\epsilon _{i}(x)\frac{1}{f(\Lambda _{i}(x)\cdot u)}=\\
&\epsilon \sum\limits_{x\in Fix (M)}\prod\limits_{i=1}^{n}\epsilon _{i}(x)\cdot \sg (x)\prod _{i=1}^{n}\epsilon _{i}(x)\prod_{i=1}^{n}\frac{1}{f(\Lambda _{i}(x)\cdot u)}=\\
&\epsilon \sum\limits_{x\in Fix (M)}\sg (x)\prod _{i=1}^{n}\frac{1}{f(\Lambda _{i}(x)\cdot u)} = \LLL _{f}^{T^k}(\Phi (M,\theta ,c_{\tau})).
\end{align*}
\end{proof} 
\begin{cor}
The Hirzebruch genus $\LLL _{f}$ given by an odd power series $f$ is  $T^k$-rigid on $(M,\theta ,c_{\tau})$ if and only if it is $T^k$-rigid  on $(M, \theta ,c_{\tau}{'})$, where $c_{\tau}$ and $c^{'}_{\tau}$ are $\theta$ -equivariant stable complex structures.
\end{cor}
\begin{cor}
Any Hirzebruch genus $\LLL _{f}$ defined by an odd series $f$ is
\begin{itemize}
\item $T^n$-rigid and equal to zero  on $U(n)/T^n$ endowed with an arbitrary $T^n$-equivariant stable complex structure.
\item $T^{km}$-rigid and equal to zero on $U(km)/(U(m))^{k}$ endowed with an arbitrary $T^{km}$-equivariant stable complex structure for odd $m$.
\end{itemize}
\end{cor}   
\begin{cor}
The elliptic genus and the $\hat{A}$-genus are equal to zero on;
\begin{itemize}
\item $U(n)/T^n$ related to an arbitrary $T^n$-equivariant stable complex structure.
\item $U(km)/(U(m))^{k}$ related to an arbitrary $T^{km}$-equivariant stable complex structure
for odd $m$.
\end{itemize}
\end{cor}

\section{The Hirzebruch $\chi _{y}$ - genus of homogeneous spaces}
\numberwithin{thm}{section}

The Hirzebruch $\chi _{y}$ - genus is defined~\cite{HBJ} by the series 
\[
f_{y}(u)=\frac{u(1+ye^{-u(1+y)})}{1-e^{-u(1+y)}}.
\]
It is well known that the $\chi _{y}$ - genus for $y=0$ gives the Todd genus, while for $y=1$ it gives the signature.

In this Section we deduce the formula for the $\chi _{y}$ - genus of a homogeneous space $G/H$ of positive Euler characteristic  endowed with an arbitrary stable complex structure $c_{\tau}$ equivariant under the canonical action of the maximal torus assuming that $G/H$ admits an invariant almost complex structure $J$. 
  
Let $\alpha _1,\ldots \alpha _n$ be the roots that define an invariant almost complex structure $J$. 
The weights at the fixed point $\rw$ for the canonical action of the maximal torus $T^k$  related to $c_{\tau}$ are given with $\epsilon_{1}(\rw)\rw(\alpha _1),\ldots , \epsilon_{n}(\rw)\rw(\alpha _n)$, where $\epsilon_{i}(\rw)=\pm 1$ for $1\leq i\leq n$, while the signs are given with $\sg (\rw)=\epsilon \cdot \prod\limits_{i=1}^{n}\epsilon_{i}(\rw)$, where $\epsilon = \pm 1$ depending on whether or not $J$ and $c_{\tau}$ define the same orientation on $\tau (M^{2n})$.  

Let $x_1,\ldots ,x_k$ be the canonical coordinates on the Lie algebra $\TT$ for $T^k$ that correspond to the group $G$ and let $\mu$ be an arbitrary ordering on these coordinates. Denote by $ind _{\mu}(\rw)$ the number of negative roots among
$\epsilon_{1}(\rw)\rw(\alpha _{1}),\ldots ,\epsilon_{n}(\rw)\rw(\alpha _{n})$ related to the ordering $\mu$. Following terminology of~\cite{BP} we  call this number the index of the fixed point $\rw$ related to the stable complex structure $c_{\tau}$ and the ordering $\mu$.

Appealing to the results from~\cite{AH} and the Lie theory we deduce the  formula for the Hirzebruch $\chi _{y}$ - genus of homogeneous spaces under consideration.
\begin{thm}\label{TH}
Let $G/H$ be a compact homogeneous space of positive Euler characteristic endowed with the stable complex structure $c_{\tau}$ that is equivariant under the canonical action of the maximal torus $T^k$. Then the Hirzebruch $\chi _{y}$ - genus for $(G/H, c_{\tau})$ is given by the following formula:
\begin{equation}\label{H}
\chi _{y}(G/H, c_{\tau})=\epsilon\cdot \sum _{\rw\in W_G/W_H}(-y)^{ind_{\mu}\rw}\prod\limits_{i=1}^{n}\epsilon_{i}(\rw).
\end{equation}
\end{thm}
\begin{proof}
We first recall that the Atiyah-Hirzebruch formula~\cite{AH} states that the Hirzebruch $\chi _{y}$ - genus of a stable complex manifold $M^{2n}$ with an action of  $S^1$ can be computed as 
\[
\chi _{y}(M^{2n}) = \sum_{i}(-y)^{n(F_i)}\chi _{y}(F_i),
\]
where the sum goes over all $S^1$-fixed submanifolds $F_i\subset M^{2n}$, and $n(F_i)$ denotes the number of negative weights of the representation for $S^1$ in the normal bundle for $F_i$ in $M^{2n}$. 

Let  $S^1$ be a regular one-parameter subgroup in $T^k$. It is known, see for example~\cite{HS}, that the fixed point set for the canonical action of $S^1$ on $G/H$ coincide with the fixed point set for the canonical action of $T^k$. Therefore all fixed points for the given $S^1$-action  are isolated and given by the elements from $W_G/W_H$. Therefore the fixed points $\rw\in W_G/W_H$ correspond in our case to the manifolds $F_i$ from Atiyah-Hirzebruch formula. It implies that $\chi _{y}(F_i)=\chi _{y}(\rw)=\sg (\rw)$. 

We have  further that $T_{\rw}(G/H)=\gg^{\C} _{a_{1}(\rw)\rw(\alpha _{1})}\oplus\ldots \oplus\gg^{\C}_{a_{n}(\rw)\rw(\alpha _{n})}$, where $\rw\in W_G/W_H$. The inclusion  $S^1\subset T^k$ is given by the vector $v\in \Z^{k}$  and the induced representation  of $S^1$ in $T_{\rw}(G/H)$ is given by the rotation in each root subspace $\gg _{\rw(\alpha _{i})}$ with rotation numbers equal to  $\rho _{i}(\rw)=\left\langle \epsilon_{i}(\rw )\rw(\alpha _{i}), v\right\rangle$,  $1\leq i\leq n$. Note that, since we assume $S^1$ to be regular, the  numbers $\rho _{i}(\rw)$ are non zero for $\rw\in W_{G}/W_{H}$ and $1\leqslant i\leqslant n$. Therefore the number $n(\rw)$ of negative weights of the representation for $S^1$ in $T_{\rw}(G/H)$ is equal to the number of negative rotation numbers $\rho _{i}(\rw)$, $1\leqslant i\leqslant n$.
In this way the Atiyah-Hirzebruch formula gives that
\[
\chi _{y}(G/H, c_{\tau})=\sum _{\rw\in W_G/W_H}(-y)^{n(\rw)}\sg(\rw).
\]
Taking into account the expression for the $\sg (\rw)$ we can write this formula as
\begin{equation}\label{first}
\chi _{y}(G/H, c_{\tau})=\epsilon\cdot\sum _{\rw\in W_G/W_H}(-y)^{n(\rw)}\prod\limits_{i=1}^{n}\epsilon_{i}(\rw).
\end{equation}

Choose an arbitrary Weyl chamber $B$ on the Lie algebra $\TT$ for $T^k$. It is a non empty subset of $\TT$ defined with 
$B=\{ v\in \TT | \left\langle \epsilon _{\alpha}\alpha, v\right\rangle >0 \; \text{for}\; \alpha\in\Sigma\; \text{and}\; \epsilon_{\alpha} =\pm 1\},$ where $\Sigma$ denotes the root system for $G$ related to $T^k$.
The positive roots of the group $G$ related to this Weyl chamber uniquely define an ordering $\mu$ on the canonical coordinates $x_1,\ldots ,x_k$ on $\TT$.  The vice versa is obviously also trough: each ordering $\mu$ on $x_1,\ldots, x_k$ uniquely define the Weyl chamber and the corresponding system of positive roots. 

Let us consider one parameter subgroup $S^{1}\subset T^k$ given by the vector  $v\in B$ such that $v\neq 0$. This subgroup is regular as $\left\langle \epsilon_{\alpha}\alpha, v\right\rangle >0$ for any $\alpha \in \Sigma$ and some $\epsilon_{\alpha}=\pm 1$. It also implies that the number $n(\rw)$ of negative weights $\rho _{i}(\rw)$, $1\leqslant i\leqslant n$ of the representation for $S^1$ in $T_{\rw}(G/H)$ will be in this case equal to the number  of negative roots among $\epsilon_{1}(\rw)\rw(\alpha _1),\ldots ,\epsilon_{n}(\rw)\rw(\alpha _n)$ related to the Weyl chamber $B$. In other words it is the number of negative roots between $\epsilon_{1}(\rw)\rw(\alpha _1),\ldots ,\epsilon_{n}(\rw)\rw(\alpha _n)$ related to the ordering $\mu$ defined by the Weyl chamber $B$. We denote this number further by $\ind_{\mu}(\rw)$.
Using~\eqref{first} we obtain the desired formula for the Hirzebruch $\chi _{y}$ - genus for $(G/H, c_{\tau})$:
\begin{equation}\label{final}
\chi_{y}(G/H, c_{\tau})=\epsilon\cdot \sum _{\rw\in W_G/W_H}(-y)^{ind_{\mu}\rw}\prod\limits_{i=1}^{n}\epsilon_{i}(\rw).
\end{equation}
\end{proof}

\begin{rem} 
Theorem~\ref{TH} proves that the Hirzebruch $\chi _{y}$ - genus of a homogeneous space of consideration can be described only in terms of the representation theory of Lie groups.  We want to note that the motivation for our theorem comes from~\cite{P} where it is proved that the  Hirzebruch $\chi _{y}$ - genus of a quasitoric manifold can be expressed in terms of combinatorial data~\cite{BP} of a manifold. More precisely, in~\cite{P} it is obtained the formula which computes the Hirzebruch $\chi _{y}$ - genus of a quasitoric manifold in terms of signs and indexes of the vertices for the simple polytope, which corresponds to the orbit space of the torus action on a manifold. 
\end{rem}

We can rewrite the expression for $\chi(G/H, c_{\tau})$ in the following way.  Denote by $s_{i}(\rw)$ the sign for  $\rw (\alpha_i)$, $1\leqslant i\leqslant n$ regarding to the ordering $\mu$.  Then the sign for the root $\epsilon_{i}(\rw)\rw(\alpha_i)$ is $\epsilon_{i}(\rw)s_{i}(\rw)$. It gives  that
\[
ind _{\mu}\rw=\frac{1}{2}\sum_{i=1}^{n}(1-\epsilon_{i}(\rw)s_{i}(\rw)),
\]
what implies the following statement.
\begin{prop}
The Hirzebruch $\chi _{y}$ - genus for $(G/H, c_{\tau})$ is given by
\begin{equation}\label{H1}
\chi _{y}(G/H, c_{\tau}) = \epsilon\cdot \sum _{\rw\in W_G/W_H}(-y)^{\frac{1}{2}\sum\limits_{i=1}^{n}(1-\epsilon_{i}(\rw)s_{i}(\rw))}\cdot \prod_{i=1}^{n}\epsilon_{i}(\rw).
\end{equation}
\end{prop}
Using this expression for the  Hirzebruch $\chi _{y}$ - genus we obtain that the Todd genus can be described as follows.
\begin{cor}
The Todd genus for $(G/H, c_{\tau})$ is given by
\begin{equation}\label{T1}
Td(G/H, c_{\tau})=\epsilon \cdot \sum _{\stackrel{\rw\in W_G/W_H}{\epsilon_{i}(\rw)=s_{i}(\rw), 1\leq i\leq n}}\prod _{i=1}^{n}s_{i}(\rw).
\end{equation}
if the set $\{\rw\in W_G/W_H; \epsilon_{i}(\rw)=s_{i}(\rw)\; \text{for all}\; 1\leq i\leq n\}$ is nonempty. 
If this set is empty  then
\[
Td(G/H, c_{\tau}) = 0.
\]
\end{cor}
\begin{proof}
The Hirzebruch $\chi _{y}$ - genus for $y=0$ gives the Todd genus. Therefore the formula~\eqref{H} implies that
\begin{equation}\label{T}
Td(G/H, c_{\tau})=\epsilon \cdot \sum_{\stackrel{\rw\in W_G/W_H}{ind_{\mu}\rw=0}}\prod_{i=1}^{n}\epsilon_{i}(\rw).
\end{equation}
Note that we take into account that an indeterminacy $0^0$ in~\eqref{H} is resolved by setting $0^0=1$.
The description for $ind _{\mu}\rw$ implies that it is equal to zero if and only if
\[
1-\epsilon_{i}(\rw)s_{i}(\rw)=0\; \text{for}\; 1\leq i\leq n,
\]
what is equivalent to
\[
\epsilon_{i}(\rw)=s_{i}(\rw)\; \text{for}\; 1\leq i\leq n.
\] 
This proves the statement.
\end{proof}

The Hirzebruch $\chi _{y}$ - genus for $y=1$ gives the signature or the L-genus for $G/H$.  Then formula~\eqref{H1} gives the following expression for the signature for $G/H$.
\begin{cor}
\begin{equation}\label{S}
sign(G/H) = \epsilon\cdot \sum _{\rw\in W_G/W_H}(-1)^{ind _{\mu}(\rw)}\cdot \prod\limits_{i=1}^{n}\epsilon_{i}(\rw),
\end{equation}
where $\mu (\rw)=\frac{1}{2}\sum_{i=1}^{n}(1-\epsilon_{i}(\rw)s_{i}(\rw))$.
\end{cor}
\begin{rem}
Formula~\eqref{S} can be also deduced from the  general formula for  the signature of a homogeneous space of positive Euler characteristic which is obtained in~\cite{HS}.It can be done applying  the same argument as in the second part of the proof of Theorem~\ref{TH}.
\end{rem}
\subsection{The case of an invariant almost complex structure.}
\numberwithin{thm}{subsection}
For an invariant almost complex structure $J$ we have that $\epsilon=1$ and $\epsilon_{i}(\rw)=1$ for all $1\leq i\leq n$. Therefore ~\eqref{H1} implies  
\begin{cor}
\begin{equation}\label{Hinv}
\chi _{y}(G/H, J)= \sum_{\rw\in W_G/W_H}(-y)^{\frac{1}{2}\sum_{i=1}^{n}(1-s_{i}(\rw))}.
\end{equation}
\end{cor}
Using~\eqref{S} we obtain  the formula for the signature for $G/H$.
\begin{cor}\label{sign_almcom}
\[
sign(G/H)=\sum _{\rw\in W_G/W_H}(-1)^{ind_{\mu}(\rw)},
\]
where $ind_{\mu}(\rw)=\frac{1}{2}\sum_{i=1}^{n}(1-s_{i}(\rw))$.
\end{cor}
\begin{ex}
It follows from Proposition~\ref{pairs} that for flag manifolds $U(n)/T^n$ the fixed points for the canonical action of $T^n$ can be divided into the pairs $(\rw ,{\tilde \rw})$ such that for an arbitrary invariant almost complex structure $\ind_{\mu}({\tilde \rw}) = \ind_{\mu}(\rw )\pm 2l\pm 1$, for $l\geq 0$. Together with Corollary~\ref{sign_almcom} this provides one more proof that the  signature of the flag manifolds vanishes.
\end{ex}
\begin{ex}\label{Gr-sign}
Consider the Grassmann manifold $G_{4,2}$. Take the ordering $x_1>x_2>x_3>x_4$ on the canonical coordinates for the Lie algebra of the maximal torus $T^4$. The roots for the invariant complex structure are: $x_1-x_3,x_1-x_4,x_2-x_3,x_2-x_4$ and the weights at the  fixed point are given by the action of $S_4/S_2\times S_2$ on these roots. It implies that the indexes of fixeds point are: $ind(1234)=0,ind(3214)=2,ind(4213)=3, ind(1324)=1, ind(1432)=2, ind(3412)=4$. Then Corollary~\ref{sign_almcom} implies that $sign(G_{4,2})=2$.  In the same way we obtain that $sign(G_{6,2,2})=6$. Recall that the signature of Grassmannians $G_{n,k}$ is explicitly computed in~\cite{SH}, while in~\cite{HS} and~\cite{SL}, among the other examples, the signature of compact symmetric spaces is computed . 
\end{ex}

For  the Todd genus for $(G/H, J)$ we deduce from~\eqref{Hinv} that it is equal to the number of those fixed points $\rw$ for which all roots $\rw (\alpha _{i})$, $1\leqslant i\leqslant n$  are positive:
\begin{equation}\label{Toddinv}
Td(G/H, J) = \| \{\rw\in W_G/W_H\; |\; s_{i}(\rw)= 1\; \text{for}\;\text{all}\; 1\leq i\leq k \}\|.
\end{equation}
We can elaborate this further. We use the following criterion  for integrability of an invariant almost complex structure on $G/H$ proved in~\cite{BH}.  
\begin{prop}\label{PBH}
Let $\theta$ be a system of positive roots for $H$. The roots of an invariant complex structure form a closed system $\psi$ such that $\theta \cup \psi$ is a positive system of roots for $G$. Conversely, a closed set $\psi$ of complementary roots such that $\theta \cup \psi$ is the set of positive roots of $G$ for a suitable ordering, is the system of roots of an invariant complex structure of $G/H$.
\end{prop}

Let $J$ be an  integrable invariant almost complex structure on $G/H$ and $\theta$ an arbitrary system of positive roots for $H$. One can  choose an ordering $\mu$ on the canonical coordinates for $\TT$ related to $G$,  such that the roots from  $\theta$ and the roots  $\alpha _1,\ldots ,\alpha _n$ that define $J$ give the system of positive roots for $G$ related  to the ordering $\mu$. We will show that the identity element $e$ is the only fixed point $\rw \in W_{G}/W_{H}$ for which all the weights $\rw (\alpha _{i})$, $1\leqslant i\leqslant n$ are positive related to the ordering  $\mu$.
\begin{lem}\label{UN}
Let $\mu$ be an ordering on the canonical coordinates for $\TT$ related to $G$ such that $\theta \cup \psi$ is a system of positive roots for $G$, where $\theta$ is a system of  positive  roots for $H$ and $\psi$ is some system of  
complementary roots for $G$ related to $H$. Then for any $\rw \in W_{G}/W_{H}$ such that $\rw \neq e$ the system $\rw(\psi)$ is not positive related to the ordering $\mu$.
\end{lem}
\begin{proof}
Let us assume that $\rw (\psi)$  is a positive root system related  to the ordering $\mu$ for some $\rw \in W_{G}/W_{H}$, or in other word related to the Weyl chamber $B$ defined by $\mu$.  It is well known~\cite{Adams} that the Weyl group $W_{G}$ permutes the systems  of positive roots for $G$ what implies  that, for any $\rw \in W_{G}$, the system $\rw (\theta \cup \psi)$ is positive related to the Weyl chamber $\rw(B)$, but it may not be positive related to the chamber $B$. It follows that the  system  $\rw (\theta)$ is positive related to  the Weyl chamber $\rw(B)$, but it is not positive related to the Weyl chamber $B$. The positive root system $\theta ^{'}$ for $H$ related to $B$, in the new canonical coordinates given by the action of $\rw$,  we obtain by the action of an element $\ru \in W_{H}$ on $\rw (\theta)$, i.~e.~ $\theta ^{'}=\ru (\rw (\theta))$.  Note that, as $\psi$ is the system of complementary roots, it is invariant under the action of $W_{H}$, what implies that $\rw (\psi)$ will be invariant under the action of $W_{H}$ given in the new canonical coordinates determined by the action of $\rw$. Therefore we obtain that the root system $\ru (\rw (\psi \cup \theta )) = \rw (\psi)\cup \theta^{'}$ for $G$ is  positive related to $B$. It implies that it coincides with $\psi \cup \theta$, what means that $\ru \circ \rw$ belongs to $W_{H}$, i.~e.~$\rw \in W_{H}$. Therefore we obtain that $\rw=e$ in $W_{G}/W_{H}$.        
\end{proof}
Using Proposition~\ref{PBH} and Lemma~\ref{UN} we deduce from~\eqref{Toddinv} explicit values for the Todd genus of an invariant almost complex structure on $G/H$. 
\begin{cor}
Let $J$ be   an invariant almost complex structure on $G/H$. 
\begin{itemize}
\item If $J$ is integrable then 
\[
Td(G/H, J) =1.
\]
\item If $J$ is not integrable then
\[
Td(G/H, J)=0.
\]
\end{itemize}
\end{cor}   
\begin{ex}
We provide computation of the Hirzebruch $\chi _{y}$ genus and the Todd genus  for $\C P^{3}$ endowed with equivariant stable complex structures $c_{\tau}$. Because of dimension reasons the signature of $\C P^{3}$ is trivial. The roots of the standard complex structure  are $\alpha _{1}=x_1-x_4, \alpha _{2}=x_2-x_4, \alpha_{3}=x_3-x_4$, where $x_1,x_2,x_3,x_4$ are the canonical coordinates for $U(4)$. By the result of~\cite{Buch_Terz} we have on $\C P^{3}$, up to conjugation, eight equivariant stable complex structures whose weights at the fixed points are $\epsilon_{i}(\rw)\rw(\alpha_{i})$, where $\rw \in \Z _{3}$ and the coefficients $\epsilon_{i}(\rw)$ satisfy the relations $\epsilon_{1}(0)=\epsilon_{1}(3)=\epsilon_{1}(2)$, $\epsilon_{2}(0)=\epsilon_{2}(1)=\epsilon_{2}(3)$, $\epsilon_{3}(0)=\epsilon_{3}(1)=\epsilon_{3}(2)$ and $\epsilon_{1}(1)=\epsilon_{2}(2)=\epsilon_{3}(3)$, so they are determined by $\epsilon_{1}(0), \epsilon_{2}(0), \epsilon_{3}(0)$ and $\epsilon_{1}(1)$.
Let us fix an ordering $\mu$ given by $x_1>x_2>x_3>x_4$. 
\begin{itemize}
\item For $\epsilon_{1}(0)=\epsilon_{2}(0)=\epsilon_{3}(0)=\epsilon_{1}$ we obtain the standard complex structure $J$. Using Corollary~\ref{Hinv} we obtain that
\[
\chi _{y}(\C P^{3}, J)=1-y+y^2-y^3,
\]
and therefore $Td (\C P^{3}, J)=1$.
\item Let $c_{\tau}$  be determined by the values $\epsilon_{1}(0)=\epsilon_{2}(0)=\epsilon_{3}(0)=1$ and $\epsilon_{1}(1)=-1$. We have in this case that $\epsilon =-1$ and  Theorem~\ref{TH} gives that
\[
\chi _{y}(\C P^{3}, c_{\tau}) = y^2 -y,
\]
what further implies that $Td (\C P^{3}, c_{\tau})= 0$.
The same result we obtain for the stable complex structures determined by the values $\epsilon_{1}(0)=\epsilon_{2}(0)=\epsilon_{1}(1)=1$ and $\epsilon_{3}(0)=-1$, by the values $\epsilon_{1}(0)=\epsilon_{3}(0)=\epsilon_{1}(1)=1$ and $\epsilon_{2}(0)=-1$, and by the values $\epsilon_{2}(0)=\epsilon_{3}(0)=\epsilon_{1}(1)=1$ and $\epsilon_{1}(0)=-1$.
\item  For the stable complex structure $c_{\tau}$ determined by $\epsilon_{1}(0)=\epsilon_{2}(0)=1$ and $\epsilon_{3}(0)=\epsilon_{1}(1)=-1$ we obtain that
\[    
\chi _{y}(\C P^{3}, c_{\tau})= 0.
\]
The same Hirzebruch $\chi _{y}$ - genus correspond to the stable complex structures determined by the values $\epsilon_{1}(0)=\epsilon_{3}(0)=1$ and $\epsilon_{2}(0)=\epsilon_{1}(1)=-1$, and by the values $\epsilon_{1}(0)=\epsilon_{1}(1)=1$ and $\epsilon_{2}(0)=\epsilon_{3}(0)=-1$. Proceeding as in~\cite{Buch_Terz} we can show that $[(\C P^{3}, c_{\tau})]=0$ for every stable complex structure of this case.
\end{itemize}
\end{ex}      
\section{Compatible almost complex homogeneous fibrations}
\numberwithin{thm}{section}
Let $G$ be a compact connected Lie group and $H$ and $K$ its closed connected subgroup such that $K\subset H$ and $\rk G=\rk H=\rk K$. We look at the homogeneous fibration
\[
H/K\longrightarrow G/K\longrightarrow G/H\ .
\]
Assume that we are given invariant almost complex structures $J_1$ and $J_2$ on $H/K$ and $G/H$ respectively. This means that $J_1$ is invariant under the canonical action of $H$ on $H/K$ and $J_2$ is invariant under the canonical action of $G$ on $G/H$. Let $\alpha _1,\ldots ,\alpha _l$ be the roots for the structure $J_1$ on $T_{e}(H/K)$ and $\alpha _{l+1},\ldots \alpha _{n}$ be the roots for the structure $J_2$ on $T_{e}(G/H)$. Here $\dim H/K =2l$, $\dim G/H= 2(n-l)$ and $\dim G/K=2n$.  
Then $\alpha_1,\ldots ,\alpha_{n}$ will be complementary roots for $G$ related to $K$. We define the complex structure 
$J$ on $T_{e}(G/K)\cong T_{e}(H/K)\oplus T_{e}(G/H)$ to be  $J_1$  on $T_{e}(H/K)$ and $J_2$ on $T_{e}(G/H)$. 
The structure $J$ is invariant under the isotropy representation for
$K$ at $T_{e}(G/K)$ and therefore it defines an invariant almost complex structure on $G/K$.
\begin{rem} 
The integrability criterion for an invariant almost complex structure on a homogeneous space given in~\cite{BH} implies that the  structure $J$ is integrable if and only if $J_1$ and $J_2$ are integrable.
\end{rem}
\begin{rem}\label{homcomp}
Related to the fibration we consider, it is useful to note the following facts. 
\begin{enumerate}
\item The common maximal torus $T^k$ for $K$, $H$ and $G$ acts canonically on  homogeneous spaces $G/K$ and $G/H$ and regarding to this action the projection $\pi : G/K \to G/H$ is an invariant map. 
\item The structure $J$ induces the almost complex structure on the tangent bundle along the fibers $\tau _{H/K}(G/K)$. To see this we  note that $\tau _{H/K}(G/K)$ can be obtained by the action  of the Lie group $G$ on $T_{e}(H/K)$. As the structure $J$ on  $T_{e}(G/K)$ induces the structure $J_1$ on $T_{e}(H/K)$ it implies that the structure $J$ on $\tau (G/H)$ induces the almost complex structure on $\tau _{H/K}(G/K)$ given by the action of the group $G$ on $J_1$.  
\item Consider an invariant connection $\CC$ on this fibration whose horizontal subbundle $\Ha$ is obtained  by the action of the group  $G$ on the subspace $T_{e}(G/H)$ of the space $T_{e}(G/K)$. Then $J$ induces on $\Ha$ an  invariant almost complex structure $J_{\Ha}$ which can be obtained by the action of the group $G$ on the complex structure $J_2$ in $T_{e}(G/H)$. The real isomorphism $d\pi : \Ha \to \tau (G/H)$ is an almost complex map related to the structures $J_{\Ha}$ and $J_2$, meaning that $d\pi \circ J_{\Ha}=J_2 \circ d\pi$.
\item It follows from (2) and (3) that $d\pi : \tau (G/K)\to \tau (G/H)$ is also an almost complex map related to the structures $J$ and $J_2$ as well. 
\end{enumerate}
\end{rem}
The universal toric genus for $(G/K, J)$ is given by
\[
\Phi (G/K, J) = \sum\limits_{\rw \in W_{G}/W_{K}}\prod_{i=1}^{n}\frac{1}{[\rw (\alpha_{i}](u)} \ ,
\]

Note that any $\rw \in W_{G}/W_{K}$ can be uniquely written as $\rw = \rw _{1}\circ \rw _{2}$ where $\rw _{1}\in W_{G}/W_{H}$ and
$\rw _{2}\in W_{H}/W_{K}$. It implies that
\[
\Phi (G/H, J)=\sum\limits_{\rw_{1}\in W_{G}/W_{H}}\sum\limits_{\rw _{2}\in W_{H}/W_{K}}\prod_{i=1}^{n}\frac{1}{[\rw _{1}(\rw _{2}(\alpha _{i})](u)} \ .
\]
Since the structure $J_1$ is invariant under the isotropy representation for $H$ it  implies that the system of roots $\{ \alpha _{l+1},\ldots ,\alpha _{n}\}$ which define $J_1$ is invariant under the action of $W_{H}/W_{K}$. It further implies
\[
\Phi (G/H, J)= \sum\limits_{\rw _{1}\in W_{G}/W_{H}}\prod _{i=l+1}^{n}\frac{1}{[\rw _{1}(\alpha _{i})](u)}\cdot \sum\limits_{\rw _{2}\in W_{H}/W_{K}}\prod_{i=1}^{l}\frac{1}{[\rw _{1}(\rw _{2}(\alpha_{i}))](u)}.
\] 
In this way we have proved the following statement.
\begin{thm}\label{genus_fibration}
Let $G$ be a compact, connected Lie group and $K\subset H$ its closed connected subgroups such that
$\rk G=\rk H=\rk K$. Assume we are given on $H/K$ an invariant complex structure $J_1$ and on $G/H$ an invariant complex structure $J_2$ defined by the roots $\alpha _{l+1},\ldots ,\alpha _{n}$, where $\dim H/K=2l$ and $\dim G/K=2n$. The structures $J_1$ and $J_2$ define on the total space $G/K$ of the fibration $H/K\longrightarrow G/K\longrightarrow G/H$ the invariant almost complex structure $J$ whose universal toric genus is given by
\begin{equation}\label{genus_fibration_formula}
\Phi (G/K, J)=\sum\limits_{\rw _{1}\in W_{G}/W_{H}}\prod_{i=l+1}^{n}\frac{1}{[\rw _{1}(\alpha _{i})](u)}\cdot \rw _{1}(\Phi (H/K, J_2)) \ .
\end{equation}
\end{thm}

\begin{rem}\label{twisted}
Theorem~\ref{genus_fibration} proves that  universal toric genus of $(G/K, J)$ is the {\it twisted product} of the universal toric genera of the base $(G/H, J_1)$ and the fiber $(H/K, J_2)$, where the twist is done by the elements from $W_{G}/W_{H}$.   
\end{rem}

We say that the universal toric genus of the fibration
$(H/K, J_1)\longrightarrow (G/K, J)\longrightarrow (G/H, J_2)$
is {\it multiplicative} if
\[
\Phi (G/K, J)=\Phi (H/K, J_1)\cdot \Phi (G/H, J_2) \ 
\]
in $U^{*}(BT^k)$. Note that the action of the torus $T^k$ on $G/K$ and $G/H$ we obtain as the restriction of the action of the group $G$ which also commutes with the structures $J$ and $J_2$. Therefore, by Lemma~\ref{G-invariance} we see that  $\Phi (G/K, J)$, $\Phi (G/H, J_2)$ are invariant under the action of the Weyl group  
$W_{G}$ in $U^{*}(BT^k)=\Omega _{U} ^{*}[[u_1,\ldots ,u_k]]$.
The invariance of $\Phi (H/K, J_1)$ under the action of $W_{G}$ gives by  Theorem~\ref{genus_fibration} the sufficient condition for the universal toric genus of a homogeneous fibration to be multiplicative.
\begin{cor}\label{mult}
If the universal toric genus of the fiber $(H/K, J_1)$ is invariant under the action of the Weyl group $W_{G}$ in $U^{*}(BT^k)$ then the universal toric genus of the fibration $(H/K, J_1)\longrightarrow (G/K, J)\longrightarrow (G/H, J_2)$ is multiplicative.
\end{cor}
If apply the Chern-Dold character to the formula~\eqref{genus_fibration_formula} we obtain:
\begin{cor}
\begin{equation}\label{cdf}
ch_{U}(G/K, J) = \sum\limits_{\rw _{1}\in W_{G}/W_{H}}\prod_{i=l+1}^{n}\frac{f(\rw _{1}(\alpha_{i}(x)))}{\rw _{1}(\alpha_{i}(x))}\cdot \rw _{1}(ch_{U}(H/K, J_2)) \ ,
\end{equation}
where $x=(x_1,\ldots ,x_k)$.
\end{cor} 
We elaborate further  the formula~\eqref{cdf} related to the multiplicativity question. It can be written as
\[
ch_{U}\Phi (G/K, J) = \sum\limits_{\rw _{1}\in W_{G}/W_{H}}\prod_{i=l+1}^{n}\frac{f(\rw _{1}(\alpha_{i}(x)))}{\rw _{1}(\alpha_{i}(x))}\cdot \rw _{1}([(H/K, J_1)]+ ch_{U}(H/K, J_1)^{\geq l+1}) \ ,
\]
where $\Big ( ch_{U}\Phi (H/K, J_1)\Big )^{\geq l+1} = ch_{U}\Phi (H/K, J_1)-[(H/K, J_1)]$. As $W_{G}/W_{H}$ acts on the coordinates $x_1,\ldots ,x_k$ it follows that $[(H/K, J_1)]$ is invariant under this action and therefore
\[
ch_{U}\Phi (G/K,J)=ch_{U}(G/H, J_2)\cdot [(H/K, J_1)] +\]
\[+ \sum\limits_{\rw _{1}\in W_{G}/W_{H}}\prod_{i=k+1}^{n}\frac{f(\rw _{1}(\alpha _{i}(x))}{\rw _{1}(\alpha _{i}(x))}\cdot \rw _{1}(ch_{U}\Phi(H/K, J_1)^{\geq l+1}) \ .
\]
In particular it follows that
\[
[(G/K, J)]=[(H/K, J_1)]\cdot [(G/H, J_2)] + 
\]
\[+\Big (\sum\limits_{\rw _{1}\in W_{G}/W_{H}}\prod_{i=l+1}^{n}\Big ( \frac{f(\rw _{1}(\alpha _{i}(x))}{\rw _{1}(\alpha _{i}(x))}\Big )\cdot \rw _{1}\Big ( ch_{U}\Phi(H/K, J_1)\Big )^{\geq l+1}\Big ) _{n} \ ,
\]
where the subscript $n$ denotes that we have chosen the coefficient in $t^n$ in the corresponding polynomial in $t$ according to~\eqref{cob_class}.

We will say that the complex cobordism class for this fibration is {\it decomposable} if $[(G/K, J)]=[(H/K, J_1)]\cdot [(G/H, J_2)]$. 

We immediately obtain:
\begin{cor}
If the complex cobordism class for the homogeneous fibration $(H/K, J_1)\longrightarrow (G/K, J)\longrightarrow (G/H, J_2)$ is decomposable then
\begin{equation}
\Big( \sum\limits_{\rw _{1}\in W_{G}/W_{H}}\prod_{i=l+1}^{n}\frac{f(\rw _{1}(\alpha _{i}(x))}{\rw _{1}(\alpha _{i}(x))} \cdot \rw _{1}(ch_{U}\Phi(H/K, J_1)^{\geq _{l+1}})\Big ) _{n} = 0 \ .
\end{equation}
\end{cor}
This further gives:
\begin{cor}\label{multcob}
If the coefficient in $t^k$ in the formula~\eqref{cob_class} for $ch_{U}\Phi (H/K, J_1)$  is equal to zero for $l+1\leq k\leq n$, then the complex cobordism class is decomposable for any homogeneous fibration $(H/K, J_1)\longrightarrow (G/K, J)\longrightarrow (G/H, J_2)$ such that $\dim G/K=2n$.
\end{cor}
\begin{rem}
Note that in  Corollary~\ref{multcob} we obtain $n-1$ equations in $x_1,\ldots x_k$ and $a_1,\ldots a_n$. If we consider these equations as being in variables $a_1,\ldots ,a_n$ 
then they give the constraints on the almost complex homogeneous space $(H/K, J_1)$. If we consider these equations as the equations in variables $x_1,\ldots ,x_k$ they produce the system of relations $R$ in
$\Z [a_1,\ldots ,a_n]$ such that the decomposability of the complex cobordism class  of our fibration is satisfied in the quotient $\Z [a_1,\ldots a_n]/\left\langle R\right\rangle$.
\end{rem}    

\begin{ex}
Take $G=SU(4)$ and $H=S(U(1)\times U(1)\times U(2))$ and look at the fibration
\[
\C P^1\to SU(4)/T^3\to SU(4)/S(U(1)\times U(1)\times U(2)).
\]
The torus $T^3$ acts on these fibrations. Let as consider on $SU(4)/S(U(1)\times U(1)\times U(2))$ an invariant almost complex structure $J_1$  defined by the roots $\alpha_1=x_1-x_3, \alpha_2=x_2-x_3, 
\alpha_3=x_4-x_1, \alpha_4=x_4-x_2, \alpha_5=x_3-x_4$
and the canonical complex structure $J_2$ on $\C P^{1}$ by the root $x_1-x_2$. The structures $J_1$ and $J_2$ define the invariant almost complex structure on $SU(4)/T^3$ which is not integrable, as $J_1$ is not integrable. Formula~\eqref{utg} gives that
\[
\Phi (SU(4)/T^3,J)=\sum_{\rw\in W_{SU(4)}/W_{SU(2)}}\rw\Big (\Phi (\C P^1, J_2)\cdot \prod_{i=1}^{5}\frac{1}{[\alpha_{i}](\bf{u})}\Big).
\]
\end{ex}
\begin{ex}
Consider now fibration
\[
SU(3)/T^2\to G_2/T^2\to S^6=G_2/SU(3).
\]
The torus $T^2$ acts on this fibration canonically. Consider, unique up to conjugation, invariant almost complex structure $J_1$ on $S^6=G_2/SU(3)$. Its roots are  $x_1$, $x_2$ and $x_3$. Let further $J_2$ be the canonical complex structure on $SU(3)/T^2$. The structures $J_1$ and $J_2$ define an invariant almost complex structure on $G_2/T^2$. Then  formula~\eqref{utg} implies that
\begin{equation}\label{utg6}
\Phi (G_2/T^2, J) = \sum _{\rw\in W_{G_2}/W_{SU(3)}}\rw\Big( \Phi (SU(3)/T^2, J_2)\cdot\frac{1}{[(1,0,0)](\bf{u})}\cdot\frac{1}{[(0,1,0)](\bf{u})}\cdot\frac{1}{[(0,0,1)](\bf{u})}\Big ).
\end{equation}

Note that $W_{G_2}/W_{SU(3)}$ consists of two element: $\rw_1$ which is identity and $w_2$ which acts as $w_2(x_i)=-x_i$, $1\leqslant i\leqslant 3$.  The universal toric genus for $SU(3)/T^2$ is
\[
\Phi (SU(3)/T^2,J_2)=\sum _{\rw\in S_3}\frac{1}{[\rw (1,-1,0)](\bf{u})}\cdot \frac{1}{[\rw( 1,0,-1)](\bf{u})}\cdot \frac{1}{[\rw (0,1,-1)](\bf{u})}.
\]
It checks directly from this expression that $\rw_{2}\Big(\Phi (SU(3)/T^2,J_2)\Big)=\Phi (SU(3)/T^2, J_2)$. Using this 
we obtain from~\eqref{utg6} that
\begin{equation}\label{prod}
\Phi (G_2/T^2, J)=\Phi(SU(3)/T^2, J_2)\cdot \Phi (S^6, J_1).
\end{equation}
Formula~\eqref{prod} is valid for an arbitrary invariant almost complex structure on $SU(3)/T^2$ since  its roots, up to signs, coincide with those for $J_2$. 

Specially we can  take an almost complex structure $J_2$ on $SU(3)/T^2$ to be defined with the roots $x_1-x_2$, $x_3-x_1$ and $x_2-x_3$. We obtain that $(S^6, J_1)$, $(SU(3)/T^2, J_2)$ and $(G_2/T^2, J)$ are all $SU$-manifolds,
where the structure $J$ is defined using $J_1$ and $J_2$.  
\end{ex}
\begin{ex}\label{nonmult}
The fibration $U(3)/T^3\to U(4)/T^4 \to \C P^3$ can be obtained as the associated 
$T^4$-bundle over $\C P^3$ with the fiber $U(3)/T^3$. Namely, if  take $E$ to be the principal $T^4$-bundle over 
$\C P^3$ and the canonical action of $T^4$ on $(U(1)\times U(3))/T^4$, it is a classical result that the associated bundle $E\times _{T^4}U(3)/T^3$ will give us the considered fibration. Assume that the  space $U(4)/T^4$ is endowed with an invariant almost complex structure coming from the invariant almost complex structures on the base  and the fiber. Then the  universal toric genus for this fibration is not multiplicative. This follows from Example~\ref{nondecomposable} where it is proved that the cobordism class for $U(4)/T^4$ is not decomposable in $\Omega _{U}^{*}$. 
\end{ex} 
\begin{rem}
It is known~\cite{BT} that the elliptic genus $\varphi$ is multiplicative for the associated $G$-bundles $E\times _{G}F$ over the closed oriented base $B$ with the $G$-fiber $F$ whose first Chern class vanishes. Therefore if in  Example~\ref{nonmult} we take on $U(4)/T^4$  the invariant almost complex structure obtained from the invariant almost complex $SU$-structure $J_1$ on  $U(3)/T^3$  and the unique, up to conjugation, invariant almost complex structure $J_2$ on 
$\C P^3$, we have that the elliptic genus will be multiplicative for such fibration, i.~e.~$\varphi (U(4)/T^4, J)=\varphi (U(3)/T^3, J_1)\cdot \varphi (\C P^3, J_2)$.
\end{rem} 
\subsection{Invariant almost complex structures on homogeneous fibrations.} 
\numberwithin{thm}{subsection}
We now summarize the different cases describing the existence of an invariant almost complex structure on homogeneous fibrations $H/K\to G/K\to G/H$.

{\bf 1.} If $G/H$ and $H/K$ admit  invariant almost complex structures then, as it is described in the previous section, they define the invariant almost complex structure on $G/K$. 

{\bf 2.} If $G/K$ admits an invariant almost complex structure then it naturally induces the invariant almost complex structure on the fiber $H/K$, but in  general case the base $G/H$ does not admit any even stable complex structure.
It can be seen from  the homogeneous fibration:
\[
\C P^{1}\times Sp(n-1)/T^{n-1}\lra \dr{Sp(n)/T^{n}}{\pi}{Sp(n)/Sp(1)\times Sp(n-1) = \HH P^{n-1}} \ .
\]
Being totally reducible $Sp(n)/T^n$ admits $2^{n^2}$ invariant almost complex structures, while  $\HH P^{n-1}$ does not admit any almost complex structure.
\begin{rem}
As we already remarked the fact that $\HH P^{n}$ does not admit any almost complex structure was proved in~\cite{H} for $n\neq 2,3$, while in~\cite{M} it is proved that for $n>1$ it does not admit any stable complex structure making use of the ring $K(X)$ of a complex vector bundles over a space $X$. 
The fact that $\HH P^n$ does not admit any $Sp(n)$-invariant almost complex structure can be proved using toric genus. We illustrate it for $n=2$.

The canonical action of $T^{n+1}$ on $\HH P^{n}$ has $n+1$ fixed points. 
The complementary roots for $Sp(n+1)$ related to $Sp(1)\times Sp(n)$ are $x_1+x_2,\ldots,x_1+x_{n+1}$,$x_1-x_2,\ldots,x_1-x_{n+1}$. The action of the group $W_{Sp(n+1)}/W_{Sp(1)}\times W_{Sp(n)}$ is given by the permutation between $x_1$ and $x_2,\ldots,x_{n+1}$.
If $\HH P^{n}$ would admit an invariant almost complex structure its roots would be given by $\alpha_1=\epsilon_2(x_1+x_2),\ldots,\alpha_n=\epsilon_{n+1}(x_1+x_{n+1})$,$\alpha_{n+1}=\delta_2(x_1-x_2),\ldots,\alpha_{2n}=\delta_{n+1}(x_1-x_{n+1})$, where $\epsilon_{i},\delta_{i}=\pm 1$.
 
In this case by~\cite{Buch_Terz} the coefficient in $t^{l}$, $0\leq l\leq 4n-1$ in the series
\[
\sum_{\rw \in W_{Sp(n+1)}/W_{Sp(1)}\times W_{Sp(n)}}\frac{\prod\limits_{j=1}^{2n}(1+a_1t\rw(\alpha_j)+a_2t^2(\rw(\alpha_j))^2+\ldots)}{\prod\limits_{j=1}^{2n}\rw(\alpha_j)}
\]
has to vanish.

For the coefficient in  $t$ this implies that 
\[
a_1 \sum\limits_{i=1}^{n+1}\frac{(\sum\limits_{j=2}^{n+1}(\epsilon_j+\delta_j))x_i+\sum\limits_{j=2,j\neq i}^{n+1}(\epsilon_j-\delta_j)x_j+(\epsilon_i-\delta_i)x_1}{\prod\limits_{j=1,j\neq i}^{n+1}(x_i^2-x_j^2)} = 0.
\]

For $n=2$ we obtain that polynomial
\[
(-\epsilon_1+\epsilon_2+\delta_1-\delta_2)x_1^3-(\epsilon_2+2\delta_1+\delta_2)(x_1^2x_2+x_1x_3^2)+(\epsilon_1+\delta_1+2\delta_2)(x_1x_2^2+x_1^2x_3)\]
\[ +(\epsilon_2-\delta_1)(x_3^3-x_2^3)-(\epsilon_1+\epsilon_2+2\delta_2)x_2^2x_3+(2\epsilon_2+\delta_1+\delta_2)x_2x_3^2
\]
has to vanish for $x_1\neq \pm x_2,\pm x_3$ and $x_2\neq \pm x_3$.

It implies that $\epsilon_2=\delta_1$ what would further give $\epsilon_1=5\epsilon_2$. This is impossible since
$\epsilon_1,\epsilon_2=\pm 1$. 
\end{rem}
\begin{rem}
We want to point that, the property of a  fibration $F\to E\to B$ that $E$ admits  a stable complex structure  while
$B$ does not admit any  stable complex structure as the above one has, geometrically means that there is no connection on this fibration such that the stable  complex structure on the total space induces the stable complex structure on the horizontal subbundle.
\end{rem}

{\bf 3.} Not every invariant almost complex structure on the $G/K$ can be obtained from invariant almost complex structures of the base and the fiber, even in the case when the base $G/H$ admits an invariant almost complex structure. 
In order to verify this one can consider the fibration
\[
\C P^{1}\to U(3)/T^3 \to \C P^{2} \ .
\]
The base and the fiber, being totally irreducible, admit $2$ invariant almost complex structures, meaning that $4$ invariant almost complex structures on $U(3)/T^3$ come from the structures on the fiber and base in the given fibration. Since on $U(3)/T^3$ there are $2^3=8$ invariant almost complex structures, it follows that four of them do not come from the structures of the fibration.
 
Asking the question what are  the conditions for a invariant almost complex structure of $G/K$ to be obtained from  invariant almost complex structures on $G/H$ and $H/K$ we deduce the following.

\begin{prop}\label{SF}
Let $J$ be an invariant almost complex structure on $G/K$ defined by the roots $\alpha _1,\ldots,\alpha _n$ such that $\alpha _{l+1},\ldots ,\alpha _{n}$ are the complementary roots for $G$ related to $H$. The structure $J$ comes from invariant almost complex structures on the fiber and base of the fibration
$H/K\to G/K\to G/H$ if and only if the root system  $\alpha _{l+1},\ldots ,\alpha _{n}$ is invariant under the action of the Weyl group $W_{H}$. 
\end{prop}
\begin{proof}
The root system $\alpha _{l+1},\ldots ,\alpha _{n}$ defines the complex structure $J$ on $T_{e}(G/H)$. 
This structure  will define an invariant almost complex structure on $G/H$ if and only if it is invariant under the isotropy representation for $H$ at $T_{e}(G/H)$. This is further equivalent to the request  that its root system is invariant under the action of the Weyl group $W_{H}$.
\end{proof}

Regarding the universal toric genus  Proposition~\ref{SF} directly implies.

\begin{cor}
Let $J$ be an invariant almost complex structure on $G/K$ defined by the roots $\alpha _1,\ldots,\alpha _n$ such that $\alpha _{l+1},\ldots ,\alpha _{n}$ are the complementary roots for $G$ related to $H$. Let $J_1$ be the induced invariant almost complex structure on the fiber of the fibration $H/K\to G/K\to G/H$. If the root system $\alpha _{l+1},\ldots ,\alpha _{n}$ is invariant under the action of the Weyl group $W_{H}$ then for the universal toric genus of $(G/K, J)$ twisted product formula holds.  
\end{cor}

We  characterize now almost complex homogeneous fibrations whose  universal toric genus is multiplicative.

\begin{prop}\label{H-semisimple}
Let $H$ be a subgroup of $G$ such that $\rk H=\rk G$. The root system for $H$ is invariant under the action of the Weyl group $W_{G}$ if and only if $H$ is a semisimple Lie group.
\end{prop}
\begin{proof}
For the background on the Lie theory we use see~\cite{VO}. If $H$ is a semisimple Lied group, the root system for $H$ can be obtained by removing one root from the extended root system for $G$. The direct checking gives  that any such root system will be invariant under the action of $W_{G}$. To prove the opposite direction assume that $H$ is not semisimple. Then the Lie algebra $\eta$  for $H$ decomposes as $\eta = \xi (\eta )\oplus \eta ^{'}$ where the dimension of the center $\xi (\eta )$ is  positive and the root system for $\eta$ is that of $\eta ^{'}$. The Cartan algebra $\TT ^{'}$ for $\eta ^{'}$ is strictly contained in the Cartan algebra $\TT$ for $\gg$, what implies that there exists at least one canonical coordinate on $\TT$ related to  $\gg$ which does not appear in the root system for $\eta ^{'}$. Since for any two canonical coordinates for $\gg$ there always exists an element in the  Weyl group $W_{G}$ which interchanges them, it implies that the root system for $H$ is not  invariant under the action of $W_{G}$. 
\end{proof}
Corollary~\ref{mult} together with Proposition~\ref{H-semisimple} implies:
\begin{cor}\label{semisimple_mult}
If $H$ is a semisimple Lie group such that $K<H<G$ and $rk K=\rk H=\rk G$ then the universal toric genus of the almost complex homogeneous fibration $(H/K, J_1)\to (G/K, J)\to (G/H, J_2)$ is multiplicative.
\end{cor}
The pairs $(G,H)$ where $H$ is a semisimple subgroup of $G$ of equal rank are classified for the simple Lie groups $G$. in terms of the corresponding Lie algebras
such classification can be, for example, found in~\cite{On}. We list such pairs for the classical simple Lie algebras and for the  exceptional Lie algebras $G_2$ and $F_4$. They are given with:  $(B_{l}, D_{k}\oplus B_{l-k})$, $l\geq 2$, $2\leq k\leq l$, $(C_{l}, C_{k}\oplus C_{l-k})$, $l\geq 3$, $1\leq k\leq \left[ \frac{l}{2}\right]$, $(D_{l}, D_{k}\oplus D_{l-k})$, $l\geq 4$, $2\leq k\leq \left[ \frac{l+1}{2}\right]$, $(F_{4}, A_{1}\oplus C_{3})$, $(F_{4}, A_{2}\oplus A_{2})$, $(F_4, B_4)$, $(G_2, A_2)$, $(G_2, A_1\oplus A_1)$.

\subsection{On rigidity and multiplicativity of a Hirzebruch genus.}
\numberwithin{thm}{subsection}
We analyse now the relation between mulptiplicativity of a Hirzebruch genus for an almost complex homogeneous fibration and its equivariant rigidity on the fiber. It is clear that if the universal toric genus is multiplicative for some fibration, then any Hirzebruch genus will be multiplicative for that fibration. In the case of homogeneous fibrations we prove:

\begin{prop}\label{H-rig-mult}
Let us given an almost complex homogeneous fibration $H/K\to G/K\to G/H$. 
Assume that the Hirzebruch genus $\LLL _{f}$ is $T^k$-rigid on $H/K$. Then it will be  multiplicative for this fibration. Moreover
\begin{enumerate}
\item  if $\LLL _{f}$ vanishes on $H/K$ then $\LLL _{f}^{T^k}$ will vanish on $\Phi (G/K)$;
\item  if $\LLL _{f}$ does not vanish on $H/K$ then $\LLL _{f}$ will be $T^k$-rigid on $G/K$ if and only if it is $T^k$-rigid on $G/H$.
\end{enumerate}  
\end{prop}
\begin{proof}
Since $\LLL _{f}$ is $T^k$-rigid on $H/K$ it follows from Theorem~\ref{genus_fibration} that  $\LLL _{f}^{T^k}(\Phi (G/K))= \LLL _{f}(H/K)\cdot \LLL _{f}^{T^k}(\Phi (G/H))$ what implies that $\LLL _{f}(G/K) = \LLL _{f}(H/K)\cdot \LLL _{f}(G/H)$.

Also if $\LLL _{f}(H/K)=0$ it follows that $\LLL _{f}^{T^k}(\Phi (G/K))=0$ and, thus, $\LLL _{f}$ is $T^{k}$-rigid on $G/K$ and vanishes on it.

If $\LLL _{f}$ is $T^k$-rigid on $G/H$ it follows $\LLL _{f}^{T^k}(\Phi (G/K))= \LLL _{f}(H/K)\cdot \LLL _{f}(G/H)=\LLL _{f}(G/K)$.
If $\LLL _{f}$ is $T^k$-rigid on $G/K$ then  $\LLL _{f}(G/K) = \LLL _{f}(H/K)\cdot \LLL _{f}^{T^k}(\Phi (G/H))$, what implies that $\LLL _{f}^{T^k}(\Phi (G/H))=\LLL _{f}(G/H)$. 
\end{proof} 
Together with Theorem~\ref{flag-rigidity} we deduce the following.
\begin{cor}\label{flag-fibr-rig}
Any Hirzebruch genus $\LLL _{f}$ given by an odd  series $f$ will be multiplicative for the almost complex homogeneous fibrations with a fiber being the flag manifold. Moreover, it will vanish on the total space of a fibration.
\end{cor} 
This is another  way, compare to Theorem~\ref{large}, to obtain a large class of homogeneous spaces for which any Hirzebruch genus given by an odd  series vanishes.
\begin{ex}
For any Hirzebruch genus $\LLL _{f}$ given by an odd power series $f$ and $M=U(k+n_1+\ldots +n_l)/(T^k \times U(n_1)\times \cdots \times U(n_l))$, $k\geq 2$, we  obtain that $M$ is $T^{k+n_1+\ldots n_l}$-rigid and  $\LLL _{f}(M) =0$. In particular the elliptic genus and the $\hat {A}$-genus are trivial on $M$. This follows from Corollary~\ref{flag-fibr-rig}  
if look at the fibration $U(k)/T^{k}\to U(k+n_1+\ldots +n_l)/(T^k\times U(n_1)\times \cdots \times U(n_l))\to U(k+n_1+\ldots+n_l)/(U(k)\times U(n_1)\times \cdots U(n_l))$, where $k\geq 2$.
\end{ex}

Recall that the Hirzebruch genus $\LLL _{f}$ is said to be multiplicative with respect to the stable complex closed manifold $M$ if  $\LLL _{f}(M\times _{G}E)=\LLL _{f}(M)\cdot \LLL _{f}(G)$, where  $G$ is a compact Lie group whose action on $M$ preserves the stable complex structure and $E$ is the principal $G$-bundle over an arbitrary closed stable complex base $B$.   It is proved in~\cite{BPR} that a genus multiplicative with the respect to the stable complex $T^k$-manifold $M$  is $T^k$-rigid on $M$ as well, while a $T^k$- rigid genus on $M$ will be also multiplicative for such Lie groups groups for which $\Omega _{U}^{*}(BG)$ has no torsion. 

\begin{rem}
Note that it is a classical result that the signature is multiplicative for an arbitrary fibration for which the base, fiber and the total space are coherently oriented compact connected manifolds and   the fundamental group of the base acts trivially on the real cohomology of the fiber~\cite{CHS},~\cite{HBJ}. It follows that the signature will be a $T^k$-rigid genus for an arbitrary stable complex $T^k$-manifold. We want to point that,  using recurrently multiplicativity of the signature related to the fibration $U(n-1)/T^{n-1}\to U(n)/T^n\to \C P^{n-1}$,  it can be  also deduced that the signature of  flag manifolds $U(n)/T^n$ vanishes. 
\end{rem} 

In the case of  associate homogeneous fibrations using Proposition~\ref{H-rig-mult} we obtain slightly different result.

\begin{cor}
Assume that the genus $\LLL _{f}$ is $T^k$-rigid on a homogeneous space $H/K$ of positive Euler characteristic endowed with an invariant  almost  complex structure. Let $G$ be  a compact Lie group such that $H<G$, $\rk H=\rk G$  and $G/H$ admits an invariant almost complex structure. Then the genus $\LLL _{f}$ will be multiplicative for the fibration
$H/K \to H/K\times _{K}E\to G/H$, where $E$ is the principal $K$-bundle over $G/H$.  
\end{cor}

\section{Compatible tangentially stable complex $T^k$-fibrations}\label{TSCFB} 
\numberwithin{thm}{section}

We want to generalize the notion of an almost complex homogeneous $T^k$-fibration to an arbitrary fibration related to its property that for its universal toric genus the twisted product formula holds. Consider the smooth fibration $\dr{F}{i} \dr{E}{\pi}{X}$ where the base $X$ and the total space $E$ are endowed with tangentially stable complex $T^k$-structures $(c_{\tau (E)},\theta _{E})$ and $(c_{\tau (X)},\theta _{X})$. 
We set the following notions.
\begin{enumerate}
\item The  map $\pi$ is said to be $T^k$-invariant if it commutes with the given $T^k$-actions on $E$ and $X$.
\item The map $d\pi : \tau (E)\to \tau (X)$ is said to be stable complex  related to the stable complex structures $c_{\tau (E)}$ and $c_{\tau (X)}$ if the composition
\begin{equation}
\xi _{1}\stackrel{c_{\tau (E)}^{-1}}{\longrightarrow}\tau (E)\oplus
\R ^{2k}\stackrel{d\pi \oplus I_{2l}}{\longrightarrow}
\tau (X)\oplus \R ^{2l}\stackrel{c_{\tau (X)}}{\longrightarrow}\xi _{2} 
\end{equation}
is a complex transformation, where $k\geq l$ and $I_{2l}$ denotes the restriction from $\R ^{2k}$ to $\R ^{2l}$ on its first $2l$ coordinates. Here $\xi _{1}$ and $\xi _{2}$ are the complex vector bundles over $E$ and $X$ which define $c_{\tau (E)}$ and $c_{\tau (X)}$ respectively.
\end{enumerate}
\begin{defn}\label{comp}
We say that the structures $(c_{\tau (E)},\theta _{E})$ and $(c_{\tau (X)},\theta _{X})$ are compatible with the fibration $\dr{F}{i} \dr{E}{\pi}{X}$ if the following conditions are satisfied:
\begin{enumerate}
\item $\pi $ is a $T^k$-invariant map.
\item The  stable complex structure $c_{\tau (E)}$ induces the stable complex structure $c_{\tau _{F}(E)}$ on the tangent bundle along the fibers $\tau _{F}(E)$.
\item There exists a connection  $\CC$ on this bundle such that $c_{\tau (E)}$ induces the stable complex structure $c_{\Ha}$ on the horizontal bundle $\Ha$ related to $\CC$ and the real isomorphism  $d\pi : \Ha \to \tau (X)$ is a stable complex isomorphism related to the stable complex structures $c_{\Ha}$ and $c_{\tau (X)}$.
\item The stable complex structure $c_{\tau (E)}$ splits into the sum of the stable complex structures $c_{\tau _{F}(E)}$ and $c_{\Ha}$, i.~e.~ $c_{\tau (E)}=c_{\tau _{F}(E)}\oplus c_{\Ha}$. 
\end{enumerate}
\end{defn}
It follows that the existence of  compatible structures implies the  decomposition $\tau (E) \oplus \R ^{2k}= (\tau _{F}(E)\oplus \R ^{2(l-k)})\oplus (\Ha \oplus \R ^{2l})$ such that on each summand $c_{\tau (E)}$ induces the stable complex structure and the map $d\pi : \tau (E)\to \tau (X)$ is going to be stable complex as well. The vice versa is also true in the sense that if the map $d\pi :\tau (E)\to \tau (X)$ is  stable complex  we can omit the second condition in  Definition~\ref{comp} as the following Lemma shows.
\begin{lem}
Assume we are given a smooth fibration $\dr{F}{i} \dr{E}{\pi}{X}$ where the base $X$ and the total space $E$ are endowed with tangentially stable complex structures $c_{\tau (E)}$ and $c_{\tau (X)}$ such that  the map $d\pi :\tau (E)\to \tau (X)$ is stable complex   related to these structures. Then the stable complex structure $c_{\tau (E)}$ induces the stable complex structure $c_{\tau _{F}(E)}$ on the tangent bundle along the fibers $\tau _{F}(E)$.
\end{lem}
\begin{proof}
By the assumption we have that the composition
\begin{equation}
\xi _{1}\stackrel{c_{\tau (E)}^{-1}}{\longrightarrow}\tau (E)\oplus
\R ^{2k}\stackrel{d\pi \oplus I_{2l}}{\longrightarrow}
\tau (X)\oplus \R ^{2l}\stackrel{c_{\tau (X)}}{\longrightarrow}\xi _{2} 
\end{equation}
is a complex transformation. It means that
$(d\pi \oplus I_{2l})(c_{\tau (E)}^{-1}\circ J_1\circ c_{\tau (E)})=(c_{\tau (X)}^{-1}\circ J_2\circ c_{\tau (X)})(d\pi \oplus I_{2l})$, where $J_1$ and $J_2$ are complex structures on $\xi _{1}$ and $\xi _{2}$ respectively. Therefore for $V\in \Ker d\pi \oplus \R ^{2(k-l)}$ we have  $(c_{\tau (E)}^{-1}\circ J_1\circ c_{\tau (E)})V \in \Ker (d\pi \oplus I_{2l})=\Ker d\pi\oplus \R ^{2(k-l)}$. It implies that ${\bar{\xi _{1}}} = c_{\tau (E)}^{-1}(\Ker d\pi \oplus \R ^{2(k-l)})$ is the complex subbundle of $\xi _{1}$ giving the stable complex structure on $\Ker d\pi$. 
\end{proof}
\begin{ex}
It follows from Remark~\ref{homcomp} that any almost complex homogeneous $T^k$-fibration is an example of compatible stable complex $T^k$-fibration, where the projection is a stable complex map.
\end{ex}
The condition for the projection $\pi$ to be a $T^k$-invariant map has an important consequence.
\begin{lem}\label{str}
If a fibration is endowed with the compatible tangentially stable complex $T^k$-structures, 
the action $\theta _{E}$ induces $T^k$-actions $\theta _{F_x}$ on the fibers $F_x$ over the fixed points $x\in X$  for $\theta _{X}$, such that $(F_x, c_{\tau _{F_x}(E)}, \theta _{F_x})$ are stable complex $T^k$-manifolds.
\end{lem}
\begin{proof}
The projection $\pi$ is an invariant map related to $T^k$-action meaning that $\pi (\theta _{E}(f_x))=\theta _{X}(\pi (f_x))=\theta _{X}(x)$
for any $f_x \in F_x$, where $F_x$ is the fiber over $x\in X$. For $x$ being a fixed point for $\theta _{X}$ we obtain that $\pi (\theta _{E}(f_x)) =x$ for any $f_x\in F_x$. This further means that $\theta _{E} : F_{x}\to F_{x}$ and therefore it defines the action of the torus $T^{k}$ on $F_x$. Being both induced by $\theta _{E}$ and $c_{\tau (E)}$, it follows that $(\theta _{F_x}, c_{\tau _{F_x}(E)})$ is tangentially stable complex $T^k$-structure on $F_x$.
\end{proof}
Due to Lemma~\ref{str} we obtain that $(F_{x}, \theta _{E}|F_{x}, c_{\tau (E)}|F_{x})$ is tangentially stable complex $T^k$-manifold for a fixed point $x$ related to the action $\theta _{X}$ and, therefore, the universal toric genus $\Phi (F_{x}, c_{\tau _{F_x}(E)},\theta _{F_x})$ is defined. 

\begin{thm}\label{utgcomp}
Assume we are given a smooth fibration  $\dr{F}{i} \dr{E}{\pi}{X}$ such that $E$ and $X$ are endowed with tangentially stable complex $T^k$-structures $(c_{\tau (E)},\theta _{E})$ and $(c_{\tau (X)},\theta _{X})$ which are compatible with the given fibration. Then
\begin{equation}
\Phi (E, c_{\tau (E)},\theta _{E}) = \sum_{x\in Fix(X)}\sg (x)\prod _{j=1}^{n}\frac{1}{[\Lambda _{j}(x)]{\bf (u)}}\Phi (F_x,\Theta _{E}|F_{x}, c_{\tau (E)}|F_{x}),
\end{equation}
where $\Lambda _{j}(x)$, $1\leq j\leq n$ are the weights for the action $\theta _{X}$ at a fixed point $x\in X$ and $2n=\dim X$.
\end{thm}
\begin{proof}
Let $e\in E$ be a fixed point for the action $\theta _{E}$. Since the projection $\pi $ is a $T^k$-invariant map we have that $x=\pi (e)$ is the fixed point for $\theta _{X}$ and, thus, by Lemma~\ref{str} we obtain that
$e=f_{x}\in F_{x}$ is the fixed point for $\theta _{F_x}$. Therefore the set of fixed points $Fix(E)$ for $\theta _{E}$
is given by $Fix(E)=\bigcup\limits_{x\in Fix(X)}Fix(F_x)$, where $Fix(X)$ and $Fix(F_x)$ are the fixed point sets for $\theta _{X}$ and $\theta _{F_x}$. Lemma~\ref{str} gives that $\theta _{E}|F_{x}=\theta _{F_x}$  and by definition $c_{\tau (E)}$ induces the stable complex structure on $\tau _{F}(E)$. It implies that the weights for $\theta _{E}$ related to $c_{\tau (E)}$ at the fixed point $e=f_x$ contain the weights for $\theta _{F_x}$ related to $c_{\tau _{F_x}(E)}$ at the fixed point $f_x$. On the other side, $T_{e}(E)= T_{f_x}(F_x)\oplus H_{e}$, where $H_e$ is the subspace determined by the horizontal subbundle $\Ha$ of the connection $\CC$, then $d\pi : H_e \to \T_{x}(X)$ is a stable complex isomorphism  related to the stable complex structures $c_{\Ha}$ and $c_{\tau (X)}$ and the projection $\pi$ is an invariant map related to the torus actions on $E$ and $X$.  It implies that the weights for $\theta _{E}$ related to $c_{\tau (E)}$ at the fixed point $e$, complementary to those of $\theta _{F_x}$ related to $c_{\tau _{F_x}(E)}$ at the fixed point $e=f_x$, are given by the weights for $\theta _{X}$ related to $c_{\tau (X)}$ at the fixed point $x=\pi (e)$. Since $c_{\tau (E)} = c_{\tau _{F}(E)}\oplus c_{\Ha}$ it follows that  $\sg (e) = \sg (x)\cdot \sg (f_x)$. Therefore we deduce that
\[
\Phi (E, \theta _{E}, c_{\tau (E)}) = \sum _{x\in Fix(X)}\sg (x)\cdot\prod_{j=1}^{n}\frac{1}{[\Lambda _{j}(x)]{\bf (u)}}\Phi (F_x,\Theta _{E}|F_{x}, c_{\tau (E)}|F_{x}).
\]
\end{proof}  

\begin{rem}
Note that Theorem~\ref{utgcomp} shows that the universal toric genus for $(E,\theta _{E}, c_{\tau (E)})$ does not depend
on a connection $\CC$ satisfying Definition~\ref{comp}.
\end{rem}

\subsection{Construction.} We provide now the construction of a compatible tangentially stable complex $T^k$-fibration assuming we are given tangentially stable complex $T^k$-structures on the base  $X$ and the fiber $F$. We denote by $c_{\tau (F)}$ and $c_{\tau (X)}$ the stable complex structures and by $\theta _{F}$ and $\theta _{X}$ the torus actions on $F$ and $X$ respectively.  We want to define the fiber bundle $\dr{F}{i} \dr{E}{\pi}{X}$ with $T^k$-stable complex structure $(c_{\tau (E)}, \theta _{E})$  which naturally arises from the actions $\theta _{F}$ and $\theta _{X}$ and the stable complex structures $c_{\tau (F)}$ and $c_{\tau (X)}$.

\subsubsection{The fiber bundle.} 
\numberwithin{thm}{subsubsection}
Let $G$ be some subgroup of the diffemorphism group Diff$(F)$ for $F$. We assume that $(F, c_{\tau (F)}, \Gamma)$ is a tangentially stable complex $G$ manifold, where $\Gamma$ denotes the action of the group $G$ on $F$. We also assume that the action $\Gamma$ commutes with the action $\theta _{F}$. Consider the principal $G$-bundle $Y$ over $X$ and the associated to it fiber bundle $\dr{F}{i} \dr{E=Y\times _{G}F}{\pi}{X}$.  
The structure group for the  bundle $(E,X,F)$ is $G$.
\subsubsection{The actions $\theta _{F_x}$.} 
Using the action $\theta _F$ we can, in a natural way,  define the action $\theta _{F_x}$ of the torus $T^k$ on the fiber $F_x$ at $x\in X$. We do it as follows. Let us fix a $G$-atlas for $(E, X, F)$ and $\phi : \pi ^{-1}(U) \to U\times F$ be some chart from this atlas, where $U$ is an  open neighborhood for $x\in X$. We define $\theta _{F_x}$ by setting
\[
\theta _{F}(\phi _{x} (f_x)) = \phi _{x}(\theta _{F_x}(f_x)) \ 
\]
for $x\in X$ and $f_x \in F_x$, or, equivalently, $\theta _{F_x} = \phi _{x}^{-1}\circ \theta _{F} \circ \phi _{x}$, where $\phi _{x}=\phi |_{F_x}$.
This definition is good in a sense that it does not depend on the choice of the chart $(\phi , U)$  in the neighborhood of $x$ from the fixed $G$-atlas. Namely, if $(\psi , V)$ is another such chart then $\psi _{x} \circ \phi ^{-1} _{x} \in G$ and, by assumption it commutes with $\theta _{F}$. Therefore we obtain
\[
\theta _{F}\circ \psi _{x}=\theta _{F}\circ \psi _{x}  \circ \phi ^{-1} _{x}\circ \phi _{x}= 
\psi _{x} \circ \phi ^{-1}_{x}\circ \theta _{F}\circ \phi _{x} \ 
\]
what implies $\psi _{x}^{-1}\circ \theta _{F} \circ \psi _{x} = \phi _{x}^{-1}\circ \theta _{F} \circ \phi _{x}$.

\subsubsection{The stable complex structure $c_{\tau (F_{x})}$.} 
\numberwithin{thm}{subsubsection}
Using the stable complex structure $c_{\tau (F)}$ we define the stable complex structure $c_{\tau (F_x)}$ on $F_x$ for any $x\in X$ as follows. By definition $c_{\tau (F)} : \tau (F)\oplus \R ^{2l}\to \xi$ is a real isomorphism for some trivial bundle $\R ^{2l}$ and some complex vector bundle $p:\xi \to F$. Let $\phi$ be an arbitrary chart from the fixed $G$-atlas for $(E,X,F)$. Consider the complex vector bundle $\tilde{\xi _{x}}$ over $F_x$ obtained as the pullback of the complex vector bundle $\xi$ by the diffeomorphism  $\phi _{x} : F_x\to F$, i.~ e.~ $\tilde{\xi _{x}} = \{ (f_x, z)|\; \phi _{x}(f_x)=p(z),\; f_x\in F_x,\; z\in \xi \}$. Denote by $\tilde{\phi _{x}} : \tilde{\xi _{x}}\to \xi$ the projection on the second coordinate. We define the stable complex structure $c_{\tau (F_{x})}$ on $\tau (F_{x})$ by the following diagram
\begin{equation*}\begin{CD}
\tilde {\xi _{x}} @>{\tilde{\phi _{x}}}>> \xi \\
@VV{c_{\tau (F_{x})}}V                   @V{c_{\tau (F)}}VV \\ 
\tau (F_x)\oplus \R ^{2l} @>>{d\phi _{x}\oplus I}> \tau (F)\oplus \R ^{2l}.
\end{CD}\end{equation*}
From the definition of $c_{\tau (F_{x})}$ it follows that $\phi _{x} : F_x \to F$  is a stable complex transformation related to the structures $c_{\tau (F_{x})}$ and $c_{\tau (F)}$. This implies that the structure $c_{\tau (F_{x})}$ is well defined. Namely, for any two charts $\phi$ and $\psi$ in the neighborhood of $x$ and the stable complex structures $c_{\tau (F_{x})}$ and $c_{\tau (F_{x})}^{'}$ these charts define, we obtain that $\psi _{x}^{-1}\circ \phi _{x} : F_{x}\to F_{x}$ is a stable complex transformation related to $c_{\tau (F_x)}$ and $c_{\tau (F_x)}^{'}$. It implies that these stable complex structures are isomorphic as the diagram shows:
\begin{equation*}\begin{CD}
\tilde {\xi _{x}} @>>> \tilde{\xi _{x}^{'}} \\
@VV{c_{\tau (F_{x})}}V            @V{c_{\tau (F_{x})}^{'}}VV\\
\tau (F_x)\oplus \R ^{2l} @>>{d(\psi _{x}^{-1}\circ \phi _{x})}> \tau (F_x)\oplus \R ^{2l}.
\end{CD}\end{equation*}
Because of the commutative diagram
\[
\xymatrix{ \tilde {\xi _{x}} \ar[r]\ar[d] & \tau (F_x)\oplus \R ^{2l} \ar[r]\ar[d] & \tau (F_x)\oplus \R ^{2l} \ar[d]\ar[r] & \xi _{x}\ar[d] \\
\xi \ar[r] & \tau (F)\oplus \R ^{2l} \ar[r] & \tau (F)\oplus \R ^{2l} \ar[r] & 
\xi \ ,}
\]
the stable complex structure $c_{\tau (F_{x})}$ is equivariant under the $T^k$-action $\theta _{F_{x}}$ on $F_{x}$.

The way $c_{\tau (F_x})$ is defined implies that it will be an almost complex structure in the case $c_{\tau (F)}$ is a such structure.

\subsubsection{The stable complex structure $c_{\tau (E)}$.}\label{stablecs} 
\numberwithin{thm}{subsubsection}
In order to define the stable complex structure $c_{\tau (E)}$ on $E$ using the structures $c_{\tau (F)}$ and $c_{\tau (X)}$ we need some additional, geometric, structure on  $(E,X,F)$. Let $\CC$ be a (Ehresmann) connection on the fiber bundle $(E, X, F)$ defined by the smooth horizontal subbundle $\Ha$ of $\tau (E)$ which is complementary to the vertical bundle $\tau _{F}(E) = Ker d\pi$ in the sense that $\tau (E)= \Ha\oplus \tau _{F}(E)$. 

The vertical bundle $\tau _{F}(E)$ consists of vectors tangent to the fibers and, therefore, it inherits the stable complex structure $c_{\tau _{F}(E)}$ from the structures $c_{\tau (F_{x})}$.  

The real isomorphism $d \pi : \Ha \to \tau (X)$ defines the stable complex structure
$c_{\Ha}$ on the horizontal bundle $\Ha$ if we put $c_{\Ha}=(d\pi )^{-1}\circ c_{\tau (X)}^{-1}$.  We  define the stable complex structure $c_{\tau (E)} : \tau (E)\oplus \R^{2l}\oplus \R^{2m}\to \xi \oplus \eta$ by 
\begin{equation}\label{defstable}
c_{\tau (E)}|_{\Ha\oplus \R ^{2m}}\equiv c_{\Ha},\;\;\;
c_{\tau (E)}|_{\tau _{F}(E) \oplus \R ^{2l}}\equiv c_{\tau _{F}(E)}.
\end{equation}

\begin{rem}
If $c_{\tau (F)}$ and $c_{\tau (X)}$ are almost complex structures we obtain that $c_{\tau (E)}$ is an almost complex structure as well.
\end{rem}

\subsubsection{The action $\theta _{E}$.}\label{action} 
\numberwithin{thm}{subsubsection}
Denote by $S_{t}$ the one parameter subgroup of $T^k$ generated by an element $t\in T^k$. For $x\in X$ let $\alpha _{(x,t)}$ be the curve in $X$ which is the orbit of the element $x$ by the action of $S_{t}$ or in other words $\alpha _{(x,t)}=\theta _{X}(S_{t})x$. 
For any $f_x\in F_x$ there exists, regarding to the connection $\CC$, the unique horizontal lift $\tilde{\alpha } _{(f_x,t)}$ of the curve $\alpha _{(x,t)}$ such that $\tilde{\alpha}  _{(f_x,t)}(0)=f_x$. Using this we obtain the family of diffeomorphisms $\Psi _{(x,t)} : F_x \to F_{\theta _{X}(t)(x)}$ defined by $\Psi _{(x,t)}(f_x)=\tilde{\alpha}_{(f_x,t)}(1)$. 
It further defines the family of smooth maps  $\Omega _{t} : E\to E$ by $\Omega _{t}(e)=\Psi _{(x,t)}(f_x)$, where $x=\pi (e)$ and $e=f_x\in F_x$. 

The way the map $\Omega _{t}$ is defined implies that $d\Omega _{t}$ commutes with the stable complex structure $c_{\tau (E)}$ in the sense that the composition
\begin{equation}\label{com}
\dr{\xi \oplus \eta}{c_{\tau (E)}^{-1}}\dr{\tau(E)\oplus \R ^{2l}\oplus \R ^{2m}}{d\Omega _{t}\oplus I\oplus I}\dr{\tau (E)\oplus \R ^{2l}\oplus \R ^{2m}}{c_{\tau (E)}}\xi \oplus \eta 
\end{equation}
is a complex transformation for any $t\in T^k$.

We define the action $\theta _{E}$ of the torus $T^k$ on $E$ by
\begin{equation}\label{defaction}
\theta _{E}(t)(e) = (\theta _{F_{\theta _{X}(t)(x)}}\circ \Psi _{(x,t)})(f_x) \ ,
\end{equation}
where $x=\pi (e)$ and $e=f_x\in F_x$.  Note that from this definition it follows that $\pi (\theta _{E}(t)(e)) = \theta _{X}(t)(x)$. 
The way  $(E,\theta _{E}, c_{\tau (E)})$ is constructed implies that it will will be a tangentially stable complex $T^k$-manifold. 

\begin{thm}
The structures $(\theta _{E},c_{\tau (E)})$ and $(\theta _{X},c_{\tau (X)})$ are compatible with the fibration  $\dr{F}{i} \dr{E}{\pi}{X}$.
\end{thm}
\begin{proof}
It follows from~\eqref{defaction} that $\pi (\theta _{E}(t)(e))= \pi (\theta _{F_{\theta _{X}(t)(x)}}(\Psi _{(x,t)}(f_x))) =$\\$ \pi (\theta _{F_{\theta _{X}(t)(x)}}(\tilde{\alpha}_{(f_x,t)}(1))) = \pi (\tilde{\alpha}_{(f_x,t)}(1))=
\alpha _{(x,t)}(1) =\theta _{X}(t)(x)$, where $x=\pi (e)$ and $e=f_x$. Thus, the projection $\pi$ is a $T^k$-invariant map. The formula~\eqref{defstable} which defines $c_{\tau (E)}$ directly verifies that the  last two conditions in  Definition~\ref{comp} are also satisfied.   
\end{proof}

\begin{cor}\label{UTGF}
If the action $\theta _{E}$ has finite number of  isolated fixed points then the universal toric genus for
$(E,\theta _{E}, c_{\tau (E)})$ is given by
\begin{equation}\label{utgfb}
\Phi (E,\theta _{E},c_{\tau (E)}) = \sum\limits_{x\in Fix(X)}\sg (x)\prod _{j=1}^{n}\frac{1}{[\Lambda_{j}(x)]{\bf (u)}}\Phi (F_x,\theta _{F_x}, c_{\tau (F_x)}) \ ,
\end{equation}
where $\sg (x)$ and $\Lambda _{j}(x)$, $1\leq j\leq n$  are the sign and the weights for the action $\theta _{X}$ at the fixed point $x$.
\end{cor} 

\begin{rem}
Note that the universal toric genus for $(E, \theta _{E}, c_{\tau (E)})$ depends on the triples $(X, \theta _{X}, c_{\tau (X)})$ and $(F, \theta _{F}, c_{\tau (F)})$ and the differentiable $G$- structure of the fibration.
\end{rem}
\begin{cor}
If $\Phi (F_{x_1}, \theta _{F_{x_1}}, c_{\tau (F_{x_1})})=\Phi (F_{x_2}, \theta _{F_{x_2}}, c_{\tau (F_{x_2})})$ in $U^{*}(BT^k)$  for any $x_1,x_2\in Fix(X)$ then
\begin{equation}
\Phi (E,\theta _{E},c_{\tau (E)}) = \Phi (X, \theta _{X}, c_{\tau (X)})\cdot \Phi (F, \theta _{F}, c_{\tau (F)})\ . 
\end{equation}
\end{cor}

Recall that we have the set  of diffeomophisms $\phi _{x} : F_{x}\to F$ induced by the charts $\phi$ in the neighborhood of $x$  of the $G$-atlas for $(E, F, X)$ where $x\in X$. For any such $\phi _{x}$ we obtain the map 
$\bar{\phi} _{x}^{-1} : BT^k \to BT^k$ defined by the diagram  
\begin{equation*}\begin{CD}
F @>>> ET^{k}\times _{T^k}F @>>{p}> BT^k \\
@VV{\phi_{x}^{-1}}V            @V{\tilde{\phi} _{x}^{-1}}VV   @V{\bar{\phi} _{x}^{-1}}VV \\
\tau (F_x) @>>> ET^{k}\times _{T^k}F_{x} @>>{\tilde{p}}> BT^k,
\end{CD}\end{equation*}
where $\tilde{\phi} _{x}^{-1}$ is induced by $\phi _{x}^{-1}$. The following observation directly follows from the properties of the Gysin homomorphism.
\begin{lem}
The map $\bar{\phi} _{x}^{-1}$ induces the homomorphism $(\bar{\phi} _{x}^{-1})^{*} : U^{*}(BT^k) \to U^{*}(BT^k)$ which satisfies
\begin{equation}
(\bar{\phi} _{x}^{-1})^{*}(\Phi (F, c_{\tau (F)}, \theta _{F})) = \Phi (F_{x}, c_{\tau (F_{x})},\theta _{F_{x}}).
\end{equation}
\end{lem}
Then as  the direct consequence of Theorem~\ref{utgfb} we obtain the following.
\begin{cor}\label{genus-fiber}
Let $\phi _{x} : F\to F_x$ be the diffeomorphism given by an arbitrary chart $\phi$ in the neighborhood of $x$ from the given $G$-atlas for $(E, F, X)$ and $(\bar{\phi} _{x}^{-1})^{*} : U^{*}(BT^k)\to U^{*}(BT^k)$ is the induced homomorphism, where $x\in Fix(X)$. Then
\begin{equation}
\Phi (E,\theta _{E},c_{\tau (E)}) = \sum\limits_{x\in Fix(X)}\sg (x)\prod _{j=1}^{n}\frac{1}{[\Lambda_{j}(x)]{\bf (u)}}(\bar{\phi} _{x}^{-1})^{*}(\Phi (F, \theta _{F}, c_{\tau (F)})) \ . 
\end{equation}
\end{cor} 

\begin{ex}
Let $G$ a compact connected Lie group, $H$ its closed connected subgroup such that $\rk H=\rk G=k$. Assume $G/H$ to be endowed with an invariant almost complex structure and with the canonical action of the maximal torus $T^k$. Consider the  
subgroup $T$ of  Diff$(G/H)$ given by the canonical action of the torus $T^k$. Let further $X$ be an arbitrary tangentially stable complex $T^k$-manifold. Consider the principal $T$-bundle over $X$ and the associated to it fiber bundle $\dr{G/H}{i} \dr{E}{\pi}{X}$ and fix some smooth $T$-structure on this bundle. Following above construction we obtain a tangentially stable complex $T^k$-manifold $(E,\theta _{E}, c_{\tau (E)})$. Theorem~\ref{UTGF} gives that   
\[
\Phi (E,\theta _{E}, c_{\tau (E)})=\sum\limits_{x\in Fix(X)}\prod_{j=1}^{n}\frac{1}{[\Lambda_{j}(x)]({\bf u})}(\Phi ((G/H)_{x}, J_{x})) \ .
\] 
In particular we can take $X$ to be a homogeneous space $G_1/H_1$ where $\rk G_1=\rk H_1=k$.
\end{ex}
\begin{ex}
Let $G$ be compact connected Lie group, $H$ its closed connected subgroup such $\rk G =\rk H$. Assume that $G/H$ admits an invariant almost complex structure $J_1$ and fix some invariant almost complex structure $J_2$ on $H/T^k$, where $T^k$ is the common maximal torus for $H$ and $G$. Following the notations from (\ref{TSCFB}) we take $X= G/H$ and $F=H/T^k$ endowed with the structures $J_1$ and $J_2$ and canonical actions $\theta _{X}$ and $\theta _{F}$  of the torus $T^k$.   Let us fix the diffeomorphism group $T$ for $F$ given by the action of $T^k$. The associated fiber bundle to the principal $T$-bundle over $X$ with the fiber $F$ is homogeneous fibration $F=H/T^k\rightarrow E=G/T^k\rightarrow X=G/H$. Let $J$ be the almost complex structure on $E=G/T^k$ obtained from $J_1$ and $J_2$ by the construction from (~\ref{stablecs}) related to the canonical invariant connection $\CC$ on this homogeneous fibration. Using the given actions $\theta _{X}$ and $\theta _{F}$ we may define, following (~\ref{action}), the action $\theta$ of the torus $T^k$ on the total space $E=G/T^k$. Since 
$Fix(G/H)=W_{G}/W_{H}$,  Corollary~\ref{UTGF} implies that
\[
\Phi (G/T^k,\theta , J) = \sum\limits_{\rw \in W_{G}/W_{H}}\prod _{j=1}^{n}\frac{1}{[\Lambda_{j}(\rw )]{\bf (u)}}\Phi ((H/T^k)_{\rw},\theta _{(H/T)_{\rw }},(J_2)_{\rw }),  
\]
where the weights $\Lambda _{j}(\rw )$ for the action $\theta _{G/H}$ related to the almost complex structure $J_1$ are established in~\cite{Buch_Terz}. 
\end{ex}
\begin{rem}
Note that the action $\theta$ on $G/T^k$ from this  example is different from the canonical action of $T^k$ on $G/T^k$, since it induces 
$T^k$-action on each fiber, while the canonical action does not do the same. We also see from the construction of the structure $J$ that it is not necessarily invariant under the action of the group $G$ on $G/T^k$, although it's horizontal component, related to the canonical invariant connection $\CC$, is invariant.
\end{rem} 

\subsubsection{Rigidity and multiplicativity of a Hirzebruch genus.}
Regarding this question for  fibrations we have constructed in this Section it is valid the statement analogous to that for compatible almost complex homogeneous fibrations. Note that $T^k$-equivariant extension  
$\LLL _{f}^{T^k} = \LLL _{f}\circ \Phi$ of any Hirzebruch genus $\LLL _{f}$  commutes with any homomorphism $h: U^{*}(BT^k)\to U^{*}(BT^k)$, meaning that $h\circ (\LLL _{f}\circ \Phi )=\LLL _{f} \circ (h\circ \Phi)$. Then Corollary~\ref{genus-fiber} implies the following.

\begin{prop}
Assume we are given stable complex $T^k$-manifolds $X$ and $F$ and let $E$ be the stable complex manifold obtained as the total space of the fibration over $X$ with the fiber $F$ in the way we described. If the Hirzebruch genus $\LLL _{f}$ is $T^k$-rigid on $F$ then it will be multiplicative for this fibration. Moreover
\begin{enumerate}
\item if $\LLL _{f}$ vanishes on $F$, then $\LLL _{f}^{T^k}$  will vanish on $\Phi (E,\theta _{E},c_{\tau (E)})$;
\item if $\LLL _{f}$ does not vanish on $F$ then $\LLL _{f}$ it $T^k$-rigid on $E$ if and only if it is $T^k$-rigid on X.
\end{enumerate}
\end{prop} 

\section{Tangentially stable complex $T^{n}$-bundles over $\HH P^{n-1}$}
\numberwithin{thm}{section}
We start by recalling some results on toric genera for fibrations from~\cite{BPR}. Assume we are given tangentially stable complex $T^k$-bundle $F\lra \dr{E}{\pi}{X}$. This means that  $E$ and $X$ are smooth manifolds with the actions of $T^k$,  $\pi$ is a smooth map which is invariant under the action of $T^k$ and the tangent bundle $\tau _{F}(E)$  along the fibers is equipped with  $T^{k}$-equivariant stable complex structure $c_{\pi}$. 
\begin{rem}
Note that when $X$ is a point both $E$ and $F$ can be identified with some smooth $T^k$-manifold $M^{2n}$ and 
$\tau _{F}(E)$ can be identified  with its tangent bundle $\tau (M)$, while $c_{\tau}(\pi)$ gives $T^k$-equivariant tangentially stable complex structure $c_{\tau}$ on $M^{2n}$.  
\end{rem}
Denote by $U_{T^k}^{*}(X)$ the set of cobordism classes of tangentially stable complex $T^k$-fibrations over $X$.
We use this notation instead of $\Omega ^{*}_{U:T^k}(X)$ as it is in~\cite{BPR} because of consistency of the notations in the paper.  For any smooth manifold $X$ in~\cite{BPR} it is constructed the homomorphism 
\[
\Phi _{X} : U_{T^k}^{*}(X)\to U^{*}(ET^k\times _{T^k}X),
\]
called {\it the universal toric genus}.  When $X=pt$ we obtain $\Phi _{pt}$: 
\[
\Phi _{pt} : \Omega_{U:T^k}^{*}\to \Omega _{U}^{*}[[u_1,\ldots ,u_k]]
\]
and  it is satisfied $\Phi _{pt}(M^{2n}, c_{\tau}, \theta)=\Phi (M^{2n}, c_{\tau}, \theta)$.
\begin{rem}
In the same way  can be constructed  $G$-genus $\Phi _{X: G} : U_{G}^{*}(X) \to U^{*}(EG\times _{G}X)$, where $U_{G}^{*}(X)$ denotes the set of cobordism classes of tangentially stable complex $G$-fibrations over $X$. Analogously, when $X=pt$ one obtains that $\Phi _{pt:G}(M^{2n}, c_{\tau}, \gamma)=\Phi_{G} (M^{2n}, c_{\tau}, \gamma)$.
\end{rem}

Assume that all fixed points for the action of $T^k$ on $E$ are isolated. Denote by $Fix(E)$ and $Fix(X)$ the sets of  fixed points on the base $X$ and the total space $E$  and by $Fix(F_{x})$ the set of  fixed points in the fiber $F_{x}={\pi}^{-1}(x)$ for some fixed point $x\in Fix(X)$. Note that any $e\in Fix (E)$ can be looked as $e=(x, f_x)$ for  $x\in Fix(X)$ and  $f_x\in Fix(F_{x})$. Let $\omega _{j}(f_x)$ be the weights at the fixed point $f_x\in Fix(F_{x})$ that are determined by $c_{\pi}$ where $1\leq j\leq n$. The following result is proved in~\cite{BPR}, Theorem 4.6.
\begin{thm}
For any tangentially stable complex $T^k$-bundle $F\lra \dr{E}{\pi}{X}$ with isolated fixed points, the equation
\begin{equation}
\Phi _{X}(\pi )|_{Fix(X)} = \sum\limits_{x\in Fix(X)}\sum\limits_{f_{x}\in Fix(F_{x})}\sg (f_x)\prod_{i=1}^{n}\frac{1}{[\omega _{j}(f_x)]({\bf u})}
\end{equation}
is satisfied in $\oplus_{Fix(X)}U^{-2n}(BT_{+}^{k}\times x)$, where $x$ and $f_x$ range over $Fix(X)$ and $Fix(F_{x})$ respectively. 
\end{thm}
\begin{rem}
According to the previous theorem we can assign to  $\Phi (\pi )|_{Fix(X)}$  the vector $\Big( \sum\limits_{f_{x}\in Fix(F_{x})}\sg (f_x)\prod\limits_{i=1}^{n}\frac{1}{[\omega _{j}(f_x)](u)}\Big) _{x\in Fix(X)}$ which belongs to  a $Fix(X)$-dimensional module over $\Omega _{U}^{*}$.
\end{rem}

\begin{thm}
The universal toric genus of the fibration
\begin{equation}\label{quatern}
\C P^{1}\times Sp(n-1)/T^{n-1}\lra \dr{Sp(n)/T^{n}}{\pi}{\HH P^{n-1}}
\end{equation}
of  an arbitrary invariant almost complex structure on $Sp(n)/T^n$ can be written as  
\[
\Phi _{\HH P^n}(\pi )= \sum_{\omega}g_{0,\omega}|{\bf u}|^{2\omega} + \sum_{\omega}g_{1,\omega}|{\bf u}|^{2\omega}z + \ldots + \sum_{\omega}g_{n-1,\omega}|{\bf u}|^{2\omega}z^{n-1} \ ,
\]
where $g_{j,\omega}\in \Omega _{U}^{-2(2n+2+|\omega |)}$ for $\omega =(i_1,\ldots ,i_n)$, then $|{\bf u}|^{2\omega}=|u_1|^{2i_1}\cdots |u_{n}|^{2i_n}$ and $z=p_{1}^{Sp}(\rho )$ is the first symplectic Pontrjagin class of the canonical quaternionic line bundle over the projectivisation $\dr{\HH P(\eta_{\HH}^{1}\oplus\cdots \oplus \eta_{\HH}^{n})}{\rho}{BSp(1)^{n}}$.
In particular if $n=2$, for the invariant complex structure we obtain  the  coefficients $g_{0,\omega}$  to be:
\[
g_{0,\omega} = 2^{2(i_1+i_2+1)}a_{2i_1+1}a_{2i_2+1},\;\; \text{for}\;\; \omega = (i_1,i_2).
\]
\end{thm}
\begin{rem}
Note that $z=p_{1}^{Sp}(\rho )$ restricts to the fiber $\HH P^{n-1}$ of the projectivisation $\rho$ to the first symplectic Pontrjagin class of the canonical quaternionic line bundle over $\HH P^{n-1}$.
\end{rem}
\begin{proof}
Let us fix an invariant almost complex structure $J$ on the quaternionic flag manifold $Sp(n)/T^{n}$. 
The structure $J$ on $\tau (Sp(n)/T^n)$, being a $Sp(n)$-invariant almost complex structure, induces $T^n$-invariant
 almost complex structure on the tangent bundle along the fibers $\tau _{\C P^{1}\times Sp(n-1)/T^{n-1}}(Sp(n)/T^n)$  of the fibration~\eqref{quatern}. The action of  $T^n$ on $Sp(n)/T^{n}$ and on $\HH P^{n-1}$ is compatible with  the fibration~\eqref{quatern}, so it is defined the universal toric genus $\Phi _{\HH P^n}(\pi)$ which is an element in $U^{*}(ET^{n}\times _{T^n}\HH P^{n-1})$. In order to describe $\Phi _{\HH P^n}(\pi)$ explicitly, we consider the following situation. 

The structure $J$ also induces the $Sp(1)^{n}$-invariant 
almost complex structure on the tangent bundle $\tau _{\C P^{1}\times Sp(n-1)/T^{n-1}}(Sp(n)/T^n)$  and the action of the group $Sp(1)^n$ on $Sp(n)/T^{n}$ and on $\HH P^{n-1}$ is compatible with the  fibration~\eqref{quatern}, so it is defined $Sp(1)^{n}$-genus $\Phi _{\HH P^n: Sp(1)^{n}}(\pi)$. It is   an element of
$U^{*}(ESp(1)^{n}\times _{Sp(1)^n}\HH P^{n-1})$. Namely recall that fibration~\eqref{quatern} induces the fibration
\[
\C P^{1}\times Sp(n-1)/T^{n-1}\lra \dr{ESp(1)^{n}\times _{Sp(1)^{n}}Sp(n)/T^{n}}{1\times _{Sp(1)^{n}}\pi}{ESp(1)^{n}\times _{Sp(1)^{n}}\HH P^{n-1}} \ .
\]  
Then the $Sp(1)^{n}$ - genus is defined with $\Phi _{\HH P^n: Sp(1)^{n}}(\pi) = (1\times _{Sp(1)^{n}}\pi)_{!}(1)$, where 
\[
(1\times _{Sp(1)^{n}}\pi)_{!} : U^{*}(ESp(1)^{n}\times _{Sp(1)^{n}}Sp(n)/T^{n})\to U^{*}(ESp(1)^{n}\times _{Sp(1)^{n}}\HH P^{n-1})
\]
is the Gysin homomorphism induced by $1\times _{Sp(1)^{n}}\pi$.

The relation between $\Phi _{\HH P^n}$ and $\Phi _{\HH P^n: Sp(1)^n}$ is by Lemma~\ref{G-genus} given by
\begin{equation}\label{rel_gen}
Bj^{*}\Phi _{\HH P^n: Sp(1)^n} = \Phi _{\HH P^n} \ ,
\end{equation}
where $Bj^{*} : U^{*}(ET^{n}\times _{T^{n}}\HH P^{n-1})\to U^{*}(ESp(1)^{n}\times _{Sp(1)^{n}}\HH P^{n-1})$ is induced by the inclusion $j : T^n\to Sp(1)^n$. 

The cobordism ring $U^{*}(ESp(1)^{n}\times _{Sp(1)^{n}}\HH P^{n-1})$ can be described using projectivisation. Let us consider $\HH P^{n-1} = Sp(n)/Sp(1)\times Sp(n-1)$. The group $Sp(1)^{n}$ acts smoothly on $\HH P^{n-1}$  and we can consider the Borel construction for this action
\[
\HH P^{n-1}\lra ESp(1)^{n}\times _{Sp(1)^{n}}\HH P^{n-1}\lra BSp(1)^{n} \ .
\]
This bundle can be obtained as quaternionic projectivisation of the quaternionic vector bundle $V_{\HH}$ over $BSp(1)^{n}$ which is the direct sum of $n$ - copies of the tautological line bundle $\eta _{\HH }$ over $BSp(1)$, i.~e.
\[
ESp(1)^{n}\times _{Sp(1)^{n}}\HH P^{n-1} = \HH P(V_{\HH}) = \HH P(\eta _{\HH }^1\oplus \cdots \oplus \eta _{\HH}^{n}) \ .
\]
It follows that $U^{*}(ESp(1)^{n}\times _{Sp(1)^{n}}\HH P^{n-1})$ is a free $U^{*}(BSp(1)^{n})$-modul:
\begin{equation}\label{cob_ring_quatern}
U^{*}(ESp(1)^{n}\times _{Sp(1)^{n}}\HH P^{n-1}) = U^{*}(BSp(1)^{n})[z]/\langle z^{n}+p_{1}^{Sp}(V_{\HH})z^{n-1}+\ldots +p_{n}^{Sp}(V_{\HH})\rangle \ ,
\end{equation}
where $p_{1}^{Sp}(V_{\HH}),\ldots ,p_{n}^{Sp}(V_{\HH})$ are the symplectic Pontrjagin classes of the bundle $V_{\HH}\to BSp(1)^{n}$ and $z$ denotes the first symplectic Pontrjagin class of the canonical  quaternionic line bundle over the projectivisation $\HH P(V_{\HH})\to BSp(1)^n$.
 
As the bundle $V_{\HH}\lra BSp(1)^{n}$ splits into the sum of $n$ quaternionic line bundles, it follows that
\[
z^{n}+p_{1}^{Sp}(V_{\HH})z^{n-1}+\ldots +p_{n}^{Sp}(V_{\HH}) = (z+v_1)(z+v_2)\cdots (z+v_{n}) \ ,
\]
where $v_{i}=p_{1}^{Sp}(\eta _{\HH }^{i})$ are the first symplectic Pontrjagin classes of the tautological quaternionic line bundles $\eta _{\HH }^{i}$ for $1\leqslant i\leqslant n$. Note that these classes give the generators for $U^{*}(BSp(1)^{n})$ or in other words $U^{*}(BSp(1)^{n})=\Omega _{U}^{*}[[v_1,\ldots ,v_{n}]]$. We obtain that 
\[
U^{*}(ESp(1)^{n}\times _{Sp(1)^{n}}\HH P^{n-1}) = \Omega _{U}^{*}[[v_1,\ldots ,v_{n}]][z]/\langle (z+v_1)\cdots (z+v_{n})\rangle \ .
\]
It follows that $\Phi _{\HH P^n: Sp(1)^{n}}(\pi)$ can be written as
\[
\Phi _{\HH P^n: Sp(1)^{n}}(\pi) = g_{0} + g_{1}z + \ldots + g_{n-1}z^{n-1} \ ,
\]
where the coefficients $g_{j}\in \Omega _{U}^{*}[[v_1,\ldots ,v_n]]$ are of the form
\[
g_{j}=\sum_{\omega}g_{j,\omega}{\bf v}^{\omega} 
\]
for $\omega = (i_1,\ldots ,i_n)$ where  $i_1,\ldots ,i_n$ are nonnegative integers and ${\bf v}^{\omega}=v_1^{i_1}\cdots v_{n}^{i_n}$.  The coefficients $g_{j,\omega}$ lie in $\Omega _{U}^{-2(2n+2+|\omega |)}$, where $|\omega |=\omega _1 +\cdots +\omega _n$.

The inclusion $T^n \subset Sp(1)^n$ induces the inclusion $U^{*}(BSp(1)^n)=\Omega _{U}^{*}[[v_1,\ldots ,v_n]]\subset U^{*}(BT^n)=\Omega _{U}^{*}[[u_1,\ldots ,u_n]]$ given with $v_i\to u_i\bar{u_i}$ for $1\leqslant i\leqslant n$.  One can get it  from the fibre bundle $\dr{\C P^{2n-1}}{p}{\HH P^{n-1}}$ as $p_{1}(\eta _{\HH}^i) = c_2(\eta _{\C}^i\otimes \bar{\eta _{\C}^i})= c_{1}(\eta _{\C}^i)c_{1}(\bar{\eta _{\C}^i})= u_i\bar{u_i}$. Following~\cite{BN} we put $u_i\bar{u_i} = |u_i|^2$ what implies that $U^{*}(BSp(1)^n)$ includes in $U^{*}(BT^n)$ as $\Omega _{U}^{*}[[|u_1|^2,\ldots ,|u_n|^2]]$.     

Because of~\eqref{rel_gen} we obtain
\[
\Phi _{\HH P^n}(\pi )= \hat{g}_{0} + \hat{g}_{1}z + \ldots + \hat{g}_{n-1}z^{n-1} \ ,
\]
where  the coefficients $\hat{g}_{j}$ are 
\[
\hat{g}_{j} = \sum_{\omega}g_{j,\omega}|{\bf u}|^{2\omega} 
\]
for $\omega =(i_1,\ldots ,i_n)$ and $|{\bf u}|^{2\omega}=|u_1|^{2i_1}\cdots |u_{n}|^{2i_n}$.

In particular, for $n=2$ we have the fibration $\C P^{1}\times \C P^{1}\lra \dr{Sp(2)/T^2}{\pi}{\HH P^{1}}$ and
\begin{equation}\label{fiber_quater}
\Phi _{\HH P^1}(\pi) = \sum_{\omega = (i_1,i_2)}g_{0,\omega}|u_1|^{2i_1}|u_2|^{2i_2} + \big ( \sum_{\omega=(i_1,i_2)}g_{1,\omega}|u_1|^{2i_1}|u_2|^{2i_2}\big )\cdot z \ .
\end{equation}
Let us fix the invariant almost complex structure $J$ on $Sp(2)/T^2$ given by the roots $x_1+x_2,x_1-x_2, 2x_1$ and $2x_2$ for $Sp(2)$. For the ordering $x_1>x_2>0$ all these roots are positive what implies that the structure $J$ is integrable. 
Using Theorem~\ref{main} we obtain  the universal toric genus for $(Sp(2), J)$:   
\begin{equation}\label{SP}
\Phi (Sp(2)/T^2, J) = \sum _{\rw \in W_{Sp(2)}}\frac{1}{[\rw(1,1)]({\bf u})}\cdot\frac{1}{[\rw(1,-1)]({\bf u})}\cdot\frac{1}{[\rw(2,0)]({\bf u})}\cdot\frac{1}{[\rw(0,2)]({\bf u})} \ ,
\end{equation}
where ${\bf u}=(u_1,u_2)$. 
On the other hand the base $\HH P^1$ has $2$ fixed points and the weights for the $T^2$-action related to the structure induced by $J$  on the fiber $\C P^1\times \C P^1$ at these two  fixed points are $2x_1$ and $2x_2$. It implies that
\begin{equation}\label{restriction}
\Phi _{\HH P^1}(\pi)|_{Fix(\HH P^1)} = 2\cdot (\frac{1}{[(2,0)]({\bf u})}+\frac{1}{[(-2,0)]({\bf u})})\cdot (\frac{1}{[(0,2)]({\bf u})}+\frac{1}{[(0,-2)]({\bf u})}) \ .
\end{equation}
As the generator $z$ from~\eqref{fiber_quater} vanishes on the base $\HH P^{1}$ it follows that the coefficients $g_{0,\omega}$  from~\eqref{fiber_quater} for these fibrations can be computed from~\eqref{restriction}. Namely if we apply the Chern-Dold character to~\eqref{restriction} we obtain:
\[
ch_{U}\Phi _{\HH P^1}(\pi)|_{Fix(\HH P^1)} = 2\cdot (\frac{f(2x_1)}{2x_1}\cdot \frac{f(2x_2)}{2x_2} + \frac{f(2x_1)}{2x_1}\cdot \frac{f(-2x_2)}{-2x_2}+\]
\[
\frac{f(-2x_1)}{-2x_1}\cdot \frac{f(2x_2)}{2x_2} + \frac{f(-2x_1)}{-2x_1}\cdot \frac{f(-2x_2)}{-2x_2})=\sum _{k,l}2^{2(k+l+1)}x_1^{2l}x_2^{2k}a_{2l+1}a_{2k+1},
\]
what, together with the application  of the Chern-Dold character to the restriction of~\eqref{fiber_quater} to the base $\HH P^{1}$, gives that
\[
g_{0,\omega}(\pi) = 2^{2(i_1+i_2+1)}a_{2i_1+1}a_{2i_2+1},\;\; \text{for}\;\; \omega = (i_1,i_2).
\]
\end{proof}
\begin{rem}
Note that $\Phi _{\HH P^1}(\pi)|_{Fix(\HH P^1)}$ is invariant under the action of $W_{Sp(2)}$.
On the other hand we have that the complementary root system $x_1-x_2, x_1+x_2$ to the fiber is not invariant under the action of the stationary Weyl group $W_{Sp(1)}\times W_{Sp(1)}$ of the base $\HH P^1$ which is given by $\pm x_1, \pm x_2$. By Proposition~\ref{SF} it implies that this structure, and analogously any other invariant almost complex structure  does not come from some invariant almost complex structure on the base. It further provides one more proof that $\HH P^{1}$ does not admit any invariant almost complex structure. 
\end{rem}

\begin{rem}
We can proceed in  the same way to compute $\Phi _{\HH P^{n-1}}(\hat{\pi})$ for the fibration $\C P^{1}\to \dr{\C P^{2n-1}}{\hat{\pi}} \HH P^{n-1}$, where we consider unique up to conjugation  $Sp(n)$-invariant complex structure $J$  on $\C P^{2n-1}$ and the action of the torus $T^n$  on $\C P^{2n-1}$ and $\HH P^{n-1}$,  since 
$\C P^{2n-1} = Sp(n)/U(1)\times Sp(n-1)$. 
For $n=2$, the roots for $J$ are $x_1-x_2, x_1+x_2, 2x_1$, while  $2x_1$ is the root of the induced invariant almost complex structure on the fiber.
Therefore we have
\begin{equation}\label{restriction1}
\Phi _{\HH P^1}({\hat \pi})|_{Fix (\HH P^{1})} = 2\cdot (\frac{1}{[(0,2)]({\bf u})} + \frac{1}{[(0,-2)]({\bf u})}).
\end{equation}
For the coefficients $g_{0,\omega}$  from~\eqref{fiber_quater} for this fibration,  by application of the Chern-Dold character to~\eqref{restriction1}, as in the previous proof,  we obtain:
\[
ch_{U}\Phi _{\HH P^1}({\hat \pi})|_{Fix (\HH P^{1})} = 2\cdot (\frac{f(2x_2)}{2x_2}+\frac{f(-2x_2)}{-2x_2})=\sum _{k=0}^{\infty}2^{2k+2}x_2^{2k}a_{2k+1} .
\]
\[
g_{0,\omega }({\hat \pi}) = 0,\;\text{for}\; \omega = (i_1,i_2),\; i_1\neq 0,\; g_{0,\omega}({\hat \pi})=2^{2i_2}a_{2i_2+1}\; \text{for}\; \omega = (0,i_2).
\] 
\end{rem}
\subsection{On the multiplicativity question of the universal toric genus.} 
\numberwithin{thm}{subsection}
The multiplicativity question for the universal toric genus of a tangentially stable complex $T^k$ manifold $M^{2n}$, and in particular homogeneous spaces we consider, is closely related to the algebraic question of the type of decomposability of $\Phi (M^{2n}, c_{\tau},\theta )$ in $\Omega _{U}^{*}[[u_1,\ldots ,u_k]]$.
We summarize it for homogeneous spaces $G/K$ as follows:
\begin{enumerate}
\item If there is fibration $H/K\to G/K\to G/H$ such that the invariant almost complex structure on $G/K$ induces invariant almost complex structures on both  $H/K$ and $G/H$ then by Theorem~\ref{genus_fibration} the element $\Phi (G/K)$  can be represented as the sum of the elements  from the orbit of  $\Phi (G/H)_{e}\cdot \Phi (H/K)$ by the action of the quotient of groups $W_{G}/W_{H}$  in $\Omega _{U}^{*}[[u_1,\ldots, u_k]]$, where $\Phi (G/H)_{e}$ denotes the universal toric genus of $G/H$ restricted to the identity point.   
\item If $G/K$ is the total space of a homogeneous fibration regarding to which the universal toric genus for $G/K$ is multiplicative it implies that the element $\Phi (G/K)$ decomposes into the product of two elements in $\Omega _{U}^{*}[[u_1,\ldots ,u_k]]$. Examples of such decomposable  elements in $\Omega _{U}^{*}[[u_1,\ldots ,u_k]]$ are given by the universal toric genera of the spaces provided by Corollary~\ref{semisimple_mult}. 
\item If for some homogeneous fibration the invariant almost complex structure on $G/K$ does not induce an invariant almost complex structure on the base $G/H$ then the universal toric genus of $G/K$ can not be obtained as the sum of the elements of the orbit of the multiple of $\Phi (H/K)$ by the action of $W_{G}/W_{H}$, where  $H/K$ is endowed  with the  induced invariant almost complex structure.
\end{enumerate}

\begin{ex}
Let us consider the universal toric genus for $(Sp(2)/T^2, J)$ given by~\eqref{SP}. We know that $Sp(2)/T^2$ fibers over $\HH P^1$ with the fiber $\C P^1\times \C P^1$, but it checks directly that~\eqref{SP} can not be obtained  as the sum  of the elements from the  orbit of $\Phi (\C P^1)\cdot \Phi (\C P^1) = (\frac{1}{[(2,0)]({\bf u})}+\frac{1}{[(-2,0)]({\bf u})})\cdot (\frac{1}{[(0,2)]({\bf u})}+\frac{1}{[(0,-2)]({\bf u})})$ multiplied by $\frac{1}{[(1,1)]({\bf u})}\cdot \frac{1}{[(1,-1)]({\bf u})}$, by the action of the group $W_{Sp(2)}/W_{Sp(1)}\times W_{Sp(1)}=\Z _{2}$.    

On the other hand since $W_{Sp(2)} = S_{2}(\pm x_1,\pm x_2)$ we see that $\Phi (Sp(2)/T^2)$ can be written as 
\[
\sum _{\rw \in S_{2}(\pm x_1, x_2)}\frac{1}{[\rw(1,1)]({\bf u})}\cdot\frac{1}{[\rw(1,-1)]({\bf u})}\cdot\frac{1}{[\rw(2,0)]({\bf u})}\cdot \rw (\frac{1}{[(0,2)]({\bf u})}+\frac{1}{[(0,-2)]({\bf u})}) =
\]
\begin{equation}
\sum _{\rw \in S_{2}(\pm x_1, x_2)}\frac{1}{[\rw(1,1)]({\bf u})}\cdot\frac{1}{[\rw(1,-1)]({\bf u})}\cdot\frac{1}{[\rw(2,0)]({\bf u})}\cdot \rw (\Phi (\C P^1, J_1)) ,
\end{equation}
where $J_1$ is the standard complex structure on $\C P^1$.  It checks directly that this twisted product formula  for $\Phi (Sp(2)/T^2, J)$ can be realized by the fibration $\C P^1 \to Sp(2)/T^2 \to \C P^3$, where the base and the fiber are endowed with the standard complex structures.
\end{ex}

\bibliographystyle{amsplain}

\begin{thebibliography}{10}
\bibitem{Adams} 
J.~F.~Adams, {\sl Lectures on Lie groups.} W.~A.~Benjamin, Inc., New York-Amsterdam 1969.


\bibitem{AH}
M.~Atiyah and F.~Hirzebruch, {\em Spin manifolds and group actions}, Essays on Topology and Related Topics,
Sprienger, 1970, 18--28.

\bibitem{B}
A.~Borel, {\em Sur la cohomologie des espaces fibr\'es principaux et des espaces homog\`enes de groupes de {L}ie compacts}, Ann.~of Math.~{\bf 57} (1953), 115--207.

\bibitem{BH}
A.~Borel and F.~Hirzebruch, {\em Characteristic classes and
homogeneous spaces I}, Amer.~J.~Math.~{\bf 80} (1958),
459--538.

\bibitem{BHII}
A.~Borel and F.~Hirzebruch, {\em Characteristic classes and
homogeneous spaces II}, Amer.~J.~Math.~{\bf 81} (1959), 315--382.


\bibitem{BT}
R.~Bott and C.~H.~Taubes, {\em On the rigidity theorem of Witten}, J.~Amer.~Math.~Soc.~{\bf 2} (1989), 137--186. 

\bibitem{BF}
H.~W.~Braden and K.~E.~Feldman, {\em Functional equations and the generalised elliptic genus}, J.~Nonlinear Math.~ Phys.~{\bf 12}, Supplement 1 (2005), 74-85.

\bibitem{BCH}
V.~M.~Buchstaber, {\em The Chern-Dold character in cobordisms I}, (Russian) Mat.~Sb.~(N.~S.) {\bf 83 (125)} (1970), 575--595. 

\bibitem{B-10}
V.~M.~Buchstaber, {\em General Krichever genus}, (Russian) Uspekhi Mat.~Nauk, 65:6 (395) (2010),
187--188.

\bibitem{B-Bun-10}
V.~M.~Buchstaber, E.~Yu.~Bunkova, {\em Elliptic formal group
laws, integral Hirzebruch genera and Krichever genera}, arXiv:
1010.0944v1 [math-ph] 5 Oct 2010.


\bibitem{BN}
V.~M.~Buchstaber and S.~P.~Novikov, {\sl Formal groups, power systems and Adams operations}, (Russian) Mat.~Sbornik {\bf 84 (1)}, (1971), 81--118. (english translation in Mathematics of the USSR - Sbornik {\bf (13) 1},
(1971), 80--116.
  
\bibitem{BP}
V.~M.~Buchstaber and T.~E.~Panov, {\sl Torus action and their application in topology and combinatorics}, University Lecture Series, 24, American Mathematical Society, Providence, RI, 2002. 

\bibitem{BPR}
V.~ M.~Buchstaber, T.~E.~Panov and N.~Ray {\sl Toric genera}, Internat.~ Math.~Research Notices {\bf 16} (2010), 3207--3262.

\bibitem{BR}
V.~M.~Buchstaber and N.~Ray, {\em The universal equivariant genus
and Krichever's formula}, (Russian) Uspekhi Mat.~Nauk {\bf 62 (1)}
(2007), 195--196. (english translation in Russian Math.~Surveys {\bf
62 (1)} (2007))

\bibitem{Buch_Terz} 
V.~M.~Buchstaber and S.~Terzi\'c, {\em Equivariant complex structures on homogeneous spaces and their cobordism classes}, American Mathematical Society Translations, Series 2, Volume 224, Advances in the Mathematical Sciences, 2008, 27--57.

\bibitem{CHS}
S.~S.~Chern, F.~Hirzebruch and J-P.~Serre, {\em On the index of fibered manifold}, Proc.~Amer.~Math.~Soc.~{\bf 8} (1957), 587--596.

\bibitem{G}
A.~Gray, {\em Riemannian manifolds with geodesic symmetries of order $3$}, J.~Differential Geom.~{\bf 7} (1972),
343--369.

\bibitem{H}
F.~Hirzebruch, {\em \"Uber die quaternionalen projektiven R\"aume}, S.-B.~Math.~Nat.~Kl.~Bayer.~Akad.~Wiss.~1953, (1954), 301--312. 

\bibitem{HTM}
F.~Hirzebruch, {\sl Topological methods in algebraic geometry}, Springer-Verlag New York, Inc., New York 1966. 

\bibitem{HLN}
F.~Hirzebruch, {\em Elliptic genera of level N for complex manifolds}, in Proceedings "Differential geometrical methods in theoretical physics" edited by K.~Bleuler and M.~Werner, Kluwer Academic Publishers (1988), 37--65.  

\bibitem{HBJ}
F.~Hirzebruch, T.~Berger and R.~Jung, {\sl Manifolds and modular forms}, Aspects of Mathematics, E20, 
Vieweg and Braunschweig, 1992.

\bibitem{HS}
F.~Hirzebruch and P.~Sladowy, {\em Elliptic genera, involutions, and homogeneous spin manifolds}, Geom.~Dedicata {\bf 35} (1990), no.~1-3, 309--343. 

\bibitem{GG}
G.~Grantcharov, {\em Geometry of compact complex homogeneous spaces with vanishing first Chern class}, Adv.~in Math.~{\bf 226} (2011), no.~4, 3136--3159.

\bibitem{K}
T.~Koda, {\em The first Chern class of Riemannian $3$-symmetric spaces: the classical case}, Note Mat.~{\bf 10} (1990),
no.~1, 141--156.

\bibitem{KT}
D.~Kotschick and S.~Terzi\'c, {\em Chern numbers and the geometry of partial flag manifolds},
Comment.~Math.~Helv.~{\bf 84} (2009), 587--616.

\bibitem{Krichever-74}
I.~M.~Krichever, {\em Formal groups and the Atiyah--Hirzebruch
formula}, (Russian) Izv.~Akad.~Nauk SSSR, {\bf 38} (1974), no.~6, 1271--1285.

\bibitem{Krichever-76}
I.~M.~Krichever, {\em Obstructions to the existence of
$S^1$-actions. Bordism of remified covering}, (Russian) Izv.~Akad.~Nauk SSSR,
{\bf 10} (1976), no.~4,  783--797.

\bibitem{KR}
I.~M.~Krichever, {\em  Generalized elliptic genera and Baker-Akhiezer functions} (Russian) Math.~Zametki {\bf 47} (1990), no.~2, 34--45; translation in Math.~Notes {\bf 47} (1990), no.~2, 132--142.


\bibitem{MC}
I.~G.~Macdonald, {\sl Symmetric functions and Hall polinomials}, 2nd edition, Oxford University Press, 1995.

\bibitem{M}
W.~S.~Massey, {\em Non-existence of almost complex structures on quaternionic projective space}, Pacific J.~Math.~{\bf 12} (1962),  no.~1, 1379--1384. 

\bibitem{NM}
S.~P.~Novikov, {\em Homotopy properties of Thom complexes}, (Russian) Mat.~Sb.~(N.~S.) {\bf 57 (99)} (1962), 406--442.

\bibitem{Novikov-67}
S.~P.~Novikov, {\em Methods of algebraic topology from point of view
of cobordism theory}, (Russian) Math. USSR Izv. {\bf 31} (1967), no.~4, 827--913.

\bibitem{Novikov-68}
S.~P.~Novikov, {\em Adams operators and fixed points}, (Russian) Izv.~Akad.~Nauk
SSSR, {\bf 32} (1968), no.~6, 1193--1211.



\bibitem{OCH}
S.~Ochanine, {\em Sur les genres multiplicatifs définis par des intégrales elliptiques}, (French) Topology  {\bf 26}  (1987),  no.~2, 143–151. 

\bibitem{On}
A.~L.~Onishchik, {\sl Topologiya tranzitivnykh grupp preobrazovanii}, (Russian) Fizmatlit ``Nauka'', Moscow, 1995. 

\bibitem{VO}
A.~L.~Onishchik and E.~B.~Vinberg, {\sl Seminar po gruppam Li i algebraicheskim gruppam}, (Russian) URSS, Moscow, 1995.
\bibitem{P}
T.~E.~Panov, {\em Hirzebruch genera of manifolds with torus action} (Russian) Izv.~Ross.~Akad.~Nauk Ser.~Mat.~{\bf 65}
(2001), no.~3, 123--138; translation in Izv.~Math.~{\bf 65} (2001), no.~3, 543--556.

\bibitem{Rudin}
W.~Rudin, {\sl Real and complex analysis}, McGraw-Hill Book, 1987. 

\bibitem{SH}
P.~Skanahan, {\em On the signature of Grassmannians}, Pacific J.~Math.~{\bf 84} (1979), no.~2, 483--490.

\bibitem{SL}
P.~Slodowy, {\em On the signature of homogeneous spaces}, Geom.~Dedicata {\bf 43} (1992), no.~1, 109--120.

\bibitem{TF}
S.~Terzi\'c, {\em Cohomology with real coefficients of generalized
symmetric spaces} (Russian), Fundam.~Prikl.~Mat.~Vol. {\bf 7} 
(2001), no.~1  131--157.

\bibitem{Q}
D.~Quillen, {\em On the formal groups laws of unoriented
and complex cobordism theory}, Bull.~Amer.~Math.~Soc.~{\bf 75}
(1969), 1293--1298.

\bibitem{Wang}
H.-C.~Wang, {\em Closed manifolds with homogeneous complex structure}, Amer.~J.~Math.~{\bf 76} (1954), no.~1, 1--32. 

\bibitem{WT}
Y.~Watanabe and K.~Takamatsu, {\em On a $K$-space of constant holomorphic sectional curvature}, Kodai Math.~Sem.~Rep.~{\bf 25} (1973), 297--306. 

\bibitem{WG}
J.~A.~Wolf and A.~Gray, {\em Homogeneous spaces defined by Lie group automorphisms, I $\&$ II},
J.~Differential Geometry {\bf 2} (1968), 77--114 and 115--159.
\end{thebibliography}

Victor M.~Buchstaber\\
Steklov Mathematical Institute, Russian Academy of Sciences\\ 
Gubkina Street 8, 119991 Moscow, Russia\\
E-mail: buchstab@mi.ras.ru\\ 
School of Mathematics, University of Manchester\\
Oxford Road, Manchester M13~9PL, England\\
E-mail: Victor.Buchstaber@manchester.ac.uk
\\ \\ 

Svjetlana Terzi\'c \\
Faculty of Science, University of Montenegro\\
D\v zord\v za Va\v singtona bb, 81000 Podgorica, Montenegro\\
Email: sterzic@ac.me

\end{document}